\documentclass[12pt]{amsart}
\usepackage[foot]{amsaddr}

\usepackage{amssymb,amscd,amsthm,verbatim,amsmath,color,fancyhdr,mathrsfs}
\usepackage{graphicx}
\usepackage{turnstile}

\usepackage[letterpaper, left=2.5cm, right=2.5cm, top=2.5cm,bottom=2.5cm,dvips]{geometry}

\usepackage{styfile} 

\newtheorem{theorem}{Theorem}[section]
\newtheorem{proposition}[theorem]{Proposition}
\newtheorem{lemma}[theorem]{Lemma}
\newtheorem{remark}[theorem]{Remark}
\newtheorem{corollary}[theorem]{Corollary}

\newtheorem{definition}[theorem]{Definition}

\newtheorem{assumption}{Assumption}
\newtheorem*{takeaway}{Takeaway}

\DeclarePairedDelimiterX{\bracket}[3]{#1}{#2}{#3}
\newcommand{\round}[1]{\bracket*{(}{)}{#1}}

\newcommand{\squarebrack}[1]{\bracket*{\lbrack}{\rbrack}{#1}}

\providecommand*{\eu}{\ensuremath{\mathrm{e}}} 

\providecommand{\newoperator}[3]{\newcommand*{#1}{\mathop{#2}#3}}
\providecommand{\renewoperator}[3]{\renewcommand*{#1}{\mathop{#2}#3}}

\renewoperator{\Re}{\mathrm{Re}}{\nolimits}
\renewoperator{\Im}{\mathrm{Im}}{\nolimits}

\providecommand*{\slot}[1]{\ifblank{#1}{\,\cdot\,}{#1}}

\DeclarePairedDelimiterXPP{\nrm}[2]{}{\lVert}{\rVert}{\ensuremath{_{#1}}}{\ifblank{#2}{\:\cdot\:}{#2}}
\newcommand{\norm}[2]{\nrm*{#1}{#2}}

\newcommand{\enorm}[1]{\norm{2}{#1}} 

\newcommand{\opnorm}[1]{\norm{\mathrm{op}}{#1}}
\newcommand{\normF}[1]{\norm{\mathrm{F}}{#1}}

\newcommand{\abs}[1]{\bracket*{\lvert}{\rvert}{#1}}

\newcommand{\inner}[1]{\bracket*{\langle}{\rangle}{#1}}

\DeclarePairedDelimiterXPP\prob[1]{\mathbb{P}}{\lbrace}{\rbrace}{}{\renewcommand\given{\nonscript\:\delimsize\vert\nonscript\:\mathopen{}}#1} 
\newcommand{\Prob}[1]{\prob*{#1}} 

\DeclarePairedDelimiterXPP\probability[2]{\mathbb{P}_{#1}}{\lbrace}{\rbrace}{}{\renewcommand\given{\nonscript\:\delimsize\vert\nonscript\:\mathopen{}}#2} 

\DeclarePairedDelimiterXPP\expectation[1]{\mathbb{E}}{\lbrack}{\rbrack}{}{\renewcommand\given{\nonscript\:\delimsize\vert\nonscript\:\mathopen{}}#1} 
\newcommand{\E}[1]{\expectation*{#1}} 

\DeclarePairedDelimiterXPP\expectationdist[2]{\mathbb{E}_{#1}}{\lbrack}{\rbrack}{}{\renewcommand\given{\nonscript\:\delimsize\vert\nonscript\:\mathopen{}}#2} 
\newcommand{\Exp}[2]{\expectationdist*{#1}{#2}} 

\DeclarePairedDelimiterXPP\variance[1]{\mathrm{Var}}{\lbrack}{\rbrack}{}{\renewcommand\given{\nonscript\:\delimsize\vert\nonscript\:\mathopen{}}#1} 
\newcommand{\Var}[1]{\variance*{#1}} 

\DeclarePairedDelimiterXPP\variancedist[2]{\mathrm{Var}_{#1}}{\lbrack}{\rbrack}{}{\renewcommand\given{\nonscript\:\delimsize\vert\nonscript\:\mathopen{}}#2} 

\DeclarePairedDelimiterXPP\covariance[2]{\mathrm{Cov}}{(}{)}{}{#1,\mathopen{}#2} 

\newcommand{\inv}[1]{\frac{1}{#1}}

\newcommand{\divergence}{\nabla \cdot}

\newoperator{\supp}{\mathrm{supp}}{\nolimits}

\makeatletter
\providecommand*{\diff}
{\@ifnextchar^{\DIfF}{\DIfF^{}}}
\def\DIfF^#1{
	\mathop{\mathrm{\mathstrut d}}
	\nolimits^{#1}\gobblespace}
\def\gobblespace{
	\futurelet\diffarg\opspace}
\def\opspace{
	\let\DiffSpace\!
	\ifx\diffarg(
	\let\DiffSpace\relax
	\else
	\ifx\diffarg[
	\let\DiffSpace\relax
	\else
	\ifx\diffarg\{
	\let\DiffSpace\relax
	\fi\fi\fi\DiffSpace}

\providecommand*{\pdiff}
{\@ifnextchar^{\pDIfF}{\pDIfF^{}}}
\def\pDIfF^#1{
	\mathop{\mathrm{\mathstrut \partial}}
	\nolimits^{#1}\gobblespace}
\def\gobblespace{
	\futurelet\diffarg\opspace}
\def\opspace{
	\let\DiffSpace\!
	\ifx\diffarg(
	\let\DiffSpace\relax
	\else
	\ifx\diffarg[
	\let\DiffSpace\relax
	\else
	\ifx\diffarg\{
	\let\DiffSpace\relax
	\fi\fi\fi\DiffSpace}

\providecommand*{\deriv}[3][]{\frac{\diff^{#1}{#2}}{\diff {#3}^{#1}}}
\providecommand*{\pderiv}[3][]{\frac{\pdiff^{#1}{#2}}{\pdiff {#3}^{#1}}}

\makeatother

\providecommand\given{}
\newcommand\SetSymbol[1][]{
	\nonscript\:#1\vert
	\allowbreak
	\nonscript\:
	\mathopen{}}
\DeclarePairedDelimiterX\Set[1]\{\}{
	\renewcommand\given{\SetSymbol[\delimsize]}
	#1
}

\newcommand{\interval}[1]{\squarebrack{#1}}

\newcommand{\setinline}[1]{\Set{#1}}

\newcommand{\indicator}[2]{\mathbbm{1}\ensuremath{_{#1}}\ifblank{#2}{}{\Set*{#2}}} %

\newcommand{\Rd}[1]{\mathbb{R}^{#1}}
\newcommand{\Natural}{\mathbb{N}}

\newcommand{\Rational}{\mathbb{Q}}

\newcommand{\id}{\mathrm{id}}

\newcommand{\Bcal}{\mathcal{B}}

\newcommand{\Ecal}{\mathcal{E}}
\newcommand{\Fcal}{\mathcal{F}}

\newcommand{\Ical}{\mathcal{I}}

\newcommand{\Mcal}{\mathcal{M}}

\newcommand{\Pcal}{\mathcal{P}}

\newcommand{\Tcal}{\mathcal{T}}

\newcommand{\Wcal}{\mathcal{W}}

\newcommand{\tauvec}{{\boldsymbol{\tau}}}

\newcommand{\rr}{\mathbb{R}}
\newcommand{\eq}{\begin{equation}}
\newcommand{\en}{\end{equation}}

\providecommand*{\Graphons}{{\widehat\Wcal}}
\providecommand*{\cut}{\square}
\newcommand{\cutnorm}[1]{\norm{\cut}{#1}}

\newcommand{\cutmetric}[2]{\delta_\cut\round{\slot{#1},\slot{#2}}}
\newcommand{\metric}[3]{{#1}\round{\slot{#2},\slot{#3}}}
\newcommand{\toplim}[1]{#1\text{-}\lim} 
\newcommand{\effdom}{\mathop{\mathrm{eff}\text{-}\mathrm{Dom}}\nolimits}
\newcommand{\Ent}{\mathcal{E}}
\newcommand{\RaghavS}[1]{{\color{orange}[Somani: #1]}}

\title{Gradient flows on graphons: existence, convergence, continuity equations}
\author{Sewoong Oh}
\address{Sewoong Oh\\ Paul G. Allen School of Computer Science \& Engineering \\ University of Washington\\ Seattle WA 98195, USA\\ {Email: sewoong@cs.washington.edu}}
\author{Soumik Pal}
\address{Soumik Pal\\ Department of Mathematics \\ University of Washington\\ Seattle WA 98195, USA\\ {Email: soumikpal@gmail.com}}
\author{Raghav Somani}
\address{Raghav Somani\\ Paul G. Allen School of Computer Science \& Engineering \\ University of Washington\\ Seattle WA 98195, USA\\ {Email: raghavs@cs.washington.edu}}
\author{Raghavendra Tripathi}
\address{Raghavendra Tripathi\\ Department of Mathematics \\ University of Washington\\ Seattle WA 98195, USA\\ {Email: raghavt@uw.edu}}

	

\thanks{This research is partially supported by the following grants. Pal is supported by NSF grant DMS-2052239 and a PIMS CRG (PIHOT). Pal and Oh are supported by NSF grant DMS-2134012. Oh is supported by NSF grant CCF-2019844. Many thanks to Persi Diaconis, Apoorva Khare and Stefan Steinerberger for helpful conversations and references and the PIMS Kantorovich Initiative for facilitating this collaboration. The authors are listed in alphabetical order.}

\date{\today}

\begin{document}

\begin{abstract}
Wasserstein gradient flows on probability measures have found a host of applications in various  optimization problems. They typically arise as the continuum limit of exchangeable particle systems evolving by some mean-field interaction involving a gradient-type potential. However, in many problems, such as in multi-layer neural networks, the so-called particles are edge weights on large graphs whose nodes are exchangeable. Such large graphs are known to converge to continuum limits called graphons as their size grow to infinity. We show that the Euclidean gradient flow of a suitable function of the edge-weights converges to a novel continuum limit given by a curve on the space of graphons that can be appropriately described as a gradient flow or, more technically, a curve of maximal slope. Several natural functions on graphons, such as homomorphism functions and the scalar entropy, are covered by our set-up, and the examples have been worked out in detail.
\end{abstract}

\maketitle



\section{Introduction}\label{sec:intro}

Let $x_1, x_2, \ldots, x_n$ be $n$ vectors in $\rr^d$. Let $f_n\colon \left( \rr^d\right)^n \rightarrow \rr$ be a permutation invariant function of those $n$ variables. Here \textit{permutation invariant} means $f_n(x_1, \ldots, x_n)= f_n(x_{\pi_1}, \ldots, x_{\pi_n})$ where $\pi$ is any permutation of the set $\squarebrack{n}\coloneqq \{1,2,\ldots, n\}$. Throughout this text, we will denote the symmetric group on the finite set $\squarebrack{n}$ as $S_n$. Consider the Cauchy problem 
\begin{equation}\label{eq:finitecauchy}
    \dot{x}_i(t) = - \nabla_{i} f_n(x_1(t), \ldots, x_n(t)), \quad i\in [n], \qquad t\in\rr_+,
\end{equation}
with given initial conditions $\round{x_i(0)}_{i\in [n]}$. Here $\nabla_i$ refers to the partial derivative with respect to the $i$-th variable and $\rr_+$ denotes the set of all non-negative real numbers. The solution to this problem (which exists and is unique when, say, $\nabla f_n$ is Lipschitz~\cite{lindelof1894application}) is often called the gradient flow of $f_n$. A natural question that appears in several applications is whether the solution to the above Cauchy problem has a \textit{scaling limit} as $n$ goes to infinity.  

In order for such a limit to exist, it is imperative that there is some consistency over the dimension parameter $n$. The permutation invariance of $f_n$ offers a resolution. In fact, if $\round{x_i}_{i\in [n]}$ is thought of as positions of particles in space, $f_n$ can be thought of as a function acting on the empirical distribution $\mu_n\coloneqq \frac{1}{n} \sum_{i=1}^n \delta_{x_i}$. Here $\mu_n$ is a discrete probability measure that puts mass $1/n$ on the positions of each of the $n$ particles, represented by the delta mass $\delta_{\cdot}$.

For any metric space $(\Omega,d)$, denote its Borel sigma algebra as $\Bcal\round{\Omega}$ and the set of all Borel probability measures as $\Pcal(\Omega)$. Consider $\rr^d$ with the usual Euclidean metric and let $F\colon \Pcal(\rr^d) \rightarrow \rr$ be a suitable function. The function $F$ induces a sequence of permutation invariant functions $\round{f_n}_{n\in\Natural}$ as above by the definition 
\[
    f_n\left( x_1, \ldots, x_n \right) \coloneqq F\left( \mu_n \right), \quad n \in \Natural.
\]
For such an $f_n$ for any $n\in\Natural$, the evolution~\eqref{eq:finitecauchy} can be thought of as an evolution on the space of probability measures by defining
\[
    \mu_n(t) \coloneqq \frac{1}{n} \sum_{i=1}^n \delta_{x_i(t)}, \quad t\in\rr_+.
\]
Now the following question makes sense. Suppose that the sequence of initial measures $\round{\mu_n(0)}_{n\in\Natural}$ \textit{converges} to a limiting probability measure $\mu(0)$ where the convergence is typically in the sense of weak convergence of probability measures. Does the sequence of curves $\big(\left( \mu_n(t) \right)_{t\in\rr_+}\big)_{n\in\Natural}$ converge to some limiting curve on $\Pcal(\rr^d)$ possibly after rescaling time? 

The answer to the above, under suitable assumptions on $F$, is the so-called Wasserstein gradient flow~\cite{V03, santambrogio2015optimal} of $F$ on the metric space $\Pcal\left(\rr^d \right)$ equipped with the Wasserstein-$2$ metric, $\mathbb{W}_2$. There is now a general theory of curves of maximal slopes (AKA gradient flows) developed for functions on metric spaces which may lack a differentiable structure~\cite{ambrosio2005gradient}. The Wasserstein space is a prominent example that has been thoroughly studied~\cite{ambrosio2005gradient,santambrogio2017euclidean}. Recently there has been a surge in interest in the application of the above convergence of gradient flows in the context of single hidden layer neural networks, see~\cite{song2018mean,chizat2018global,rotskoff2018parameters,mei2019mean,carrillo2019blob,araujo2019mean,nguyen2020rigorous,sirignano2020clt,sirignano2020lln,tzen2020mean,bach2021gradient}.

However, we are interested in optimization problems where the arguments can be thought of as weights attached to the edges of a large dense graph. Let $G=([n], E)$ be a graph. For $\{i,j\} \in E$ one has an associated variable $W_{i,j}=W_{j,i}$ that we take it to be real-valued in this article. For all the applications we consider, we can take $W_{i,j}=0$ if $\{i,j\}\not\in E$. Thus our variables can be arranged in an $n \times n$ symmetric matrix $\left( W_{i,j} \right)_{i,j\in [n]}$. Let the set of $n\times n$ real valued symmetric matrices be denoted by $\Mcal_n(\rr)$. Let $f_n\colon \Mcal_n(\rr) \rightarrow \rr \cup \{\infty\}$ be a function of such matrices. The crucial difference from the previous set-up is that we want $f_n$ to satisfy a permutation invariance property with respect to relabeling the vertices of $G$: for every $\pi \in S_n$, 
\[
    f_n\left( \round{W_{\pi_i, \pi_j}}_{i,j \in [n]}\right) = f_n\left( \round{W_{i,j}}_{i,j \in [n]}\right).
\]
That is, the function value does not change if we permute the rows and the columns of the symmetric matrix $\round{W_{i,j}}_{i,j \in [n]}$ by the same permutation. In other words, such functions are invariant under graph isomorphisms of $G$. We call such functions {\em invariant}. One can now ask the same question as before. Consider the gradient flow Cauchy problem
\begin{equation}\label{eq:finitecauchymatrix}
    \dot{W}_{i,j}(t) = - \nabla_{i,j} f_n\left( \round{W_{i,j}(t)}_{i,j \in [n]}\right), \quad i,j\in\squarebrack{n}, 
\end{equation}
with given $\left(W_{i,j}(0)\right)_{i,j \in [n]}$. Is there a suitable scaling limit as $n$ goes to infinity? This paper answers this question in affirmative under reasonable conditions on $f_n$.  

We restrict ourselves to the case where the edge weights $\round{W_{i,j}}_{i,j \in [n]}$ all lie in the bounded interval $[-1,1]$. Without loss of generality, we can take our graph to be the complete graph with its weighted adjacency matrix $\round{W_{i,j}}_{i,j \in [n]}$. Just like empirical distributions of particle systems converge to probability measures, these graph adjacency matrices with bounded edge weights, identified up to graph isomorphisms, converge to a limiting object called graphons~\cite{lovasz2006limits,borgs2008convergent,borgs2012convergent}. This is intimately connected with the theory of exchangeable arrays in probability theory~\cite{aldous1981representations, aldous1982exchangeability, hoover1982row, kallenberg1989representation}. For a definitive modern account of exchangeable arrays and their connections with the limits of large graphs, the reader is referred to~\cite{diaconis2007graph,austin2008exchangeable, austin2012exchangeable, AustinExchMeasure}. The theory of graph limits and graphons has found applications in many fields including extremal graph theory, combinatorics, data analysis, biology. We refer the reader to~\cite{DIAO2015183,DiffCal,bhattacharya2020upper,ben2021ordered} and references therein for further details. 

\subsection{Setup and Main results}\label{subsec:Setup}
Let us define graphons rigorously. A \textit{kernel} is a Borel measurable function $W\colon \interval{0,1}^{(2)} \to \interval{-1,1}$ that is symmetric, i.e., $W(x,y) = W(y,x)$ for Lebesgue a.e. $(x,y)\in[0,1]^{(2)}$. Two kernels are identified if they are equal Lebesgue a.e. Unless specified otherwise, we always consider the domain of a kernel to be the standard Borel probability space on $[0,1]^{(2)}$. We will denote the set of all kernels by $\Wcal$. Let $U, V\in \Wcal$. We say $U\cong V$ if there exists Lebesgue measure preserving transforms $\varphi$ and $\psi$ on $\interval{0,1}$ and a kernel $W$ such that $U(x, y)=W(\varphi(x), \varphi(y))$ and $V(x, y)=W(\psi(x), \psi(y))$ for Lebesgue a.e. $(x,y)\in[0,1]^{(2)}$.
\begin{definition}[Graphons]\label{def:graphons}
    Equivalence class of functions in $\Wcal$ under ($\cong$) are called {\em graphons} and we denote the set of graphons by $\Graphons \coloneqq \Wcal/{\cong}$.
\end{definition}
For a kernel $W\in\Wcal$, we denote its equivalence class as $\squarebrack{W} \coloneqq \Set*{U\in\Wcal \given U\cong W}$.
See Section~\ref{sec:background} for more details. The set of graphons $\Graphons$ will be endowed with two different metrics. One, called the invariant $L^2$ metric~\cite{lovasz2012large, janson2016graphons, borgs2016sparse}, $\delta_2$, plays a similar role as the $\mathbb{W}_2$ metric does on probability measures. The metric space $(\Graphons, \delta_2)$ is a geodesic space (Proposition~\ref{prop:graphon_geodesic_space}). Our gradient flows will be with respect to this metric. The other, called the cut metric, $\delta_\cut$, is more traditional~\cite{borgs2008convergent, lovasz2012large} and is important here for the topology that it generates and will play the role of the metric of weak convergence of probability measures. Thus, although our gradient flows will be defined with respect to the invariant $L^2$ metric $\delta_2$, the various convergence statements will be with respect to the cut metric $\delta_\cut$.

To state our main results we need a set-up that is similar to particle systems and their limiting probability measures. For any $k\in\Natural$, we denote the set of all $k\times k$ symmetric matrices with entries in $[-1,1]$ as $\Mcal_k$. Every symmetric matrix in $\Mcal_k$ identified up to the same permutation on rows and columns can be embedded in the space of block graphons $\Graphons_k\subseteq \Graphons$ (see Section~\ref{sec:background} for details). Thus, any function $F\colon\Graphons\rightarrow \mathbb{R}\cup\Set*\infty$ induces a sequence of functions $\big(F_k\colon\Graphons_k\rightarrow \mathbb{R}\cup\Set*\infty\big)_{k\in\Natural}$, by restriction, and a sequence of invariant functions $\round{f_k}_{k\in\Natural}$ on such $k \times k$ symmetric matrices with entries in $[-1, 1]$. 

The first pertinent question is that given $F\colon\Graphons\to\mathbb{R}\cup\{\infty\}$ and $[U_0]\in\Graphons$ in the proper effective domain $\effdom(F) \coloneqq \setinline{\squarebrack{U}\in\Graphons\given F\round{\squarebrack{U}}<\infty}$ of $F$ (see~\cite[equation 1.2.1]{ambrosio2005gradient}), under what assumptions on $F$ a ``gradient flow'' of $F$ on $(\Graphons, \delta_2)$ exists starting at $[U_0]$. On a general metric space, a gradient flow curve (i.e., a curve of maximal slope~\cite[Definition 1.3.2]{ambrosio2005gradient}) is obtained by showing that the limits of the solutions of implicit Euler iterations (see Section~\ref{sec:grad_flows}) exist and satisfy added assumptions. The existence of the limit of the sequence of such implicit Euler iterations requires the {\em local slope} $\abs{\pdiff F}$ of $F$, defined as
\begin{align}
    \abs{\pdiff F}\round{\squarebrack{V}} &\coloneqq \limsup_{\squarebrack{W}\in\Graphons,\,\delta_2\round{\squarebrack{W},\squarebrack{V}} \to 0} \frac{\round{F\round{\squarebrack{V}} - F\round{\squarebrack{W}}}^+}{\delta_2\round{\squarebrack{W},\squarebrack{V}}}.
\end{align}
In practice, however, evaluating the local slope is not easy. We introduce the concept of \textit{Fr\'echet-like derivative} in Section~\ref{sec:frechet_derivatives} and show that for functions that have Fr\'echet-like derivative, the local slope $\abs{\pdiff F}$ admits a more amenable expression in terms of $L^2$-norm of the Fr\'echet-like derivative. Moreover, under a semiconvexity assumption, just the existence of Fr\'echet-like derivative suffices for the existence of a curve of maximal slope (see Theorem~\ref{thm:Existence_with_FrehetDerivative}).

\sloppy Since $(\Graphons,\delta_2)$ is a geodesic space, one can talk about $\lambda$-semiconvex functions for $\lambda\in\rr$ over geodesics and generalized geodesics (see Section~\ref{sec:background} and Section~\ref{sec:prelims}). Theorem~\ref{thm:GF_convergence} shows that under suitable assumptions on a semiconvex function $F$, the gradient flow of $F$ on $(\Graphons, \delta_2)$ can be obtained as the time-scaled limit of the Euclidean gradient flows of $\round{f_k}_{k\in\Natural}$ on $\round{\Mcal_k}_{k\in\Natural}$ respectively. That is, suppose we have a sequence of block graphons $(\squarebrack{U_{k,0}}\in \Graphons_k)_{k\in\Natural}\xrightarrow{\delta_{\cut}}[U_0]\in\Graphons$ and let $\omega = \round{\omega_t}_{t\in\rr_+}$ be the gradient flow of $F$ on $(\Graphons, \delta_2)$ starting at $\omega_0=[U_0]\in\Graphons$. Then the gradient flow $\omega^{(k)}=\big(\omega^{(k)}_t\big)_{t\in\rr_+}$ of $F_k$ starting at $\omega^{(k)}_0=[U_{k, 0}]\in\Graphons_k$ converges, in $\delta_\cut$ to $\omega$ uniformly over compact time intervals as $k\to\infty$. The next argument shows that for any $k\in\Natural$, $\omega^{(k)}$ is a time-scaling of the Euclidean gradient flow of $f_k$.

For any $k\in\Natural$, the Euclidean gradient flow of the function $f_k$ over $\Mcal_k$ can be approximated via the implicit Euler method. Starting from $X_{k,\tau}\in \Mcal_k$ with a step size of $\tau>0$, the next iterate of the implicit Euler method, say $X_{k,\tau,+}$, is obtained as
\begin{align}
    X_{k,\tau,+} &\in \argmin_{X_k\in \Mcal_k} \left[  f_k\round{X_k} + \inv{2\tau}\enorm{X_k - X_{k,\tau}}^2\right].\label{eq:matrix_implicit_Euler}
\end{align}
Let $K$ be the natural embedding map from $k\times k$ symmetric matrices to the space of block kernels $\Wcal_k$ (Definition~\ref{def:K_Mk}). Since the function $f_k$ is permutation invariant, setting $\squarebrack{U_{k,\tau}} = \squarebrack{K\round{X_{k,\tau}}}$, equation \eqref{eq:matrix_implicit_Euler} is equivalent to obtaining
\begin{align}
    \squarebrack{U_k} &\in \argmin_{\squarebrack{U_k}\in\Graphons_k}\left[ F_k\round{\squarebrack{U_k}} + \frac{k^2}{2\tau}\min_{\substack{X_k\in \Mcal_k,\\\squarebrack{U_k} = \squarebrack{K\round{X_k}}}}\inv{k^2}\sum_{i,j=1}^{k}\abs{\round{X_k}_{i,j} - \round{X_{k,\tau}}_{i,j}}^2\right]\nonumber\\
    &= \argmin_{\squarebrack{U_k}\in\Graphons_k} \left[ F_k\round{\squarebrack{U_k}} + \frac{k^2}{2\tau}\delta^2_2\round{\squarebrack{U_k},\squarebrack{U_{k,\tau}}}\right],\label{eq:graphon_implicit_Euler_scaled}
\end{align}
via the substitution $\squarebrack{U_{k}} = \squarebrack{K\round{X_k}}$. Equation~\eqref{eq:graphon_implicit_Euler_scaled} is precisely the implicit Euler iteration for gradient flow on $(\Graphons_k,\delta_2)$ with a step size of $\tau/k^2$ (see Section~\ref{sec:implicit_euler}). Since iterations of the form in equation~\eqref{eq:matrix_implicit_Euler} converge to the Euclidean gradient flow as $\tau\to 0$, its image via the map $X\mapsto \squarebrack{K\round{X}}$ in equation~\eqref{eq:graphon_implicit_Euler_scaled} converges to the gradient flow on $(\Graphons_k,\delta_2)$ as $\tau\to 0$.

\begin{theorem}[Convergence of Gradient Flows]\label{thm:GF_convergence}
    Suppose $F\colon\Graphons \to \rr\cup\Set*\infty$ satisfies the following conditions:
    \begin{enumerate}
        \item $F$ is continuous in $\delta_\cut$.
        \item\label{item:condition_2_convergence} $F$ is $\lambda$-semiconvex (Definition~\ref{def:lambda_cvx_along_curve}) along generalized geodesics on $(\Graphons,\delta_2)$ (Definition~\ref{def:gen_geodesics}), for some $\lambda\in\rr$.
    \end{enumerate}
    Consider the gradient flow $\omega^{(k)} = \big(\omega^{(k)}_t\big)_{t\in\rr_+}\subset \Graphons_k$ of $F$ on each $\Graphons_k$, starting at some $\omega^{(k)}_0 = \squarebrack{U_{k,0}}$ for $k\in\Natural$. Assume that the sequence $\round{\squarebrack{U_{k,0}}}_{k\in\Natural}\xrightarrow{\delta_\cut}\squarebrack{U_0}\in\Graphons$, and $\abs{\pdiff F}\round{\squarebrack{U_{0}}}<\infty$ and $\limsup_{k\to\infty}\abs{\pdiff F}\round{\squarebrack{U_{k,0}}} \leq G < \infty$, for some $G\geq 0$. Then,
    \begin{align}
        \limsup_{k\to\infty}\sup_{t\in[0,T]}\cutmetric{\omega^{(k)}_t}{\omega_t} &= 0,
    \end{align}
    for any $T\in\rr_+$, where $\omega = \round{\omega_t}_{t\in\rr_+}$ is the unique minimizing movement curve~\cite[Definition 2.0.6, page 42]{ambrosio2005gradient} on $\Graphons$ for the function $F$ starting at $\omega_0 = \squarebrack{U_0}$.
    In addition, if the conditions for the existence of curves of maximal slope (Theorem~\ref{thm:existence} or Theorem~\ref{thm:Existence_with_FrehetDerivative}) hold, then $\omega$ is a curve of maximal slope.
\end{theorem}

\begin{remark}\label{rem:gen_geo_cvx_from_W}
    \sloppy Condition~\ref{item:condition_2_convergence} in Theorem~\ref{thm:GF_convergence} is satisfied if the invariant extension $f\colon\Wcal\to\rr\cup\Set*\infty$ of $F$ is $\lambda$-semiconvex on $(\Wcal,d_2)$ (see Section~\ref{sec:topologies_on_graphons}, and Definition~\ref{def:semiconvexity}).
\end{remark}

An important theme in this work is that many important curves in $\Graphons$ are actually obtained as the projection of some curve defined in $\Wcal$. In particular, the gradient flow of $F$ is the natural image of an absolutely continuous curve in $(\Wcal, d_2)$. More precisely, if $f$ has a Fr\'echet-like derivative $D_{\Wcal}f(W)$ for all $W\in\Wcal$, then there gradient flow $\omega$ of $F$ can be written as the image of the curve $\round{[W_t]}_{t\in\rr_+}$ defined as
\[
    W_{t}(x, y)=W_0(x, y)-\int_{0}^{t}D_{\Wcal}f(W_s)(x,y)\indicator{G_{W_s}}{(x, y)}\diff s,
\] 
for a.e. $(x,y)\in[0,1]^{(2)}$, where $\indicator{G_{W_s}}{}$ is a `boundary correction term' that makes sure that the curve remains inside $[-1, 1]$. The reader is referred to Section~\ref{sec:frechet_derivatives} for details. In this sense our work is in the vein of \cite{Gangbo19} where the authors view the Wasserstein geometry as a projection of an $L^2$ geometry.  

\subsubsection{An example and discussion}\label{sec:intro_example}
To elucidate our results, consider the scalar entropy function $\Ent\colon\Graphons\to\rr\cup\Set*\infty$ (see~\cite{chatterjee2011large} for applications to large deviations of Erd\H{o}s-R\'enyi random graphs):
\begin{equation}\label{eq:scalarentropy}
    \Ent([W]) \coloneqq \int_0^1 \int_0^1 h(W(x,y))\diff x\diff y,\qquad [W]\in\Graphons,
\end{equation}
where $h\colon\rr \rightarrow \rr\cup\Set*\infty$ is the convex entropy function $h(p)\coloneqq p\log p + (1-p) \log (1-p)$, if $p\in(0,1)$, $h(0)=h(1)=0$, and $h(p)=\infty$, otherwise.

The function $\Ent$ is lower semicontinuous in the cut metric~\cite[Lemma 2.1]{chatterjee2011large}. However, for a given $\epsilon\in (0,1/2)$, if we restrict the domain of $\Ent$ to be all $[W]\in\Graphons$ such that $\epsilon \leq W \leq 1- \epsilon$, a.e., then $\Ent$ is $\delta_\cut$-continuous on this restricted domain. The function $\Ent$ can be shown to be convex along generalized geodesics and its local slope can be computed easily (see Section~\ref{sec:examples_GF} for all the details). The Fr\'echet-like derivative $D_{\Wcal}\Ent$ of $\Ent$ at any such graphon $[W]$ is given by another graphon that is ``coupled'' with $[W]$ (see Definition~\ref{def:coupled_graphons}) 
\begin{equation}\label{eq:velocityentropy}
    D_{\Wcal}\Ent(W)(x,y)= \log\left( \frac{W(x,y)}{1- W(x,y)} \right)\, , \qquad (x,y) \in [0,1]^{(2)}.
\end{equation}
Thus, starting from a graphon $[U_0]$ such that $\epsilon \leq U_0 \leq 1 - \epsilon$, a.e., the gradient flow $\omega$ evolves every coordinate of the graphon $\omega_t$ at time $t\in\rr_+$ by the velocity $-D_{\Graphons}\Ent(\omega_t)$ (see Section~\ref{sec:frechet_derivatives}). But the expression in equation~\eqref{eq:velocityentropy} is positive for any $W(x,y) > 1/2$ and negative for any $W(x,y) < 1/2$. Thus, the gradient flow converges, as $t \rightarrow \infty$, to the constant graphon $\omega_\infty \equiv 1/2$, a.e., which is the unique minimizer of the function $\Ent$. Restrict the domain of the scalar entropy function on $k\times k$ symmetric matrices $A$ with entries in $\interval{\epsilon, 1-\epsilon}$. Define $\Ent_k(A)\coloneqq k^{-2}\sum_{i=1}^k \sum_{j=1}^k h(A_{i,j})$. Then Theorem~\ref{thm:GF_convergence} further says the following. If $\omega^{(k)}$ is an Euclidean gradient flow of $\Ent_k$, and if $\toplim{\delta_\cut}_{k\rightarrow \infty}\omega^{(k)}(0)=[U_0]\in \Graphons$, then $(\omega^{(k)})_{k\in\Natural}$ converges, uniformly in the cut metric on compact sets $[0, T]$ for $T>0$, as $k\rightarrow \infty$, to the curve $\omega$ described.

It is important to note that the curve $\omega$ is actually obtained as the natural image of the curve $t\mapsto W_t$ obtained by solving
\begin{align}
    W_t(x, y)=W_0(x, y)-\int_{0}^{t}\log\round{\frac{W_s(x,y)}{1-W_s(x,y)}}\diff s,\quad \text{a.e. 
 } (x,y)\in\interval{0,1}^{(2)},
\end{align}
for $t\in\rr_+$, where the above integral is defined pointwise (See Section~\ref{sec:examples_GF} for details). It is a recurring theme in this paper that many important curves in $\Graphons$ can be seen as the natural image of a curve in $\Wcal$.

\sloppy More examples have been worked out in Section~\ref{sec:examples_GF}. This includes the case when $F$ is any (simple) graph homomorphism function that features prominently in the theory of graph convergence to graphons. Our examples also cover the gradient flow of any linear combination of the scalar entropy function and homomorphism functions that are of particular interest in the study of exponential random graph models (see~\cite{chatterjee2017large, chatterjee2013estimating, kenyon_yin_2017,eldan2018exponential,ghafouri2020survey}) and in the large deviation principle of dense random graphs~\cite{chatterjee2011large, chatterjee2017large, lubetzky2015replica, cook2020large} where one is interested in optimizing the so-called rate function. For example, see~\cite[Section 5.4.1]{ChernThesis} where the author is interested in computing the maximum likelihood estimate for a model of exponential random graphs analyzed in~\cite{chatterjee2013estimating}. Since graphs are discrete, the optimization is more amenable to analytical tools on the limiting graphon space. But the space of graphons is infinite dimensional and the author uses gradient descent on discretized block graphons to perform the gradient descent. Our results show that the method is consistent as the discretization gets finer and provides a mathematical justification to the algorithm in~\cite{ChernThesis}. 

It is well-known in optimal transport that an absolutely continuous curve in $\mathbb{W}_2$ has an associated continuity equation~\cite[Theorem 5.14]{santambrogio2015optimal}. Proposition~\ref{prop:AC_to_CE} gives a partial analogue of this result in current setting. That is, a curve of gradient flow (which is absolutely continuous) has an associated family of equations, not just one continuity equation. This connection becomes is more definitively explored in~\cite{HOPST22} where authors show that the gradient flows (and more general class of curves) can be described by a McKean-Vlasov type equation.

\subsection{A Computational example from extremal graph theory}\label{sec:mantel}
We give an interesting example in this section that suggests how the tools developed in this paper can be used to provide heuristics or search for counterexamples in many problems that are of interest in extremal graph theory.

Mantel's theorem~\cite{mantel1907problem} (a special case of Tur\'{a}n's theorem) states that the maximum number of edges in an $n$-vertex triangle-free graph is $n^2/4$. Further, any Hamiltonian graph with at least $n^2/4$ edges must either be the complete bipartite graph $K_{n/2,n/2}$ or it must be pancyclic~\cite{bondy1971pancyclic}. 

This suggests that if one maximizes the edge density subject to the condition that triangle density is $0$, then the maximizer should correspond to a complete bipartite graph. Our current theory does not allow for such constrained optimization. However, one can attempt to computationally ``verify'' this result by simulating gradient flows of linear combinations of homomorphism functions, $T_\triangle - \alpha T_{\mathrel{-}}$, over $(\Graphons,\delta_2)$ for sufficiently small weight $\alpha>0$. We make an arbitrary choice of weight, say $\alpha=1/10$ for numerical simulation and consider minimizing $T_\triangle - T_{\mathrel{-}}/10$ over $[W]\in\Graphons$ such that $0\leq W\leq 1$ a.e., where $\triangle$ and $\mathrel{-}$ are the triangle and the edge graphs respectively. This is akin to minimizing triangle density while also maximizing the edge density as much as possible. We see that graphons with small function values have small triangle density and large edge density. We set $n = 128$, step size $\tau = 10^{-3}$ and use the forward Euler method starting from an initial graphon $\squarebrack{W^{(n)}_0}\in\Graphons_{n}$ as shown in Figure~\ref{fig:mantel_init}. Figure~\ref{fig:Mantel} shows instances of the iterative process after $10^3$, $1.5\times 10^3$, $2.5\times 10^3$, $5\times 10^3$ and $10^4$ many steps. We see in Figure~\ref{fig:mantel_final} that after $10^4$ iteration, the kernel $W^{(n)}_{10^4}$ is close to the one corresponding to a complete bipartite graph as one would expect from Mantel's theorem. Our Theorem~\ref{thm:GF_convergence} implies that one should expect a similar evolution for all large values of $n$.

\begin{figure}
    \centering
    \subcaptionbox{$W^{(128)}_0$\label{fig:mantel_init}}{\includegraphics[width=0.3\linewidth]{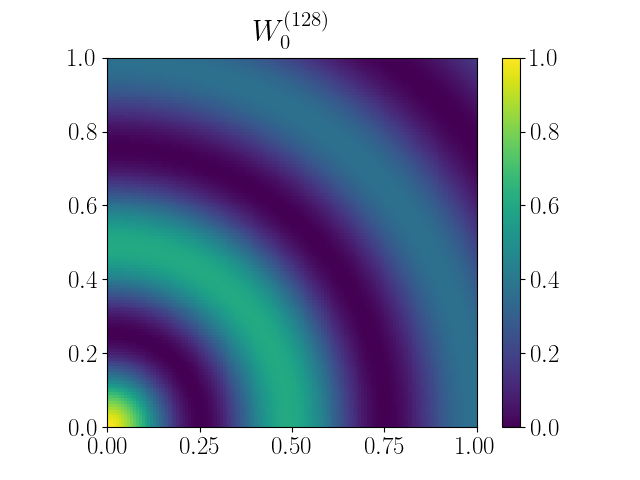}}
    \subcaptionbox{$W^{(128)}_{10^3}$}{\includegraphics[width=0.3\linewidth]{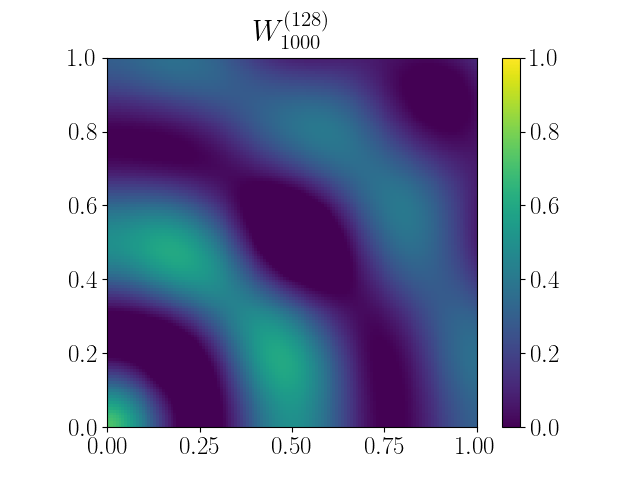}}
    \subcaptionbox{$W^{(128)}_{1.5\times 10^3}$}{\includegraphics[width=0.3\linewidth]{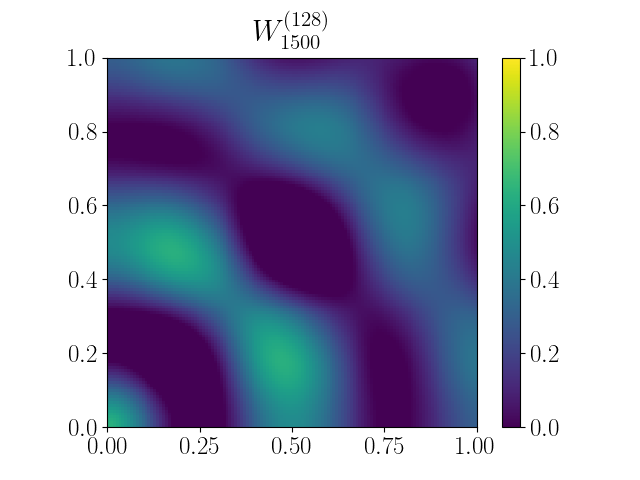}}
    \subcaptionbox{$W^{(128)}_{2.5\times 10^3}$}{\includegraphics[width=0.3\linewidth]{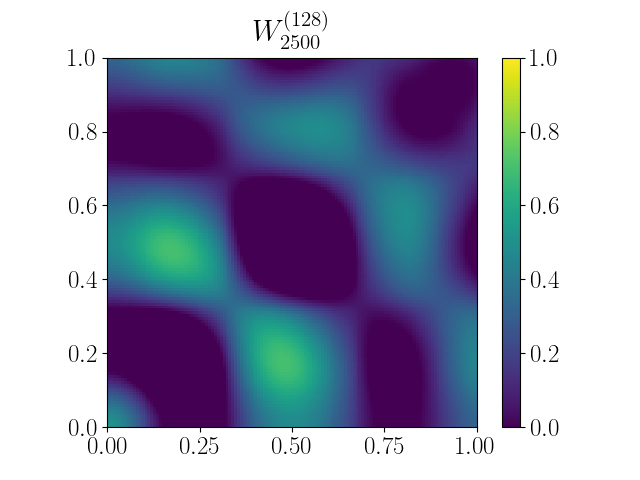}}
    \subcaptionbox{$W^{(128)}_{5\times 10^3}$}{\includegraphics[width=0.3\linewidth]{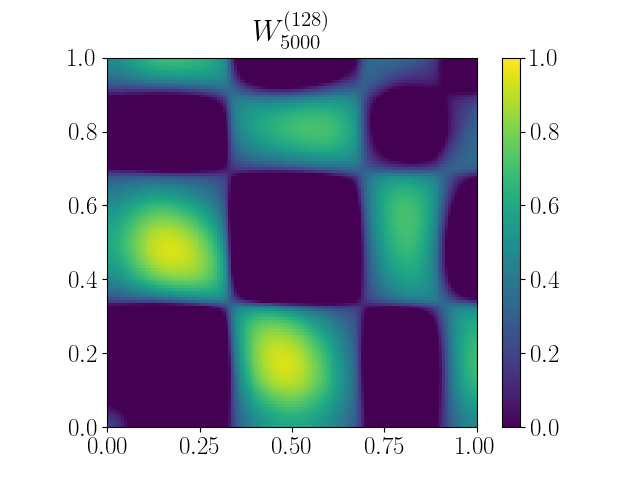}}
    \subcaptionbox{$W^{(128)}_{10^4}$\label{fig:mantel_final}}{\includegraphics[width=0.3\linewidth]{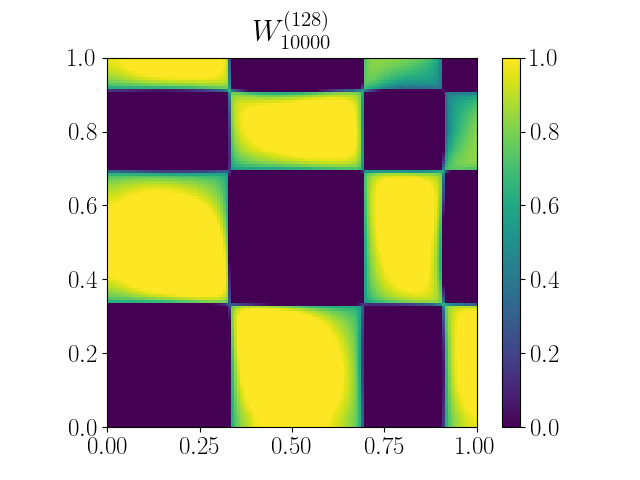}}
    \caption{A gradient descent simulation over $T_{\triangle} - T_{\mathrel{-}}/10$}
    \label{fig:Mantel}
\end{figure}

Linear combinations of the homomorphism density functions are non necessarily strictly convex (see Section~\ref{sec:examples_GF}) and hence the theory does not guarantee exponential rates of convergence. However, in the simulation discussed above, we statistically observe an exponential rate of convergence.

Such a computational and optimization driven approach can be used to study conjectures in extremal graph theory, especially for producing counterexamples. For instance, studying the non-negativity of a linear combination of homomorphism density is an important problem in extremal graph theory and it is known to be undecidable~\cite[Section 16.6.1]{lovasz2012large}. Although our techniques would not yield a proof but it can be a useful tool for the search of counterexamples. We leave this as an interesting direction for further investigation.

\subsection{Extensions and future directions}

A more complicated example of the previous set-up is a neural network (NN) with multiple hidden layers. See Figure~\ref{fig:finite_NN}, where the NN consists of $b\in\Natural$ hidden layers, an initial input $x_0 \in \rr^{d_0}$, and a terminal output $\hat{y}(x_0)\in \rr$ (say), and an intermediate sequence of transformations
\[
    x_0 \mapsto x_1 \in \rr^{d_1} \mapsto x_2\in \rr^{d_2} \mapsto \cdots \mapsto x_b \in \rr^{d_b} \mapsto \hat{y}.
\]
Each transformation involves a matrix $A_{i+1}\in\Rd{d_{i+1}\times d_i}$, a vector $\beta_{i+1}\in\Rd{d_{i+1}}$, and the transformation is defined as
\begin{align}
    x_{i+1} &= \sigma\round{A_{i+1}x_i + \beta_{i+1}},\label{eq:layer_update}
\end{align}
for all $i\in\Set*{0}\cup\squarebrack{b-1}$, where the function $\sigma\colon \rr\to\rr$ acts coordinate-wise. Finally, take $\hat{y}$ to be just the average of elements in $x_b$. 

\def\layersep{2cm}

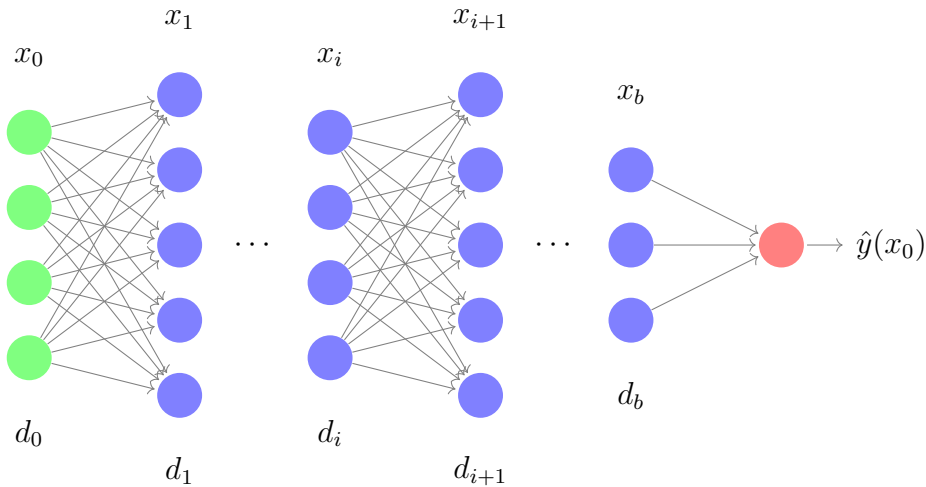
\begin{figure}
    \centering
    \begin{tikzpicture}[shorten >=1pt,->,draw=black!50, node distance=\layersep]
    \tikzstyle{every pin edge}=[<-,shorten <=1pt]
    \tikzstyle{neuron}=[circle,fill=black!25,minimum size=17pt,inner sep=0pt]
    \tikzstyle{input neuron}=[neuron, fill=green!50];
    \tikzstyle{output neuron}=[neuron, fill=red!50];
    \tikzstyle{hidden neuron}=[neuron, fill=blue!50];
    \tikzstyle{annot} = [text width=4em, text centered]

    \def\inp{4}
    \def\hone{5}
    \def\hi{4}
    \def\hipone{5}
    \def\hb{3}
    
    \foreach \name / \y in {1,...,\inp}
        \node[input neuron] (I-\name) at (0,-\y) {};

    \foreach \name / \y in {1,...,\hone}
        \path[yshift= (\hone-\inp)*0.5 cm]
            node[hidden neuron] (H1-\name) at (\layersep,-\y cm) {};

    \path[yshift=0cm]
    node[] (dots1) at (1.5*\layersep,-2.5 cm) {$\cdots$};

    \foreach \name / \y in {1,...,\hi}
        \path[yshift=(\hi-\inp)*0.5cm]
            node[hidden neuron] (Hi-\name) at (2*\layersep,-\y cm) {};

    \foreach \name / \y in {1,...,\hipone}
        \path[yshift=(\hipone-\inp)*0.5cm]
            node[hidden neuron] (Hipone-\name) at (3*\layersep,-\y cm) {};

    \path[yshift=0cm]
    node[] (dots1) at (3.5*\layersep,-2.5 cm) {$\cdots$};

    \foreach \name / \y in {1,...,\hb}
        \path[yshift=(\hb-\inp)*0.5cm]
            node[hidden neuron] (Hb-\name) at (4*\layersep,-\y cm) {};

    \node[output neuron,pin={[pin edge={->}]right:$\hat{y}(x_0)$}] (O) at (5*\layersep,-2.5 cm) {};

    \foreach \source in {1,...,\inp}
        \foreach \dest in {1,...,\hipone}
            \path (I-\source) edge (H1-\dest);

    \foreach \source in {1,...,\hi}
        \foreach \dest in {1,...,\hipone}
            \path (Hi-\source) edge (Hipone-\dest);

    \foreach \source in {1,...,\hb}
        \path (Hb-\source) edge (O);

    \node[annot,above of=I-1, node distance=1cm] (i) {$x_0$};
    \node[annot,above of=H1-1, node distance=1cm] (h1) {$x_1$};
    \node[annot,above of=Hi-1, node distance=1cm] (h2) {$x_i$};
    \node[annot,above of=Hipone-1, node distance=1cm] (h2) {$x_{i+1}$};
    \node[annot,above of=Hb-1, node distance=1cm] (hb) {$x_b$};
    
    \node[annot,below of=I-\inp, node distance=1cm] (di) {$d_0$};
    \node[annot,below of=H1-\hone, node distance=1cm] (dh1) {$d_1$};
    \node[annot,below of=Hi-\hi, node distance=1cm] (dh2) {$d_i$};
    \node[annot,below of=Hipone-\hipone, node distance=1cm] (dh2) {$d_{i+1}$};
    \node[annot,below of=Hb-\hb, node distance=1cm] (dhb) {$d_b$};
\end{tikzpicture}
    \caption{Finite width Neural Network with multiple hidden layers.}
    \label{fig:finite_NN}
\end{figure}

Given some probability model $\mu$ on $\rr^{d_0} \times \rr$, the goal is to minimize the risk function $R$, given by quadratic loss here for specificity, 
\begin{align}
    R\left( A_{i}, \beta_{i},\; i \in [b-1] \right)\coloneqq \Exp{(X,Y)\sim \mu}{\round{Y-\hat{y}(X)}^2},\label{eq:risk_def}
\end{align}
where the minimization is over all choices of the sequence of matrices $A_{i}$ and vectors $\beta_i$, for $i\in\squarebrack{b-1}$. Let us ignore the vectors $\beta_i$ from our discussion. The entries of the matrix $A_{i+1}$ can be thought of as associated with the edges of the bipartite graph connecting the nodes in layers $i$ and $i+1$. The output $R$ in equation~\eqref{eq:risk_def} does not depend on the labeling of the nodes in either layer, in the sense that if we relabel the nodes and correspondingly permute the rows and columns of $A_{i+1}$ the output $R$ remains the same. Therefore the risk function $R$ can be thought of as a function of edge weights of a sequence of bipartite graphs that is invariant under vertex relabeling. A gradient descent algorithm on the NN tries to compute the Euclidean gradient flow with respect to these edge weights to reach the minimum value of $R$. One can again ask the question: as the number of vertices in each layer goes to infinity, is there a scaling limit for the gradient flow of $R$? This is a multivariate generalization of our set-up in this paper. Instead of a single graph we have a sequence of $b$ graphs, all bipartite, and successive graphs share vertices. Nevertheless, if we assume that all the dimensions $d_0, \ldots, d_b$ go to infinity at the same rate while the edge weights remain bounded, the sequence of $b$ matrices converge in the space of graphons in the cut metric topology~\cite{lovasz2006limits, lovasz2012large}. Although we do not take up this multivariate extension in this paper, it serves as an inspiration for our current work. Also, the theory developed here does not cover this case of noisy gradient flows since the technology needed there are quite different. This construction is left for our upcoming work.

One may also wonder if our theory applies to hydrodynamic limits of natural stochastic gradient flows of random graphs or matrices~\cite{CraneAAP, AHR21}. This is what happens, for example, to particle systems diffusing according to the Fokker-Planck/Smoluchowski equations~\cite{JKO98}.

\subsection{Notations}
For any set $X$, we use $X^2$ to denote the usual Cartesian product $X\times X$ while $X^{(2)}$ is used to denote the set $X^2/{\sim}$ where we identify $(a,b)\sim(b,a)$ for all $a,b\in X$. We use this notation for domains of symmetric functions like kernels and arrays. In particular, the notation $[-1, 1]^{[n]^{(2)}}$ is used, with an abuse of notation, to mean the space of all symmetric matrices with entries in $[-1, 1]$. Similarly, for $\mathbb{R}^{[n]^{(2)}}$. 

Throughout the article we use the symbols $A$, $X$, $Y$ and $Z$ (and their variations with sub/superscripts, hats, tildes and primes) to represent matrices. For example, $X_{k}, A_k$ would stand for $k\times k$ (symmetric) matrices.

We use the letters $U$, $V$, $W$ to represent \emph{kernels}, that is, symmetric real-valued Borel measurable functions on the unit square. The set of all kernels is denoted by $\Wcal$ and $\Graphons$ denotes the set of graphons. 

For a kernel $V$, we always use $[V]$ to represent its corresponding graphon. Note that a $k\times k$ matrix, say $A_k$, can be naturally embedded in the space of kernels (see Section~\ref{sec:background} for details). We use $K(A_k)$ to denote the kernel corresponding to the matrix $A_k$.

\subsection{Organization of the paper} Section~\ref{sec:background} develops the background material on  graphons. The relevant definitions and results from the existing literature are given in Section~\ref{sec:background}. It also discusses concepts from the general theory of gradient flow on metric spaces that are required in this paper.

Section~\ref{sec:prelims} we prove some preliminary results. In particular, we show that the set of graphons $(\Graphons, \delta_2)$ equipped with the invariant $\delta_2$ metric is a geodesic metric space. This allows us to talk about convexity and so on. Another important takeway from this section is Lemma~\ref{lem:representation_graphons} which shows that every absolutely continuous curve in $(\Graphons, \delta_2)$ can be obtained as the natural image of an absolutely continuous curve in $(\Wcal, d_2)$.

Section~\ref{sec:grad_flows} is forms the core of the technical part of this paper. In particular, it introduces one of our key tools, namely Fr\'echet-like derivatives and shows the existence of curves of maximal slopes (a.k.a. gradient flow). Another important aspect of this section is to show that the the gradient flow can be written as the image of an absolutely continuous curve in $(\Wcal, d_2)$ and this absolutely continuous curve admits a simple and explicit form. This section also explores the relation between the the Fr\'echet-like derivative of a function on $\Graphons$ and the Euclidean derivative when this function is restricted to the finite dimensional symmetric matrices. This is particularly important in the proof of our main theorem. Section~\ref{sec:convergence} gives the proof of Theorem~\ref{thm:GF_convergence}. Finally, Section~\ref{sec:continuity_eq} discusses the continuity equations that are associated to the gradient flows. 

Section~\ref{sec:examples_GF} works out examples of functions on graphons for which we can compute Fr\'echet-like derivatives and gradient flows. This includes the well-known graph homomorphism functions and scalar entropy and their linear combinations. Parallels are drawn to the well-known classes of functions studied in optimal transport: potential energy, interaction energy, and internal energy. 
\section{Background material}\label{sec:background}

In Section~\ref{sec:topologies_on_graphons}, we introduce the required metric on graphons and other properties related to graphons. The material in this section is mostly borrowed from~\cite{lovasz2012large, janson2016graphons}. In Section~\ref{sec:grad_flow_graphons}, we introduce the necessary terminology to talk about the gradient flow on a metric space. The material in that subsection is adapted from~\cite{ambrosio2005gradient}. Readers familiar with theory of graphons~\cite{lovasz2012large} and the theory of gradient flow in arbitrary metric spaces~\cite{ambrosio2005gradient} may safely skip this section.

\subsection{Graphons and Metrics on Graphons}\label{sec:topologies_on_graphons}
Recall that a kernel $W\colon [0, 1]^{(2)}\to [-1, 1]$ is a Borel measurable, symmetric function. On the space of kernels, $\Wcal\subset L^2\big([0,1]^{(2)}\big)$, we have the usual $L^2$ norm, $\enorm{}$, that is, $\enorm{W}^2\coloneqq \int_{[0, 1]^2}\abs{W(x, y)}^2\diff x \diff y$. We also define the cut norm, denoted $\cutnorm{}$, on $\Wcal$ as follows. 

\begin{definition}[Cut norm]\label{def:cut_norm}
The {\em cut norm} $\cutnorm{}\colon\Wcal \to \rr_+$ is defined as
\begin{align*}
    \cutnorm{W} &\coloneqq \sup_{S,T\subseteq \interval{0,1}}\abs{\int_{S\times T} W(x,y)\diff x \diff y}\\
    &= \sup_{\norm{\infty}{f}, \norm{\infty}{g} \le 1}\abs{\int_{\interval{0,1}^{2}} W(x,y)f(x)g(y)\diff x\diff y},
\end{align*}
for all $W\in\Wcal$ where $S$ and $T$ are Borel measurable subsets of $\interval{0,1}$, and $f$ and $g$ are Borel measurable functions on $\interval{0,1}$, and $\norm{\infty}{}$ denotes the usual $L^\infty$ norm.
\end{definition}
The cut norm was first introduced in~\cite{frieze1999quick} in the context of matrices and was later extended to kernels in~\cite{borgs2008convergent}. The cut norm is used to define a metric called the cut metric, $\delta_\cut$, on the space of graphons. In the following definitions we use $\Tcal$ to denote the set of all Lebesgue measure preserving maps $\varphi\colon[0, 1]\to [0, 1]$, and $\Ical$ to denote the set of all invertible Lebesgue measure preserving maps $\varphi\colon[0, 1]\to [0, 1]$. Given a kernel $W\in\Wcal$ and a Lebesgue measure preserving map $\varphi\in\Tcal$, one can define $W^\varphi\in\Wcal$ as $W^\varphi(x,y) \coloneqq W(\varphi(x),\varphi(y))$ for Lebesgue a.e. $x,y\in\interval{0,1}$.

From hereon in the text, we will always refer to Lebesgue measure preserving transformations, and Lebesgue almost everywhere (a.e.) unless explicitly specified otherwise.

For any $k\in\Natural$, let $Q_k \coloneqq \Set*{Q_{k,i}}_{i\in[k]}$ be a partition of $\interval{0,1}$ into intervals of equal Lebesgue measure, for example $Q_{k,1} \coloneqq \interval{0,1/k}$ and $Q_{k,i}\coloneqq ((i-1)/k,i/k]$ for $i\in[k]\setminus\Set*1$. For every permutation $\pi\in S_k$ we can define an invertible Lebesgue measure preserving map $\tilde\pi\colon\interval{0,1}\to\interval{0,1}$ such that $\tilde{\pi}$ is an increasing affine homeomorphism from $Q_{k,i}$ to $Q_{k,\pi(i)}$ for each $i\in\squarebrack{k}$. We denote the set of all such maps by the set $\Ical_k$.

\begin{definition}[Cut metric~{\cite[Section 3.2]{borgs2008convergent}}]\label{def:cut_metric}
The {\em cut metric} $\delta_\cut\colon \Graphons\times \Graphons \to \rr_+$ is defined as
\begin{align}
    \begin{split}
        \cutmetric{\squarebrack{W_0}}{\squarebrack{W_1}} &\coloneqq \inf_{\varphi_0,\varphi_1\in\Tcal}\cutnorm{W_0^{\varphi_0} - W_1^{\varphi_1}}\\
        &= \inf_{\varphi\in\Ical}\cutnorm{W_0 - W_1^{\varphi}}
        = \lim_{k\to\infty}\min_{\tilde\pi\in\Ical_k}\cutnorm{W_0 - W_1^{\tilde\pi}},
    \end{split}
    \label{eq:cut_metric}
\end{align}
for all $\squarebrack{W_0}, \squarebrack{W_1}\in\Graphons$, where the latter two equalities are due to~\cite[Lemma 3.5]{borgs2008convergent}.
\end{definition}

It is worth pointing that $\delta_{\cut}$ naturally extends to kernels, but it only defines a pseudometric on $\Wcal$. In fact, $\delta_{\cut}(W_1, W_2)=0$ if and only if there exists $U\in \Wcal$ and Lebesgue measure preserving transforms $\varphi_1, \varphi_2\colon [0, 1]\to [0, 1]$ such that $W_i=U^{\varphi_i}$, for $i\in\squarebrack{2}$. The kernels $W_1, W_2$ in this case are said to be \textit{weakly isomorphic}. In other words, graphons can be defined as the class of kernels identified up to weak isomorphism. We can also define the so-called invariant $L^2$ metric on $\Graphons$.
\begin{definition}[Invariant $L^2$ metric~\cite{borgs2016sparse,janson2016graphons}]\label{def:L2_metric}
    The {\em invariant $L^2$ metric} $\delta_2 \colon \Graphons\times \Graphons \to \rr_+$ is defined as
    \begin{align}
        \delta_2\round{\squarebrack{W_0},\squarebrack{W_1}} &\coloneqq \inf_{\varphi_1,\varphi_2\in\Tcal}\norm{2}{W_0^{\varphi_1}-W_1^{\varphi_2}} = \inf_{\varphi\in\Ical}\norm{2}{W_0 - W_1^{\varphi}},
    \end{align}
    for all $\squarebrack{W_0},\squarebrack{W_1}\in\Graphons$, where $\norm{2}{}\colon L^2\big(\interval{0,1}^{2}\big) \to \rr_+$ is the usual $L^2$-norm, and the second equality is a consequence of~\cite[Theorem 8.13]{lovasz2012large}.
\end{definition}
We denote the metrics induced by the cut norm and the $L^2$-norm as $d_\cut$ and $d_2$ respectively. The space $(\Graphons, \delta_\cut)$ is a compact metric space~\cite{Lovsz2007SzemerdisLF},~\cite[Section 9.3]{lovasz2012large} while the metric space $(\Graphons,\delta_2)$ is complete and separable, but not compact. It is clear that convergence in $\delta_2$ implies the convergence in $\delta_{\cut}$, that is, the topology generated by $\delta_{\cut}$ is weaker than the topology generated by $\delta_{2}$. The following Lemma says that the metric $\delta_2$ is lower semicontinuous with respect to $\delta_{\cut}$.

\begin{lemma}{~\cite[Lemma 14.16]{lovasz2012large}}
\label{lem:delta_2_lsc}
    \sloppy The metric $\delta_2$ is sequentially $\delta_\cut$-lower semicontinuous, i.e., if sequences $\round{\squarebrack{U_n}}_{n\in\Natural}, \round{\squarebrack{V_n}}_{n\in\Natural} \subset\Graphons$, and $\squarebrack{U},\squarebrack{V}\in\Graphons$ such that $\round{\squarebrack{U_n}}_{n\in\Natural}\xrightarrow{\delta_\cut}\squarebrack{U}$ and $\round{\squarebrack{V_n}}_{n\in\Natural}\xrightarrow{\delta_\cut}\squarebrack{V}$, then 
    \[
        \liminf_{n\to\infty}\delta_2\round{\squarebrack{U_n},\squarebrack{V_n}} \geq \delta_2\round{\squarebrack{U},\squarebrack{V}}.
    \]
\end{lemma}

As we mentioned in the Section~\ref{sec:intro}, the gradient flow on the space of graphons will be with respect to the invariant $L^2$ metric, but the convergence statements will be with respect to the topology generated by the cut metric.

Let us mention in passing that the invariant $L^2$ metric is closely related to the popular Gromov-Wasserstein metric~\cite{GWMemoli} used to compare two metric measure spaces or their sample equivalents~\cite{GWsample}. This can be seen by considering $\interval{0,1}$ as a metric measure space where the measure is the Lebesgue measure and, for a given bounded metric $\mathbf{d}$, one defines a graphon as $W(x,y)=\mathbf{d}(x,y)$ for $x,y\in[0,1]$. Then the Gromov-Wasserstein distance (for $p=2$) between $(\interval{0,1}, \lambda_{[0,1]}, \mathbf{d})$ and $(\interval{0,1}, \lambda_{[0,1]}, \mathbf{d}')$, for two distances $\mathbf{d}$ and $\mathbf{d}'$, is the same as computing the invariant $L^2$ distance between the corresponding graphons. In this vein also see the unpublished article \cite{Sturm12} which constructs gradient flows on the space of metric measure spaces in a spirit that is quite similar to ours.

\subsubsection{Block Graphons, Matrices and Graphons}\label{sec:matrix_kernel_graphon}

For any $k\in\Natural$, define the set of kernels $\Wcal_k\subset \Wcal$ which contain kernels which are constant a.e. over sets in $Q_k\times Q_k$. The set $\Wcal_k$ can be naturally identified with a convex subset of the finite dimensional vector space of symmetric $k\times k$ matrices. Since this identification will be used often, we make it a definition.
\begin{definition}[Kernels and finite symmetric matrices]\label{def:K_Mk}
    For any $k\in\Natural$, and a symmetric matrix $A\in\Mcal_k$, the kernel $K(A)$ corresponds to the element in $\Wcal_k$ which takes the constant value $A_{i,j}$ on $Q_{k,i}\times Q_{k,j}$ for $i,j\in[k]$. The inverse map from $\Wcal_k$ to $\Mcal_k$ is denoted by $M_k$.
\end{definition}
Kernels in $\Wcal_k$ can be thought of as adjacency matrices of vertex labeled graphs with edge weights. For a finite graph $G = (V, E)$ with vertices labeled as $\squarebrack{k}$ and associated weights with every edge with weight $w(\Set*{i,j})\in\interval{-1,1}$ for $\Set*{i,j}\in E$, we can construct its adjacency matrix $A\in\Mcal_k$ by defining 
\[
    A_{i,j} \coloneqq w(\Set*{i,j})\indicator{}{\Set*{i,j}\in E},\quad i,j\in\squarebrack{k}.
\]
Thus, we can also map vertex labeled graphs with weights associated with its edges to kernels by considering $K(A)\in\Wcal_k$. Naturally, we can also define $\Graphons_k \coloneqq \Wcal_k/{\cong}$. Given a graphon $[W]\in\Graphons$ and $k\in\Natural$, we can obtain a $k\times k$ exchangeable symmetric array as described in the following definition. 

\begin{definition}[Sampling random graphs from graphons]\label{def:sampling_from_graphon}
    Starting with a graphon $\squarebrack{W}\in\Graphons$, for any $k\in\Natural$ we can sample a random graph $G_k{\squarebrack{W}}$ of size $k$ as follows. Consider any representative element $W\in\squarebrack{W}$ and sample $k$ i.i.d. elements $\Set*{U_i}_{i=1}^k$ uniformly at random from $\interval{0,1}$, and assign edge weight $W(U_i, U_j)$ to edge $\Set*{i,j}$ for all $(i,j)\in\squarebrack{k}^{(2)}$. By an abuse of notation, we also denote the exchangeable symmetric $k\times k$ weighted adjacency matrix of this random graph by $G_k\squarebrack{W}$. The distinction will be apparent from the context. In either case, $G_k\squarebrack{W}$ is measurable with respect to $\sigma(\Set*{U_i}_{i=1}^k)$. 
\end{definition}

\begin{definition}[Extensions to $L^p$ kernels for $p\in\interval{1,\infty}$]\label{def:Lp_extensions}
    Sometimes in our text, we will consider kernels and matrices whose entries are not necessarily in $\interval{-1,1}$, but are rather elements in $L^2\big(\interval{0,1}^{(2)}\big)$ or $L^\infty\big(\interval{0,1}^{(2)}\big)$. For any $k\in\Natural$, just like we defined $\Wcal_k$, we can restrict our attention to the subset of functions $L^p_k\big(\interval{0,1}^{(2)}\big) \subset L^p\big(\interval{0,1}^{(2)}\big)$ for every $p\in\interval{1,\infty}$ which contain symmetric measurable step functions over $Q_k\times Q_k$. Using the equivalence relation $\cong$, just like we defined $\Graphons$ and $\Graphons_k$, we can similarly define $\widehat{L}^p\big(\interval{0,1}^{(2)}\big) \coloneqq L^p\big(\interval{0,1}^{(2)}\big)/{\cong}$ and $\widehat{L}^p_k\big(\interval{0,1}^{(2)}\big) \coloneqq L^p_k\big(\interval{0,1}^{(2)}\big)/{\cong}$ for any $p\in\interval{1,\infty}$. When it is clear from the context, we will also call the elements in $\widehat{L}^{\infty}$ graphons. For simplicity, we use $K$ and $M_k$ for $k\in\Natural$ from Definition~\ref{def:K_Mk} even when the kernels are in $L^p\big(\interval{0,1}^{(2)}\big)$ for $p\in\interval{1,\infty}$.
\end{definition}

\subsection{Gradient Flows on metric spaces}\label{sec:grad_flow_graphons}
The theory of gradient flow on a general metric space is well-developed by now and can be found in~\cite{ambrosio2005gradient}. Since our goal is to define gradient flows on $(\Graphons,\delta_2)$, the definitions below are sometimes not the most general versions as given in~\cite{ambrosio2005gradient} but adapted to our particular setting.

\begin{definition}[Absolutely continuous curves]\label{def:ac_curve}
    \sloppy For a metric space $(X,d)$, and any time horizon $T\in\rr_+$, a curve $\omega = \round{\omega_t}_{t\in[0,T]}$ in $X$ is {\em absolutely continuous} with respect to the metric $d$ if there exists $m\in L^1(\interval{0,T})$ such that for all $0\leq r<s\leq T$
    \begin{align}
        d\round{\omega_r,\omega_s} &\leq \int_{r}^s m(t)\diff t.\label{eq:absolutely_continuous_ineq}
    \end{align}
    The set of such absolutely continuous curves will be denoted as $\mathrm{AC}(X,d)$.
\end{definition}

\begin{definition}[Metric derivative]\label{def:metric_derivative}
    For a metric space $(X,d)$, and any $T\in\rr_+$, the {\em metric derivative} $\abs{\omega'}(t)$ of a curve  $\omega = \round{\omega_t}_{t\in[0,T]}$ in $X$ at $t\in\round{0,T}$ is defined as
    \begin{align}
        \abs{\omega'}(t) \coloneqq \lim_{s\to t}\frac{d\round{\omega_s,\omega_t}}{\abs{s-t}},\label{eq:metric_derivative}
    \end{align}
    provided this limit exists.
\end{definition}

If $\omega\in\mathrm{AC}\round{X, d}$, then the limit in equation~\eqref{eq:metric_derivative} exists for a.e. $t\in\round{0,T}$ and $\abs{\omega'}\in L^1\round{\interval{0,T}}$~\cite[Theorem 1.1.2]{ambrosio2005gradient}. In other words, every absolutely continuous curve in a metric space has metric derivative defined almost everywhere. And conversely, if the metric derivative $\abs{\omega'}(t)$ exists for a.e. $t\in (0, T)$ and $\abs{\omega'}\in L^1([0, T])$, then $\omega$ is absolutely continuous. 

We now need to define some notion for the derivative of a function $F\colon X\to \mathbb{R}\cup \{\infty\}$. On a metric space the usual notion of derivative can not be defined. However, the following~\cite[Definition 1.2.4]{ambrosio2005gradient} acts as a substitute in many situations of interest.

\begin{definition}[Local slope]\label{def:local_slope}
    The {\em local slope} $\abs{\pdiff F}(v)$ of $F\colon X\to\rr\cup\Set*{+\infty}$ on a metric space $(X,d)$, at $v\in \effdom(F)$ is defined as
    \begin{align}
        \abs{\pdiff F}(v) \coloneqq \limsup_{w\in X,\,d(v,w)\to 0}\frac{\round{F(v)-F(w)}^+}{d(v,w)}.
    \end{align}
\end{definition}

The definition below is narrower than the one in~\cite[Definition 1.3.2]{ambrosio2005gradient} since we restrict our choice of \textit{upper gradient} in that definition to the local slope~\cite[Theorem 1.2.5]{ambrosio2005gradient}.

\begin{definition}[Curves of maximal slope]\label{def:curves_of_maximal_slope}
    On a metric space $(X,d)$, any locally absolutely continuous curve $\omega = \round{\omega_t}_{t\in[0,T]}$ in $X$ on a finite time horizon $T>0$ is a {\em curve of maximal slope} for the function $F\colon X\to\rr\cup\Set*{+\infty}$ with respect to its local slope, if $F\circ \omega =G$ a.e. for some non-increasing map $G$ on $(0,T)$, and
    \begin{align}
        G'(t) &\leq -\inv{2}\abs{\omega'}^2(t) - \inv{2}\abs{\pdiff F}^2(\omega_t), \qquad \text{a.e.}\quad t\in (0, T).
    \end{align}
\end{definition}

On a general metric space, a curve of maximal slope can be referred to as a gradient flow although the concept of gradient itself is absent. See~\cite[Section 1.3]{ambrosio2005gradient} for the intuition. 

\begin{definition}[Length]\label{def:length}
    Given the metric space $(X,d)$, and a curve  $\omega = \round{\omega_t}_{t\in[0,T]}$ in $X$, the {\em length} of $\omega$ is defined as
    \begin{align*}
        \ell(\omega) &\coloneqq \sup\Set*{\sum_{k=0}^{n-1} d\round{\omega_{t_k},\omega_{t_{k+1}}} \given n\in\Natural, 0 = t_0 < t_1 < \cdots < t_n = T }\ .
    \end{align*}
\end{definition}
It is clear from Definition~\ref{def:length} that for any absolutely continuous curve  $\omega = \round{\omega_t}_{t\in[0,T]}$ in $X$ and $x,y\in X$ such that $\omega_0=x$, $\omega_T=y$, we have $\ell(\omega)\ge d(x, y)$. Given $x, y\in X$ it is natural to ask if there is an absolutely continuous curve $\omega$ from $x$ to $y$ that achieves the length $\ell(\omega)=d(x, y)$. Such a curve is called a \emph{geodesic} between $x$ and $y$. If there exists a geodesic $\omega$ between any two points $x, y\in X$, we say that $(X, d)$ is a geodesic metric space. In a geodesic metric space, notions like convexity and semiconvexity make sense. We make those precise in the following definitions.

\begin{definition}[Geodesic metric space]\label{def:geodesic_metric_space}
    A metric space $(X,d)$ is called a {\em geodesic metric space} if for all $x,y\in X$
    \begin{align*}
        d(x,y) &= \min\Set*{\ell(\omega) \given \omega\in\mathrm{AC}(X,d), \omega_0 = x, \omega_1 = y}\ .
    \end{align*}
\end{definition}

\begin{definition}[Constant speed geodesics]\label{def:constant_speed_geodesics}
    On a metric space $(X,d)$, a curve  $\omega = \round{\omega_t}_{t\in[0,1]}$ in $X$ is a {\em constant speed geodesic} if for all $0\leq r\leq s\leq 1$,
    \begin{equation}
    \label{eqn:def_constant_speed}
        \metric{d}{\omega_r}{\omega_s} = \metric{d}{\omega_0}{\omega_1}\round{s-r}.
    \end{equation}
\end{definition}

Note that if a curve $\omega$ satisfies equation~\eqref{eqn:def_constant_speed}, then $\omega$ is clearly Lipschitz and hence absolutely continuous. It is easy to see that such a curve $\omega$ is indeed a geodesic and the metric derivative $\abs{\omega'}(t)=d(\omega_0, \omega_1)$ for a.e. $t\in[0,1]$. This justifies the name `constant speed geodesic'. 
\begin{remark}\label{rem:reparameterize}
    It is also worth pointing that not only every geodesic, but every absolutely continuous curve can be reparametrized so that it becomes Lipschitz~\cite[Box 5.1]{santambrogio2015optimal} under the new parametrization.
\end{remark}

We now make precise the notion of convexity in metric spaces. On a metric space, we first define convexity (and semiconvexity) along curves. If a function is convex (or semiconvex) along every constant speed geodesic, then we call it convex with respect to the metric.

\begin{definition}[$\lambda$-semiconvexity along curves w.r.t. a metric]\label{def:lambda_cvx_along_curve}
    On a metric space $(X,d)$, a function $F\colon X\to\rr\cup\Set*\infty$ is said to be {\em $\lambda$-semiconvex} with respect to the metric $d$ along a curve  $\omega = \round{\omega_t}_{t\in[0,1]}$ in $X$ for some $\lambda\in\rr$, if
    \begin{align}
        F(\omega_t) &\leq (1-t)F(\omega_0) + tF(\omega_1) - \inv{2}\lambda t(1-t)\metric{d^2}{\omega_0}{\omega_1},\label{eq:lambda_cvx_along_curve}
    \end{align}
    for all $t\in\interval{0,1}$. Particularly, if the above inequality holds for $\lambda = 0$, then we say that $F$ is convex with respect to the metric $d$ along the curve $\omega$.
\end{definition}
\begin{definition}[$\lambda$-geodesic semiconvexity w.r.t. a metric]\label{def:semiconvexity}
    On a metric space $(X,d)$, a function $F\colon X\to\rr\cup\Set*\infty$ is {\em $\lambda$-geodesically semiconvex} with respect to the metric $d$, if for any $v_0, v_1\in \effdom(F)$ there exists a constant speed geodesic  $\omega = \round{\omega_t}_{t\in[0,T]}$ on $(X,d)$ (Definition~\ref{def:constant_speed_geodesics}) with $\omega_0=v_0$ and $\omega_1=v_1$ such that $F$ is $\lambda$-semiconvex on $\omega$ with respect to the metric $d$ for some $\lambda\in\rr$ (Definition~\ref{def:lambda_cvx_along_curve}).
\end{definition}

\section{Preliminaries}\label{sec:prelims}
In this section, we prove some results that are used in the proof of our main results but are also of independent interest. The two key results in this section are Lemma~\ref{lem:representation_graphons} and Proposition~\ref{prop:graphon_geodesic_space}.  Proposition~\ref{prop:graphon_geodesic_space} states that $(\Graphons, \delta_2)$ is a geodesic metric space.  If $\round{W_t}_{t\in\interval{0,1}}\in\mathrm{AC}(\Wcal, d_2)$. It is easily seen that $\round{\omega_t\coloneqq[W_t]}_{t\in\interval{0,1}}\in\mathrm{AC}(\Graphons, \delta_2)$. Lemma~\ref{lem:representation_graphons} shows that the converse is also true.

\begin{lemma}\label{lem:representation_graphons}
    Let $\omega = \round{\omega_t}_{t\in[0,1]}\in\mathrm{AC}(\Graphons,\delta_2)$. Then there exists $W = \round{W_t}_{t\in\interval{0,1}} \in \mathrm{AC}\round{\Wcal,d_2}$ such that $\omega_t = \squarebrack{W_t}$, and $\delta_2(\omega_t,\omega_s)=\enorm{W_t-W_s}$ for all $s,t\in\interval{0,1}$.
\end{lemma}
The proof of Lemma~\ref{lem:representation_graphons} requires a strengthening of~\cite[Theorem 8.13]{lovasz2012large},~\cite[Theorem 6.16]{janson2010graphons} that we state and prove below. Before we begin the proof, we define some notations. Let $(\Omega, \Fcal, \mu)$ be a probability space over some Polish space $\Omega$ equipped with the usual Borel sigma algebra $\Fcal$. For a kernel $W$ on $\Omega$, that is, $W:\Omega\times \Omega\to \rr$ we define the norm $\norm{2,\Omega,\mu}{}$ as
\[
    \norm{2,\Omega,\mu}{W}^2 \coloneqq \int_{\Omega^2}\abs{W(x, y)}^2 \mu(\diff x)\mu(\diff y).
\]

Also, Let $W\in \Wcal$ be a kernel. Let $\varphi\colon(\Omega, \Fcal, \mu)\to ([0, 1],\Bcal([0,1]),\lambda_{[0,1]})$ be a measure preserving map (i.e., $\mu(\varphi^{-1}(B))=\lambda_{[0,1]}(B)$ for all Borel sets $B\subseteq[0,1]$). We can define $W^{\varphi}$ as a kernel on $\Omega^{(2)}$ as 
\begin{equation}\label{eq:generalomega}
    W^{\varphi}(\omega_1, \omega_2) \coloneqq W(\varphi(\omega_1), \varphi(\omega_2)), \quad \text{for $\mu$-a.e.}\quad \omega_1,\omega_2\in\Omega.
\end{equation}

Let $\pi, \rho\colon[0, 1]^2\to [0, 1]$ be the usual coordinate projection maps, that is, $\pi\colon(x, y)\mapsto x$ and $\rho\colon(x, y)\mapsto y$. Using equation~\eqref{eq:generalomega}, we can define $W^{\pi}$ and $W^{\rho}$ as kernels on $\Omega=[0, 1]^2$ for every kernel $W\in \Wcal$. For example, 
\[
    W^{\pi}((x_1, y_1), (x_2, y_2))\coloneqq W(x_1, x_2), \qquad (x_1,y_1),(x_2,y_2)\in[0,1]^2.
\]
It is easy to see that $W^{\pi}$ is symmetric on $[0, 1]^2\times [0, 1]^2$.

In the following discussion, we always equip $[0, 1]$ with the Borel sigma-algebra and the Lebesgue measure, often without explicitly mentioning.
\begin{lemma}\label{lemma:MinimizingCoupling_gen}
    Let $\omega_1, \ldots, \omega_n\in \Graphons$. Then there exist $W_1, \ldots, W_n\in \Wcal$ such that $[W_i]=\omega_i$ and $\enorm{W_i-W_{i+1}}=\delta_2(\omega_i, \omega_{i+1})$ for every $i\in\squarebrack{n-1}$.
\end{lemma}
\begin{proof}
Let $U_i\in\omega_i$ for $i\in [n]$. From~\cite[Theorem 8.13]{lovasz2012large} there exist probability measures $\mu_i$ on $[0, 1]^2$ for $i\in [n-1]$ such that each $\mu_i$ is a coupling of Lebesgue measures satisfying 
\begin{align}\label{eqn:delta2coupling}
    \delta_2(\omega_i, \omega_{i+1}) &= \norm{2,[0, 1]^2, \mu_i}{U_i^{\pi}-U_{i+1}^{\rho}}.
\end{align}

\sloppy Let $\pi_i\colon[0, 1]^n\to [0,1]$ be the usual projection map on the $i$-th coordinate. By the gluing lemma~\cite[Lemma 7.6]{V03}, there exists a measure $\tilde{\mu}$ on $[0, 1]^n$ such that $(\pi_i, \pi_{i+1})_\sharp\tilde{\mu}=\mu_i$. Therefore we have
\begin{equation}
\label{eqn:from_2_to_n}
    \norm{2,[0, 1]^2, \mu_i}{U_i^{\pi}-U_{i+1}^{\rho}} =\norm{2,[0, 1]^n,\tilde{\mu}}{U_i^{\pi_i}-U_{i+1}^{\pi_{i+1}}}.
\end{equation}

Let $\eta\colon[0, 1]\to ([0, 1]^n, \tilde{\mu})$ be a measure preserving bijection and let $\varphi_i\coloneqq\pi_i\circ\eta$. Then $\varphi_i\colon[0, 1]\to [0, 1]$ is measure preserving and therefore we obtain 
\begin{align}
\label{eqn:from_n_to_1}
 \norm{2,[0, 1]^n,\tilde{\mu}}{U_i^{\pi_i}-U_{i+1}^{\pi_{i+1}}}=\norm{2, [0, 1]}{U_i^{\varphi_i}-U_{i+1}^{\varphi_{i+1}}}\;.
\end{align}
Combining equations~\eqref{eqn:delta2coupling},~\eqref{eqn:from_2_to_n} and~\eqref{eqn:from_n_to_1}, and taking $W_i=U_i^{\varphi_i}$ for all $i\in\squarebrack{n}$, yields $\delta_2(\omega_i, \omega_{i+1})=\enorm{W_i-W_{i+1}}$. This completes the proof.
\end{proof}

\begin{proof}[Proof of Lemma~\ref{lem:representation_graphons}]
Following Remark~\ref{rem:reparameterize}, assume (possibly after a reparametrization) that the curve $\omega$ is Lipschitz with Lipschitz constant $L\geq 0$. Let $n\in\Natural$. From Lemma~\ref{lemma:MinimizingCoupling_gen}, there exists $W_{i,n}\in\Wcal$ such that $\squarebrack{W_{i,n}}=\omega_{i/n}$ for all $i\in \Set*{0}\cup\squarebrack{n}$, and
\[
    \enorm{W_{i,n}-W_{i+1,n}}=\delta_2(\omega_{i/n}, \omega_{(i+1)/n}),
\] 
for all $i\in\squarebrack{n-1}$. For each $n\in \mathbb{N}$, let us define the curve $W^{(n)} = \big(W^{(n)}_t\big)_{t\in[0,1]}$ as
\[
    W^{(n)}_t\coloneqq(1-nt+i)W_{i,n}+(nt-i)W_{i+1,n},
\]
when $t\in [i/n, (i+1)/n]$ for some $i\in\squarebrack{n-1}$. Note that $W^{(n)}$ is also Lipschitz with constant $L$ and therefore the family $\setinline{W^{(n)}}_{n\in \mathbb{N}}$ is equicontinuous w.r.t. $d_2$.

Since $\Wcal\subseteq L^2\big([0,1]^{(2)}\big)$ is bounded in $L^2\big([0,1]^{(2)}\big)$, it is weak-$*$ precompact~\cite[Box 1.2]{santambrogio2015optimal}. Since $\setinline{W^{(n)}}_{n\in \mathbb{N}}$ is equicontinuous w.r.t. $d_2$, it will also be equicontinuous w.r.t. the weak-$*$ topology. It follows from Ascoli's theorem~\cite[Theorem 47.1]{munkres2000topology} (possibly after passing to a subsequence and relabeling) that $\big(W^{(n)}\big)_{n\in\Natural}$ converges uniformly in weak-$*$ to some curve $\round{W_t}_{t\in[0,1]}\subseteq L^2\big([0,1]^{(2)}\big)$. It is easy to see that $W_t$ is symmetric and $\abs{W_t}\leq 1$ a.e. on $[0,1]^{(2)}$ and hence $W_t\in \Wcal$ for every $t\in [0, 1]$. 

To conclude our proof, we show that $\round{W_t}_{t\in[0,1]}$ is Lipschitz in $\norm{2}{}$ and that $\delta_2([W_t], \omega_t)=0$ for all $t\in [0, 1]\cap\Rational$ (therefore $[W_t]=\omega_t$ for rational $t$). Since $\omega$ is also Lipschitz, it follows that $\omega_t=[W_t]$ for all $t\in [0, 1]$. 

To see that $\round{W_t}_{t\in[0,1]}$ is Lipschitz, observe that for any $s,t\in[0,1]$,
\begin{align*}
    \left\langle W^{(n)}_t-W^{(n)}_s, W_t-W_s\right\rangle \to \enorm{W_t-W_s}^2.
\end{align*}
Using Cauchy--Schwarz inequality, we obtain
\begin{align*}
    \enorm{W_t-W_s}\le \liminf_{n\to\infty}\enorm{W^{(n)}_t-W^{(n)}_s}\le L\abs{t-s}.
\end{align*}
We now show that $\delta_{2}([W_t], \omega_t)=0$ for all $t\in [0, 1]\cap \Rational$. To this end, fix a $t\in [0, 1]\cap \Rational$ and let $t=p/q$ for some $p, q\in \Natural$. To see this, note that it follows from the proof of~\cite[Lemma 14.16]{lovasz2012large} that $\delta_{2}([W_t], \omega_t)\leq \liminf_{n\to \infty} \delta_{2}([W_{np,nq}], \omega_t)=0$. Note that the hypothesis in~\cite[Lemma 14.16]{lovasz2012large} states that $[W_{np, nq}]\to [W_t]$ in cut-sense, but the proof only requires $W_{np, nq}\to W_t$ in weak-$*$ sense.
\end{proof}

\begin{corollary}\label{cor:metric_derivative_eq_enorm_Wt'}
    If $\omega\in\mathrm{AC}(\Graphons,\delta_2)$, then $\abs{\omega'}(t) = \enorm{W'_t}$ for a.e. $t\in\round{0,1}$, where $\round{W_t}_{t\in\interval{0,1}}\in\mathrm{AC}\round{\Wcal,d_2}$ is obtained as in Lemma~\ref{lem:representation_graphons}.
\end{corollary}
\begin{proof}
    Let $\omega$ and $(W_t)_{t\in[0,1]}$ be as above. Recall that $(W_t)_{t\in[0,1]}$ is an absolutely continuous curve in $(\Wcal, d_2)$. Since every absolutely continuous curve in a Hilbert space is differentiable a.e. (Radon--Nikod\'ym property)~\cite[page 30, Theorem 5]{HUFF19771}, there exists a family $W'_t \in L^2\big(\interval{0,1}^{(2)}\big)$, for a.e. $t\in\interval{0,1}$, such that $W_t-W_0=\int_0^t W'_s \diff s$ holds pointwise a.e. on $[0,1]^{(2)}$. It follows from Lebesgue differentiation theorem~\cite[Theorem 6.32]{hunter2014notes} that $\norm{2}{\frac{W_t-W_s}{t-s}-W'_t}\to 0$ as $s\to t$. We know from Lemma~\ref{lem:representation_graphons} that $\delta_2(\omega_t, \omega_s)= \norm{2}{W_t-W_s}$. Thus, it follows that $\abs{\omega'}(t)=\lim_{s\to t}\frac{\delta_2(\omega_t, \omega_s)}{\abs{t-s}}=\norm{2}{W'_t}$ for a.e. $t\in(0,1)$.
\end{proof}

\begin{lemma}\label{lem:cut_metric_equals_least_action}
The invariant $L^2$ metric between two graphons $\squarebrack{U},\squarebrack{V}\in\Graphons$ satisfies
\begin{align}
    \metric{\delta_2}{\squarebrack{U}}{\squarebrack{V}} &= \min \int_r^s \norm{2}{W_t'} \diff t\ ,
\end{align}
for any $0\leq r<s \leq 1$, where the minimum is taken over $(W_t)_{t\in\interval{r,s}}\in\mathrm{AC}\round{\Wcal,d_2}$ with domain $\interval{r,s}$ such that $W_r \in \squarebrack{U}$ and $W_s\in \squarebrack{V}$.
\end{lemma}
\begin{proof}
Let $\round{W_t}_{t\in\interval{r,s}}\subseteq\mathrm{AC}\round{\Wcal,d_2}$ be such that $W_r\in\squarebrack{U}$ and $W_s\in\squarebrack{V}$. Applying Jensen's inequality, we obtain
\begin{align}
    \int_r^s \enorm{W_t'} \diff t &\geq \enorm{\int_r^s W'_t \diff t} = \enorm{W_s - W_r} \ge  \metric{\delta_2}{\squarebrack{U}}{\squarebrack{V}}\ .\label{eq:lb_integral_W'}
\end{align}
\sloppy Following Definition~\ref{def:L2_metric}, there exists $\varphi_1, \varphi_2\in\Tcal$ such that
\begin{align}
    \metric{\delta_2}{\squarebrack{U}}{\squarebrack{V}} &= \enorm{U^{\varphi_1} - V^{\varphi_2}}.\label{eq:L2_UV_eps}
\end{align}
Therefore, we can define an curve $\round{W_{t}}_{t\in\interval{r,s}}\in\mathrm{AC}\round{\Wcal,d_2}$ as $W_{r} \coloneqq U^{\varphi_1}$, $W_{s} \coloneqq V^{\varphi_2}$ and $W_{t} \coloneqq \round{(s-t)W_{r} + (t-r)W_{s}}/(s-r)$ for $t\in(r,s)$.  Since for any $r\leq a < b \leq s$,
\begin{align}
    \enorm{W_{b}-W_{a}} &= \frac{\enorm{W_{s}-W_{r}}}{s-r} \cdot (b-a) = \frac{\enorm{U^{\varphi_1}-V^{\varphi_2}}}{s-r} \cdot (b-a),
\end{align}
therefore $\round{W_{t}}_{t\in\interval{r,s}}\in \mathrm{AC}\round{\Wcal,d_2}$ and $W'_{t} = \round{U^{\varphi_1}-V^{\varphi_2}}/(s-r)$ exists for all $t\in\round{r,s}$. With this choice of $\round{W_{t}}_{t\in\interval{r,s}}\in\mathrm{AC}\round{\Wcal,d_2}$, from equation~\eqref{eq:L2_UV_eps} we get
\begin{align}
    \int_r^s\enorm{W'_{t}}\diff t &= \enorm{U^{\varphi_1} - V^{\varphi_2}} = \metric{\delta_2}{\squarebrack{U}}{\squarebrack{V}}.\label{eq:int_eq_metric}
\end{align}
Combining equation~\eqref{eq:lb_integral_W'} and equation~\eqref{eq:int_eq_metric} completes the proof.
\end{proof}

As a consequence of above discussion,  we obtain that $(\Graphons, \delta_2)$ is a geodesic space. To the best of our knowledge it has not been recorded in the earlier literature.

\begin{proposition}\label{prop:graphon_geodesic_space}
    The space $(\Graphons,\delta_2)$ is a geodesic metric space.
\end{proposition}
\begin{proof}
    Recall that (see the remark after the Definition~\ref{def:length}) for any $\omega = \round{\omega_t}_{t\in[0,1]}\in\mathrm{AC}(\Graphons,\delta_2)$ such that $\omega_0=[U]$, and $\omega_1=[V]$, we have
    \begin{align}
       \ell(\omega) \ge \metric{\delta_2}{\omega_0}{\omega_1}= \metric{\delta_2}{[U]}{[V]}. \label{eq:min_geq_cutmetric}
    \end{align}
    Given $\squarebrack{U}, \squarebrack{V}\in\Graphons$, it suffices to construct a curve $\omega^* = \round{\omega^*_t}_{t\in[0,1]}\in\mathrm{AC}(\Graphons,\delta_2)$ such that $\omega^*_0=\squarebrack{U}$, $\omega^*_1=\squarebrack{V}$, and $\ell(\omega^*) \leq \metric{\delta_2}{[U]}{[V]}$. Without any loss of generality, we can choose $U,V\in\Wcal$ such that $\metric{\delta_2}{\squarebrack{U}}{\squarebrack{V}}=\nrm{2}{U-V}$. Define $\omega^*$ as $\omega^*_t \coloneqq \squarebrack{W_{t}}$ where $W_{t} \coloneqq (1-t)U + tV$ for all $t\in [0, 1]$. The curve $\omega^*\in\mathrm{AC}(\Graphons,\delta_2)$ since
    \begin{align}
        \delta_2([W_s], [W_r])\le \enorm{W_{s}-W_{r}} &=  \enorm{U-V}\cdot(s-r),
    \end{align}
    for all $0\leq r<s\leq 1$. Now observe that
    \begin{align}
        \ell(\omega^*) &= \sup\Set*{\sum_{k=0}^{n-1}\metric{\delta_2}{\squarebrack{W_{t_k}}}{\squarebrack{W_{t_{k+1}}}} \given n\in\Natural, 0=t_0<t_1\cdots < t_n = 1}\nonumber\\
        &\le \sup\Set*{\sum_{k=0}^{n-1}\enorm{U-V}(t_{k+1}-t_k)\diff t \given n\in\Natural, 0=t_0<t_1\cdots < t_n = 1}\nonumber\\
        &= \enorm{U-V}= \metric{\delta_2}{\omega^*_0}{\omega^*_1}.\label{eq:min_leq_cutmetric}
    \end{align}
    This completes the proof.
\end{proof}

Since $(\Graphons, \delta_2)$ is a geodesic metric space, the usual notions of geodesic convexity and semiconvexity makes sense in $(\Graphons, \delta_2)$.  We the subsequent sections we will need a notion of generalized geodesics( defined below) and show that generalized geodesics exist.

\begin{definition}[Generalized geodesics on $(\Graphons,\delta_2)$]\label{def:gen_geodesics}
    Let $\squarebrack{W_0},\squarebrack{W_1}\in\Graphons$. For every $\squarebrack{W}\in \Graphons$, one can construct an absolutely continuous curve $\vartheta$ (depending on $\squarebrack{W}$) as follows. From Lemma~\ref{lemma:MinimizingCoupling_gen}, we obtain $\varphi,\varphi_0,\varphi_1\in\Tcal$ such that 
    \begin{align}
        \delta_2\round{\squarebrack{W},\squarebrack{W_0}} = \enorm{W^\varphi - W_0^{\varphi_0}}\, ,\quad\text{and}\quad\delta_2\round{\squarebrack{W},\squarebrack{W_1}} = \enorm{W^\varphi - W_1^{\varphi_1}}.
    \end{align}
    Define the curve $\vartheta\coloneqq \round{\squarebrack{W_t}}_{t\in\interval{0,1}}$, where $W_t\coloneqq (1-t)W_0^{\varphi_0} + tW_1^{\varphi_1}$ for every $t\in\interval{0,1}$. This curve $\vartheta$ is called a {\em generalized geodesic} (\emph{with base} $\squarebrack{W}$) between the graphons $\squarebrack{W_0}$ and $\squarebrack{W_1}$ with respect to $\delta_2$. Often, when the base is clear from the context, we simply refer it as a generalized geodesic. From the construction, we can see that any geodesic between $\squarebrack{W_0},\squarebrack{W_1}\in\Graphons$ is also a generalized geodesic (with suitably chosen base) between them.
\end{definition}

Finally, we prove the following lemma that will be useful in the proof of Theorem~\ref{thm:GF_convergence}.

\begin{lemma}\label{lem:gen_geodesic_cvx_delta_2}
    If $\squarebrack{W},\squarebrack{W_0},\squarebrack{W_1}\in\Graphons$, then there exists $\vartheta = \round{\vartheta_t}_{t\in[0,1]}\in \mathrm{AC}(\Graphons,\delta_2)$ such that $\vartheta_0 = \squarebrack{W_0}$, $\vartheta_1 = \squarebrack{W_1}$, and $\delta_2^2\round{\squarebrack{W},\slot{}}/2$ is $1$-semiconvex over $\vartheta$ w.r.t. $\delta_2$.
\end{lemma}
\begin{proof}
    From Lemma~\ref{lemma:MinimizingCoupling_gen} we obtain $\varphi,\varphi_0,\varphi_1\in\Tcal$ such that
    \begin{align}
        \delta_2\round{\squarebrack{W},\squarebrack{W_0}} = \enorm{W^\varphi - W_0^{\varphi_0}}, \quad\text{and}\quad \delta_2\round{\squarebrack{W},\squarebrack{W_1}} = \enorm{W^\varphi - W_1^{\varphi_1}}.\label{eq:chain_eq}
    \end{align}
    Defining $W_t \coloneqq (1-t)W_0^{\varphi_0} + t W_1^{\varphi_1}$ and $\vartheta_t \coloneqq [W_t]$ for $t\in[0,1]$, we get that that
    \begin{align*}
        \delta_2^2\round{\squarebrack{W},\squarebrack{W_t}} &\leq \enorm{W^\varphi - (1-t)W_0^{\varphi_0} - tW_1^{ \varphi_1}}^2\nonumber\\
        &= (1-t)\enorm{W^\varphi - W_0^{\varphi_0}}^2 + t\enorm{W^\varphi - W_1^{\varphi_1}}^2 - t(1-t)\enorm{W_0^{\varphi_0} - W_1^{\varphi_1}}^2\nonumber\\
        &= (1-t)\delta_2^2\round{\squarebrack{W},\squarebrack{W_0}} + t\delta_2^2\round{\squarebrack{W},\squarebrack{W_1}} - t(1-t)\enorm{W_0^{\varphi_0} - W_1^{\varphi_1}}^2\nonumber\\
        &\leq (1-t)\delta_2^2\round{\squarebrack{W},\squarebrack{W_0}} + t\delta_2^2\round{\squarebrack{W},\squarebrack{W_1}} - t(1-t)\delta_2^2\round{\squarebrack{W_0},\squarebrack{W_1}}.
    \end{align*}
    Therefore, $\delta_2^2\round{\squarebrack{W},\slot{}}/2$ is $1$-semiconvex along $\vartheta$ w.r.t. $\delta_2$.\qedhere
\end{proof}

\section{Gradient Flows on Graphons}\label{sec:grad_flows}
The goal of this section is to prove Theorem~\ref{thm:GF_convergence}. Section~\ref{sec:implicit_euler} is mostly for the completeness. In Section~\ref{sec:implicit_euler} we define an implicit Euler scheme following~\cite{ambrosio2005gradient} that allows one to prove the existence of gradient flow in Theorem~\ref{thm:existence}. However, the conditions in Theorem~\ref{thm:existence} are hard to verify. We do not use this existence result. Later we prove an alternate existence theorem for gradient flow. However, the main highlight of Section~\ref{sec:implicit_euler} is Proposition~\ref{prop:gamma_convergence} where we prove a $\Gamma$-convergence result for Euler iterates on finite dimensional matrices (or step-kernels). This result is crucial in the proof of our main theorem and we believe that it would be of independent interest as well.

In Section~\ref{sec:frechet_derivatives} we introduce a notion of `Fr\'echet-like derivative' which in turn allows us to prove the existence of gradient flow for function $F$ that have Fr\'echet-like derivative. We do verify this condition for a large class of functions in Section~\ref{sec:examples_GF}. As pointed out in the introduction, the existence of a Fr\'echet-like derivative not only allows us to prove the existence of gradient flow, but it also allows a kernel representation of the gradient flow (see Theorem~\ref{thm:Existence_with_FrehetDerivative}) which is extremely useful in practice. This section is the most technical part of the paper and has many results of independent interest. 

 In Section~\ref{sec:convergence} we give the proof of main result Theorem~\ref{thm:GF_convergence}. This section can be read directly after Section~\ref{sec:prelims} and Proposition~\ref{prop:gamma_convergence}. 

Finally, Section~\ref{sec:continuity_eq} complements the discussion by proving that the gradient flows in $(\Graphons, \delta_2)$ can also be described by a family of continuity equations. This plays an important role in~\cite{HOPST22}.

\subsection{Implicit Euler method, Generalized Minimizing Movements}\label{sec:implicit_euler}
Here we introduce an implicit Euler scheme and use to show the existence of gradient flow. We begin with the setup and definitions. Given $F\colon\Graphons\to \mathbb{R}\cup\{\infty\}$, a step size $\tau > 0$ and $\squarebrack{U}\in\Graphons$, define a functional $\Phi_F\round{\tau,\squarebrack{U};\slot{}}\colon \Graphons \to \rr\cup\Set*{\infty}$, called \emph{penalized functional}, given by
\begin{align}
    \Phi_F(\tau,\squarebrack{U};\squarebrack{V}) &\coloneqq F\round{\squarebrack{V}} + \inv{2\tau}\delta_2^2\round{\squarebrack{V},\squarebrack{U}}\, ,\label{eq:Phi_def}
\end{align}
and a set-valued {\em resolvent operator} $J_\tau$ on $\Graphons$ as
\begin{align}
    J_\tau\round{\squarebrack{U}} \coloneqq \argmin_{\Graphons}\Phi_F\round{\tau,\squarebrack{U};\slot{}},\quad \text{for $\squarebrack{U}\in\Graphons$.} \label{eq:resolvent_op}
\end{align}

For a sequence $\tauvec \coloneqq \round{\tau_n}_{n\in\Natural}$ of positive time steps with $\abs{\tauvec} \coloneqq \sup_{n\in\Natural}\tau_n < \infty$, we can associate a partition of the time interval $\round{0,\infty}$ as
\[
    P_\tauvec \coloneqq \Set*{I^n_\tauvec \coloneqq (t^{n-1}_\tauvec,t^n_\tauvec]}_{n\in\Natural}, \qquad \tau_n = t^n_\tauvec- t^{n-1}_\tauvec,
\]
if $t^0_\tauvec =0$ and $\lim_{n\to\infty}t^n_\tauvec = \infty$. Given such a sequence $\tauvec$ and $\squarebrack{U_{\tauvec, 0}}\in \Graphons$, we can obtain a sequence $\round{[U_{\tauvec, n}]}_{n\in\Natural}$ by iteratively solving for $[U_{\tauvec, n}]$, by setting
\begin{align}\label{eqn:recursiveSol}
    [U_{\tauvec, n}]\in J_{\tau_n}\round{\squarebrack{U_{\tauvec, n-1}}},
\end{align}
provided $J_{\tau_n}\round{\squarebrack{U_{\tauvec, n-1}}}$ is non-empty for every $n\in\Natural$. Note that the iterates $([U_{\tauvec, n}])_{n\in\Natural}$ tries to minimize the function $F$ at each step, but it is penalized against taking big jumps. In practice, therefore, it makes sense to treat the curve obtained by joining these iterates as a proxy for gradient flow. We need the following definition to makes this idea precise.

\begin{definition}[Discrete solution]\label{def:discrete_sol}
    Given the sequence $\round{\squarebrack{U_{\tauvec,n}}}_{n\in\Natural}$ as above in equation~\eqref{eqn:recursiveSol}, interpolate the discrete points by a piece-wise constant left-continuous function $\overline{\squarebrack{U_\tauvec}}\colon[0,\infty)\to\Graphons$, defined as
    \begin{align}
        \overline{\squarebrack{U_\tauvec}}(0) \coloneqq \squarebrack{U_{\tauvec,0}}\, ,\qquad \overline{\squarebrack{U_\tauvec}}(t) \coloneqq \squarebrack{U_{\tauvec,n}}\, ,\quad t\in (t_{n-1},t_n].
    \end{align}
    We call $\overline{\squarebrack{U_\tauvec}}$ to be a {\em discrete solution} corresponding to the partition $P_\tauvec$.
\end{definition}

Discrete solutions are simply a way of creating a curve from the iterates of resolvent operator. One should expect that as one take $\abs{\tauvec_n}\to 0$, the discrete solutions yield a curve that is often a good candidate for gradient flow (a.k.a. curves of maximal slope) in an arbitrary metric space setting. Such curves are called generalized minimizing movements that we define below.

\begin{definition}[Generalized minimizing movements]\label{def:GMM}
    For a function $F$, its corresponding functional $\Phi_F$ as defined in equation~\eqref{eq:Phi_def}, and an initial datum $\squarebrack{U_0}\in\Graphons$, we say that a curve  $\omega = \round{\omega_t}_{t\in\rr_+}$ in $\Graphons$ is a {\em generalized minimizing movement} (GMM) for $\Phi_F$ starting from $\squarebrack{U_0}\in\Graphons$ if there exists a sequence of sequences $\round{\tauvec_k}_{k\in \Natural}$ with $\lim_{k\to\infty}\abs{\tauvec_k} = 0$ and a corresponding sequence of discrete solutions $(\overline{\squarebrack{U_{\tauvec_k}}})_{k\in\Natural}$ defined as in Definition~\ref{def:discrete_sol} such that for all $t\in\rr_+$,
    \begin{align}
    \begin{split}
        \lim_{k\to\infty} F\round{\squarebrack{U_{\tauvec_k,0}}} &= F\round{\squarebrack{U_0}}, \quad \limsup_{k\to\infty} \delta_2\round{\squarebrack{U_{\tauvec_k,0}},\squarebrack{U_0}} < \infty,\\   &\toplim{\delta_\cut}_{k\to\infty}\overline{\squarebrack{U_{\tauvec_k}}}(t) = \omega_t.
    \end{split}
    \label{eq:GMM_conditions}
    \end{align}
\end{definition}
There is a related definition of \textit{minimizing movement} (MM) curves that can be found in~\cite[Definition 2.0.6]{ambrosio2005gradient} where the conditions in equation~\eqref{eq:GMM_conditions} need to hold for all sequences of partitions with vanishing norm.
The set of all minimizing movements and generalized minimizing movements on the metric space $(\Graphons,\delta_2)$ with respect to the metric $\delta_\cut$ starting from $\squarebrack{U_0}\in\effdom(F)$ are denoted by $\mathrm{MM}_{\delta_2,\delta_\cut}\round{\Phi_F,\squarebrack{U_0}}$ and $\mathrm{GMM}_{\delta_2,\delta_\cut}\round{\Phi_F,\squarebrack{U_0}}$ respectively. From their definitions it can be verified that the set of minimizing movements is contained in the set of generalized minimizing movements. See~\cite[Definition 2.0.6]{ambrosio2005gradient} for the precise difference between them.
Since $(\Graphons,\delta_2)$ is a bounded metric space, the second conditions in equation~\eqref{eq:GMM_conditions} and~\cite[equation 2.0.10]{ambrosio2005gradient} are trivially satisfied. 

\begin{lemma}\label{lem:F_assumptions}
    If $F\colon\Graphons\to\rr\cup\Set*{\infty}$ is sequentially $\delta_\cut$-lower semicontinuous, then
    \begin{enumerate}
        \item for every $\tau>0$ and $\squarebrack{U}\in\Graphons$, we have $\inf_{\Graphons}\Phi_F\round{\tau,\squarebrack{U};\slot{}}>-\infty$, where $\Phi_F$ is defined in equation~\eqref{eq:Phi_def}, and
        \item if $\round{\squarebrack{U_n}}_{n\in\Natural}\subset\Graphons$ with $\sup_{n\in\Natural}F\round{\squarebrack{U_n}} < \infty$, then $\round{\squarebrack{U_n}}_{n\in\Natural}$ admits a $\delta_\cut$-converging subsequence.
    \end{enumerate}
\end{lemma}
\begin{proof}
    Since $(\Graphons,\delta_\cut)$ is a compact metric space~\cite{Lovsz2007SzemerdisLF}, from the Weierstrass Theorem~\cite[Box 1.1]{santambrogio2015optimal}, both $\argmin_{\Graphons}F$ and $\argmin_{\Graphons}\Phi_F\round{\tau,\squarebrack{U};\slot{}}$ exist for all $\tau>0$ and $\squarebrack{U}\in\Graphons$. Thus the minima are greater than $-\infty$, and every sequence admits a $\delta_\cut$-converging subsequence.
\end{proof}

From Lemma~\ref{lem:delta_2_lsc} we know that the topology induced by $\delta_2$ is sequentially $\delta_{\cut}$-lower semicontinuous. This with Lemma~\ref{lem:F_assumptions} shows that the assumptions in \cite[Proposition 2.2.3]{ambrosio2005gradient} are satisfied, guaranteeing that $\mathrm{GMM}_{\delta_2,\delta_\cut}\round{\Phi_F,\squarebrack{U_0}}$ is non-empty. If $\abs{\pdiff F}$ is $\delta_\cut$-lower semicontinuous and $F$ is $\delta_{\cut}$-continuous on the sublevel sets of $\abs{\pdiff F}$, then it follows from~\cite[Theorem 2.3.1]{ambrosio2005gradient} that every element $\omega\in \mathrm{GMM}_{\delta_2,\delta_\cut}\round{\Phi_F,\squarebrack{U_0}}$, for $[U_0]\in \effdom(F)$, is a curve of maximal slope. For the sake of clarity, we record the above discussion as a theorem.

\begin{theorem}[Existence of curves of maximal slope-I]\label{thm:existence}
    Suppose $F\colon \Graphons\to\rr\cup\Set*\infty$ satisfies the following conditions.
    \begin{enumerate}
    \item $F$ is $\delta_\cut$-lower semicontinuous on $\effdom(F)$. 
    \item Its local slope $\abs{\pdiff F}$ is $\delta_\cut$-lower semicontinuous in $\effdom(F)$.  
    \item $F$ is $\delta_{\cut}$-continuous on the sublevel sets of $\abs{\pdiff F}$.
    \end{enumerate}
    Then every curve $\omega\in \mathrm{GMM}_{\delta_2,\delta_\cut}\round{\Phi_F,\squarebrack{U_0}}$ for $\squarebrack{U_0} \in \effdom(F)$ is a curve of maximal slope.
\end{theorem}

In practice, it is difficult to compute $\abs{\pdiff F}$ or to ascertain its $\delta_{\cut}$-lower semicontinuity. This makes it difficult to apply Theorem~\ref{thm:existence} on natural examples. Later in  Theorem~\ref{thm:Existence_with_FrehetDerivative} we show that, when $f$ admits a Fr\'echet-like derivative that is $\lambda$-semiconvex on $(\Wcal,d_2)$ for some $\lambda\in\rr$, the existence of a curve of maximal slope follows without requiring $\delta_{\cut}$-lower semicontinuity of $\abs{\pdiff F}$.

\subsubsection{\texorpdfstring{$\Gamma$-convergence of penalized functional}{Gamma convergence of penalized functional}}
Recall that the goal of this paper is to show that the Euclidean gradient flows on matrices converge in suitable sense to gradient flow on graphons. In the previous section, we establish that the gradient flows in very general settings can be obtained as the limits of discrete solutions. In this section, we show that iterates of $J_{\tau}\vert_{\Graphons_k}$ converge in suitable sense to the iterates of $J_{\tau}\vert_{\Graphons}$ as $k\to \infty$.

More formally, for any $k\in\Natural$, define $J_{\tau}^{(k)}$ to be the resolvent operator on $\Graphons_k$ as
\begin{align}
    J_{\tau}^{(k)}\round{\squarebrack{U}} &\coloneqq \argmin_{\Graphons_k}\Phi_F(\tau, [U];\slot{})=\argmin_{\Graphons_k}\Set*{ F + \inv{2\tau}\delta^2_2\round{\squarebrack{U},\slot{}}},
\end{align}
for any $\tau>0$, $[U]\in \Graphons_k$. 

The following Lemma essentially shows $\Gamma$-convergence of the penalized functionals $\Phi_F$, restricted to $\Graphons_k$, as $k\rightarrow \infty$.

\begin{proposition}\label{prop:gamma_convergence}
    \sloppy Fix some $\delta_\cut$-continuous function $F\colon\Graphons\to\rr\cup\Set*{\infty}$ and some step size $\tau>0$.
    Consider a sequence $(\squarebrack{U_k}\in\Graphons_k)_{k\in\Natural}$ such that $\round{\squarebrack{U_k}}_{k\in\Natural}\xrightarrow{\delta_\cut}\squarebrack{U}$ as $k\to\infty$ for some $\squarebrack{U}\in\Graphons$. For each $k\in\Natural$, let $\squarebrack{U^+_{k, \tau}}\in \argmin_{\Graphons_k}\Phi_F\round{\tau,\squarebrack{U_k};\slot{}}$. Suppose $\squarebrack{U_{\infty,\tau}^+}$ is any $\delta_\cut$-limit point of the sequence $\round{\squarebrack{U_{k,\tau}^+}}_{k\in\Natural}$. Then $\squarebrack{U^+_{\infty,\tau}}\in \argmin_{\Graphons}\Phi_F\round{\tau,\squarebrack{U};\slot{}}$.
\end{proposition}
\begin{proof}
Note that for any sequence of graphons $\round{\squarebrack{W_k}}_{k\in\Natural}$ such that $\round{\squarebrack{W_k}}_{k\in\Natural}\xrightarrow{\delta_\cut}\squarebrack{W}$ for some $\squarebrack{W}\in\Graphons$, by Lemma~\ref{lem:delta_2_lsc} we have
\begin{align}
    \liminf_{k\to\infty}\delta_2\round{\squarebrack{U_k},\squarebrack{W_k}} \geq \delta_2\round{\squarebrack{U},\squarebrack{W}}.\label{eq:lsc_UW}
\end{align}
We now construct a recovery sequence of graphons $(\squarebrack{W^*_k}\in\Graphons_k)_{k\in\Natural}\subset\Graphons$ such that
\begin{align}
    \lim_{k\to\infty}\delta_2\round{\squarebrack{U_k},\squarebrack{W^*_k}} = \delta_2\round{\squarebrack{U},\squarebrack{W}}\, , \qquad \text{and} \qquad \lim_{k\to\infty}\cutmetric{\squarebrack{W^*_k}}{\squarebrack{W}} = 0 .\label{eq:recovery_seq}
\end{align}
To do so, we first obtain $\varphi,\psi\in\Tcal$ from Definition~\ref{def:L2_metric} and~\cite[Theorem 6.16]{janson2010graphons} such that
\begin{align}
    \delta_2\round{\squarebrack{U},\squarebrack{W}} &= \enorm{U^\varphi - W^\psi}.\label{eq:UW_coupled}
\end{align}
Since $\delta_\cut\round{\squarebrack{U_k},\squarebrack{U}}\to 0$ as $k\to\infty$, using~\cite[Theorem 11.59]{lovasz2012large} we can find $\round{\varphi_k\in\Ical_k}_{k\in\Natural}$ such that 
\begin{align}
    \lim_{k\to\infty}\cutnorm{U_k^{\varphi_k} - U^\varphi} = 0.
    \label{eq:UUk_coupled}
\end{align}
We now define a sequence of kernels $\round{Z_k\in\Wcal_k}_{k\in\Natural}$ as
\[
    Z_k \coloneqq U_k^{\varphi_k} - \E{U^\varphi\given \Fcal_k},
\]
where $\Fcal_k = \sigma\Set*{Q_k\times Q_k}$ for every $k\in\Natural$. 
Note that
\begin{align*}
    \cutnorm{Z_k} &\leq \cutnorm{U_k^{\varphi_k} - U^\varphi} + \cutnorm{U^\varphi - \E{U^\varphi\given \Fcal_k}}\\
    &\leq \cutnorm{U_k^{\varphi_k} - U^\varphi} + \enorm{U^\varphi - \E{U^\varphi\given \Fcal_k}}.
\end{align*}
Also note that for any $V\in\Wcal$, the martingale sequence $\round{\E{V\given \Fcal_k}\in\Wcal_k}_{k\in\Natural}$ converges to $V\in\Wcal$ in $L^2\big(\interval{0,1}^{(2)}\big)$ as $k\to\infty$. Using $L^2$ convergence of the martingales $\E{U^\varphi\given \Fcal_k}$ and equation~\eqref{eq:UUk_coupled} we conclude that 
\begin{align}
    \cutnorm{Z_k}\to 0, \qquad \text{as}\quad k\to\infty.\label{eq:Z_k_to_0}
\end{align}
The sequence of kernels $\round{W_k^*\in\Wcal_k}_{k\in\Natural}$ can now be defined as
\[
    W_k^* \coloneqq \E{W^\psi\given \Fcal_k} + Z_k.
\]
It now follows that for any $k\in\Natural$,
\begin{align*}
    \cutnorm{W_k^* - W^\psi} &\leq \cutnorm{\E{W^\psi\given \Fcal_k} - W^\psi} + \cutnorm{Z_k}\\
    &\leq \enorm{\E{W^\psi\given \Fcal_k} - W^\psi} + \cutnorm{Z_k}.
\end{align*}
Using $L^2$ convergence of the martingales, and equation~\eqref{eq:Z_k_to_0} we obtain $\nrm{\cut}{\round{W_k^*}^{\psi_k} - W^\psi} \to 0$ and therefore have
\begin{align}
    \limsup_{k\to\infty}\cutmetric{\squarebrack{W_k^*}}{\squarebrack{W}} &= 0.\label{eq:Wk*_W}
\end{align}
Moreover, 
\begin{align}
    \enorm{U_k^{\varphi_k} - W_k^*}^2 &= \enorm{\E{U^\varphi\given \Fcal_k} - \E{W^\psi\given \Fcal_k}}^2\nonumber\\
    &\leq \enorm{U^\varphi - W^\psi}^2 = \delta_2^2\round{\squarebrack{U},\squarebrack{W}} && \text{(using equation~\eqref{eq:UW_coupled})},\label{eq:limsup_ineq}
\end{align}
where the last inequality follows from~\cite[Equation 9.7]{lovasz2012large}. From equation~\eqref{eq:limsup_ineq} and Lemma~\ref{lem:delta_2_lsc} we obtain
\begin{align}
    \lim_{k\to\infty}\delta_2\round{\squarebrack{U_k},\squarebrack{W^*_k}} &= \delta_2\round{\squarebrack{U},\squarebrack{W}}.\label{eq:Uk_Wk*}
\end{align}
 Now, by the definition of $U^+_{k,\tau}$, we have
\begin{align}
    F\round{\squarebrack{U^+_{k,\tau}}} + \inv{2\tau}\delta_2^2\round{\squarebrack{U_k},\squarebrack{U^+_{k,\tau}}} &\leq F\round{\squarebrack{W_k^*}} + \inv{2\tau}\delta_2^2\round{\squarebrack{U_k},\squarebrack{W^*_k}}.\label{eq:U_ktau_min}
\end{align}
Taking $\liminf_{k\to\infty}$ on both sides of equation~\eqref{eq:U_ktau_min}, and from equation~\eqref{eq:lsc_UW}, equation~\eqref{eq:recovery_seq} and the $\delta_\cut$-continuity of $F$, we get
\begin{align}
    &F\round{\squarebrack{U^+_{\infty,\tau}}} + \inv{2\tau}\delta_2^2\round{\squarebrack{U},\squarebrack{U^+_{\infty,\tau}}}\nonumber\\
    &\leq \liminf_{k\to\infty}F\round{\squarebrack{U^+_{k,\tau}}} + \liminf_{k\to\infty}\inv{2\tau}\delta_2^2\round{\squarebrack{U_k},\squarebrack{U^+_{k,\tau}}}\nonumber\\
    &\leq \liminf_{k\to\infty}F\round{\squarebrack{W_k^*}} + \liminf_{k\to\infty}\inv{2\tau}\delta_2^2\round{\squarebrack{U_k},\squarebrack{W^*_k}}
    = F\round{\squarebrack{W}} + \inv{2\tau}\delta_2^2\round{\squarebrack{U},\squarebrack{W}}.
\end{align}
Since $[W]\in\Graphons$ was arbitrary, this completes the proof.
\end{proof}

\subsection{Fr\'echet-like derivatives and Existence of gradient flow}
In Section~\ref{sec:frechet_derivatives} we introduce the notion of Fr\'echet-like differentiability. The most important result in Section~\ref{sec:frechet_derivatives} is Lemma~\ref{lem:upper_gradient_frechet} which relates the Fr\'echet-like derivative with the local slope of the function. This plays a crucial role in Section~\ref{sec:Existence_with_FD} where we show the existence of gradient flow in Theorem~\ref{thm:Existence_with_FrehetDerivative}. 
\subsubsection{Fr\'echet-like derivatives and local slope}\label{sec:frechet_derivatives}
Recall that given a function $F\colon\Graphons\to \mathbb{R}\cup\{\infty\}$, we can define an invariant function $f\colon\Wcal\to \mathbb{R}\cup\{\infty\}$ such that $f=F\circ \squarebrack{\slot{}}$.

\begin{definition}[Fr\'echet-like derivative on $\Wcal$]\label{def:frechet_like_derivative}
    \sloppy Suppose $f\colon\Wcal\to \rr\cup\Set*\infty$ is an invariant function. Let $V\in\effdom(f)$.
    The {\em Fr\'echet-like derivative} at $V$ is given by any $\phi \in L^\infty\big(\interval{0,1}^{(2)}\big)$ that satisfies the following condition,
    \begin{align}
      \lim_{W\in\Wcal,\,\norm{2}{W-V}\to 0}\frac{f(W) - f(V) - \round{ \inner{\phi,W}- \inner{\phi,V}}}{\norm{2}{W-V}} &= 0,\label{eq:frechet_limit_def}
    \end{align}
    where $\inner{\slot{},\slot{}}$ is the usual inner product on $L^2\big(\interval{0,1}^{(2)}\big)$. If $f$ admits a Fr\'echet-like derivative at every $V\in \effdom(f)$, we denote the map that takes $V$ to the corresponding $\phi$ by $D_\Wcal f$.
    In that case we say that $f$ is Fr\'echet differentiable. 
\end{definition}

In~\cite{DIAO2015183}, the authors consider G\^{a}teuax and Fr\'echet derivatives of functions on graphons with respect to the cut metric. However, as they remark~\cite[Remark 2.18, page 195]{DIAO2015183}, such a notion of Fr\'echet derivative is too weak to cover natural functions such as homomorphism densities.

\sloppy The next lemma shows that Fr\'echet-like derivatives behave nicely under the Lebesgue measure-preserving transforms and hence is a well-defined map from $\Graphons$ to $\widehat{L}^\infty(\interval{0,1}^2)$. 
That is, we can project $D_\Wcal f$ to obtain $D_{\Graphons} F \colon  \effdom(F) \to \widehat{L}^\infty(\interval{0,1}^2)$ as $D_{\Graphons} F \round{\squarebrack{V}} \coloneqq \squarebrack{D_\Wcal f(V)}$ for $V\in\Wcal$.

\begin{lemma}\label{lem:frechet_consistency}
    Let $f\colon\Wcal\to\rr\cup\Set*\infty$ be an invariant function. Let $V,V'\in \effdom(f)$ such that $V' = V^\varphi$ for some $\varphi\in\Tcal$. Suppose that the Fr\'echet-like derivatives $D_{\Wcal}f(V)$ and $D_{\Wcal}f(V')$ exist. If $\phi = D_\Wcal f(V)$ and $\phi' = D_\Wcal f(V')$, then $\phi'=\phi^\varphi$ a.e. In particular, this implies that $D_\Wcal f(V)\in L^\infty\big(\interval{0,1}^{(2)}\big)$ if it exists, is unique.
\end{lemma}
\begin{proof}
    Let the sequence $\round{V_n}_{n\in\Natural}\subset \Wcal$ be such that $\enorm{V_n-V}\to 0$ as $n\to\infty$, then we have $\nrm{2}{V_n^{\varphi}-V'}\to 0$ as $n\to\infty$. We first show that 
    \begin{align}\label{eq:phi2_equals_phi1varphi}
        \lim_{n\to \infty} \frac{\inner{\phi' - \phi^\varphi,V_n^\varphi-V'}}{\norm{2}{V_n - V}}=0\;.
    \end{align}
    To this end, recall that $f$ is invariant and hence $f(V)=f(V^{\varphi})$ and $f(V_n^{\varphi})=f(V_n)$. Therefore, we have
    \begin{align}
        &\lim_{n\to \infty} \frac{\inner{\phi' - \phi^\varphi,V_n^\varphi}-\inner{\phi' - \phi^\varphi,V^\varphi}}{\norm{2}{V_n - V}}\nonumber\\
        &= \lim_{n\to \infty}\left[ \frac{f(V_n) - f(V) - \inner{\phi,V_n-V}}{\norm{2}{V_n-V}}
        - \frac{f(V_n^\varphi) - f(V^\varphi) - \inner{\phi', V_n^\varphi - V^\varphi}}{\norm{2}{V_n^\varphi-V^\varphi}}\right]\nonumber\\
        &= \lim_{n\to \infty} \frac{f(V_n) - f(V) - \inner{\phi, V_n- V}}{\norm{2}{V_n-V}} - \lim\limits_{n\to \infty} \frac{f(V_n^\varphi) - f(V^\varphi) - \inner{\phi', V_n^\varphi- V^\varphi}}{\norm{2}{V_n^\varphi-V^\varphi}}=0,\nonumber
    \end{align}
    where the last equality holds because each limit individually goes to $0$ by the definition of Fr\'echet differentiability and our assumption that $f$ has Fr\'echet-like derivative at $V$ and $V'$.

We now show that $\phi' - \phi^\varphi=0$ a.e. Let $A^+ \coloneqq \setinline{\phi' - \phi^\varphi > 0}$ and $A^- \coloneqq \setinline{\phi' - \phi^\varphi<0}$. It suffices to show that $\abs{A^{+}}+\abs{A^{-}}=0$. We only prove that $A^{+}$ has measure $0$, the proof for $A^{-}$ follows similarly. Let $A\coloneqq\setinline{V=1}\cap A^{+}$ and $B\coloneqq \setinline{V<1}\cap A^{+}$. We claim that both $A$ and $B$ have measure $0$. We prove this by contradiction. Suppose, for contradiction, that $\abs{B}>0$. Define the set $B^{\varphi}\coloneqq \setinline{(x, y) \in [0,1]^2\given (\varphi(x), \varphi(y))\in B}$ and note that $\abs{B}=\abs{B^{\varphi}}$ and hence $\abs{B^{\varphi}}$ has positive measure. Set $V_n\coloneqq V+\frac{1}{n}\chi_{B^{\varphi}}$ and note that $\enorm{V_n-V}=\frac{\abs{B}}{n}\to 0$ as $n\to \infty$. By equation~\eqref{eq:phi2_equals_phi1varphi} we conclude that
\begin{align*}
    0 = \inner{\phi' - \phi^\varphi, \chi_{B^{\varphi}}} = \int_{B} \round{\phi' - \phi^\varphi}(x,y)\diff x\diff y >0,
\end{align*}
which is a contradiction. Therefore, we must have that $B$ has $0$ measure. Repeating the same argument with $V_n\coloneqq V-\frac{1}{n}\chi_{A^{\varphi}}$ shows that $A$ has measure $0$. Since $A^{+}=A\cup B$, it follows that $A^{+}$ has measure zero. 

To conclude the second part, suppose that $\phi$ and $\phi'$ are two Fr\'echet-like derivatives of $f$ at $V$. Then, (taking $\varphi = \id$) we must have that $\phi = \phi'$ a.e. Hence, $D_{\Wcal}f(V)$ is a unique element in $L^\infty\big(\interval{0,1}^{(2)}\big)$.
\end{proof}

Let $F\colon\Graphons\to\mathbb{R}\cup\{\infty\}$ and let $f\colon\Wcal\to \mathbb{R}\cup\{\infty\}$ be the invariant extension of $F$. Lemma~\ref{lem:frechet_consistency} justifies saying $F$ has Fr\'echet-like derivative if $f$ has a Fr\'echet-like derivative. Note that the Lemma~\ref{lem:frechet_consistency} says that not only can the Fr\'echet-like derivative be thought of as a graphon, but also the two graphons $\squarebrack{D_\Wcal f(V)}$ and $[V]$ are `coupled' in the sense that they are two sets of ``edge weights'' associated with the edges of the same exchangeable continuum ``graph''. We make a formal definition to capture this relationship.

\begin{definition}[Coupled graphons]\label{def:coupled_graphons}
    For any $r\in\Natural$, we define the set $\squarebrack{W_1}\odot \squarebrack{W_2} \odot \cdots \odot \squarebrack{W_r}\subseteq \Graphons^r$ with initial labeling $\round{W_1, W_2,\ldots,W_r}\in\Wcal^r$ as
    \begin{align}
        \mathop{\odot}_{i=1}^r \squarebrack{W_i} &\coloneqq \Set*{\round{W_i^\varphi}_{i=1}^r \given \varphi\in\Tcal}.
    \end{align}
\end{definition}
Without loss of generality, we will always refer to the elements in $\mathop{\odot}_{i=1}^r \squarebrack{W_i}$ with the initial labeling $\round{W_i}_{i=1}^r$ unless specified. Since we can also relabel elements in $L^\infty\big(\interval{0,1}^{(2)}\big)$ (i.e., apply the map $V\mapsto V^\varphi$, for $V\in L^\infty\big(\interval{0,1}^{(2)}\big)$ and $\varphi\in\Tcal$), we can generalize Definition~\ref{def:coupled_graphons} to tuples with elements in $L^\infty\big(\interval{0,1}^{(2)}\big)\supset\Wcal$. That is, we can consider sets of the form
\begin{align}
    \mathop{\odot}_{i=1}^r \squarebrack{V_i} &\coloneqq \Set*{\round{V_i^\varphi}_{i=1}^r \given \varphi\in\Tcal},
\end{align}
with initial labeling $\big(V_i\in L^\infty\big(\interval{0,1}^{(2)}\big)\big)_{i=1}^r$. Therefore, from Lemma~\ref{lem:frechet_consistency}, if $V\in\squarebrack{V}\in\Graphons$, and $\phi = D_{\Wcal}f(V)$, then $\round{V,\phi}\in \squarebrack{V}\odot\squarebrack{\phi}$.

For $\round{V,\phi}\in \squarebrack{V}\odot\squarebrack{\phi}$, we define the set $G_V\subseteq \interval{0,1}^2$ as 
\begin{align}
    G_V \coloneqq \setinline{\abs{V}<1}\cup \setinline{V=1, \phi>0}\cup \setinline{V=-1,\phi <0}.\label{eq:G_V}
\end{align}
Since $(V,\phi)\in \squarebrack{V}\odot\squarebrack{\phi}$, the set $G_V$ is well defined on $\Graphons$. For any $\varphi\in\Tcal$, $G_{V^\varphi} = \round{G_V}^{\varphi} \coloneqq \setinline{(\varphi(x),\varphi(y))\in\interval{0,1}^2\given (x,y)\in G_V}$.

The next lemma gives an expression for the local slope of $F$ in terms of its Fr\'echet-like derivative. 
\begin{lemma}\label{lem:upper_gradient_frechet}
Let $F\colon\Graphons\to\rr\cup\Set*\infty$ be a function and $f\colon\Wcal\to\rr\cup\Set*\infty$ its invariant extension. Assume that for each $\squarebrack{V}\in \Graphons$ the Fr\'echet-like derivative $D_\Wcal f(V)$ exists for all $V\in\squarebrack{V}$, then the local slope (Definition~\ref{def:local_slope}) of $F$ at $\squarebrack{V}$ satisfies
\begin{align}\label{eqn:localSlope_FrechetDerivative}
    \abs{\pdiff F}\round{\squarebrack{V}} = \eta_F\round{\squarebrack{V}} &\coloneqq \sup_{W\in\Wcal}\frac{\round{\inner{\phi, V} - \inner{\phi, W} }^+ }{\norm{2}{V-W}}=\enorm{\phi\indicator{G_V}{}},
\end{align}
where $V\in\squarebrack{V}$, and $\phi = D_{\Wcal}f(V)$. In particular, $\abs{\pdiff F}\round{\squarebrack{V}}=\enorm{\phi}$ if $V\in \Set*{U\in \Wcal\given \abs{U}<1 \text{ a.e.}}\cap \effdom(f)$.
\end{lemma}
\begin{proof}
Fixing $\squarebrack{V}\in \effdom(F)$, we verify using Lemma~\ref{lem:frechet_consistency} that $\eta_F$ is well defined on $\Graphons$. If $V_2 = V_1^\varphi$ for $V_1\in\squarebrack{V}$ for some $\varphi\in\Tcal$, and $\phi_1 = D_\Wcal f(V_1)$ then $D_\Wcal f(V_2) = \phi_1^\varphi \eqqcolon \phi_2$, and
\begin{align}
    \begin{split}
        \sup_{W\in\Wcal}\frac{\round{\inner{\phi_1, V_1- W} }^+ }{\norm{2}{V_1-W}} &= \sup_{W\in\Wcal}\frac{\round{\inner{\phi_1^\varphi, V_1^\varphi} - \inner{\phi_1^\varphi, W^\varphi}  }^+}{\norm{2}{V_1^\varphi - W^\varphi}}\\
        &= \sup_{W\in\Wcal}\frac{\round{\inner{\phi_2, V_2} - \inner{\phi_2, W} }^+ }{\norm{2}{W-V_2}}.
    \end{split}
\end{align}
We will now break the proof of the claim into two parts:
    \begin{enumerate}
        \item For any $\varepsilon>0$, let us consider $\squarebrack{W}\in\Graphons$ such that $\delta_2\round{\squarebrack{V},\squarebrack{W}} < \delta_\varepsilon/2$ for some $\delta_\varepsilon> 0$ such that if $\varepsilon\to 0$, then $\delta_\varepsilon\to 0$. From Definition~\ref{def:cut_metric}, there exists $\varphi\in\Ical$ such that $\delta_2\round{\squarebrack{V},\squarebrack{W}} < \norm{2}{W^\varphi-V} \leq \delta_2\round{\squarebrack{V},\squarebrack{W}} + \delta_\varepsilon/2$, i.e.,
        \begin{align}
            \delta_\varepsilon/2 > \delta_2\round{\squarebrack{V},\squarebrack{W}} \geq \norm{2}{W^\varphi - V} - \delta_\varepsilon/2 > 0.\label{eq:W_phi_approx}
        \end{align}
        From assumption if we choose $W^\varphi\in\Wcal$, since $\norm{2}{W^\varphi - V} < \delta_\varepsilon$ we get
        \begin{align}
            -\varepsilon \leq \frac{f(W^\varphi) - f(V) - \round{\inner{\phi, W^\varphi} - \inner{ \phi, V}}}{\norm{2}{W^\varphi-V}} \leq \varepsilon,\label{eq:limit_in_eps}
        \end{align}
        where $\phi = D_\Wcal(V)$. Using equations~\eqref{eq:limit_in_eps} and equation~\eqref{eq:W_phi_approx}, we get
        \begin{align}
            \frac{\round{ F\round{\squarebrack{V}} - F\round{\squarebrack{W}} }^+}{\delta_2\round{\squarebrack{V},\squarebrack{W}}} &\leq \frac{\round{ F\round{\squarebrack{V}} - F\round{\squarebrack{W}} }^+}{\norm{2}{W^\varphi - V} - \delta_\varepsilon/2}\nonumber\\
            &\leq \frac{\round{\inner{\phi, V} - \inner{\phi, W^\varphi} + \varepsilon\norm{2}{W^\varphi - V}}^+}{\norm{2}{W^\varphi - V} - \delta_\varepsilon/2}\nonumber\\
            &\leq \round{\frac{\round{\inner{\phi, V} - \inner{\phi, W^\varphi}}^+}{\norm{2}{W^\varphi - V}} + \varepsilon}\frac{\norm{2}{W^\varphi - V}}{\norm{2}{W^\varphi - V} - \delta_\varepsilon/2}\nonumber\\
            &\leq \round{\eta_F\round{\squarebrack{V}} + \varepsilon}\frac{\norm{2}{W^\varphi - V}}{\norm{2}{W^\varphi - V} - \delta_\varepsilon/2},\label{eq:upper_bound_upper_gradient}
        \end{align}
        for some $V\in\squarebrack{V}$. Taking $\varepsilon\to 0$ in equation~\eqref{eq:upper_bound_upper_gradient} we get
        \begin{align}
            \abs{\pdiff F}\round{\squarebrack{V}} &\leq \eta_F\round{\squarebrack{V}}.
        \end{align}
        \item When $\eta_F\round{\squarebrack{V}}> 0$, for all $\varepsilon\in\round{0,\eta_F\round{\squarebrack{V}}}$, by the definition of $\eta_F\round{\squarebrack{V}}$, for any $V\in\squarebrack{V}$ and $\phi = D_\Wcal(V)$, there exists $W\in\Wcal$ such that
        \begin{align}
            0 < \varepsilon < \eta_F\round{\squarebrack{V}} \leq \frac{\inner{\phi, V} - \inner{\phi, W} }{\norm{2}{V-W}} + \varepsilon.
        \end{align}
        Let $W_t \coloneqq (1-t)V + tW$ for all $t\in\interval{0,1}$. Since $\Wcal$ is a convex subset of $L^2(\interval{0,1}^{(2)})$, the curve $\round{W_t}_{t\in\interval{0,1}}\subseteq \Wcal$. Since $\norm{2}{W_t - V}\to 0$ as $t\to 0$, by assumption we have
        \begin{align}
            \lim_{t\to 0} \frac{f(W_t) - f(V) - \round{\inner{\phi, W_t} - \inner{\phi, V}}}{\norm{2}{W_t-V}} &= 0\nonumber\\
            \implies \lim_{t\to 0} \frac{f(W_t) - f(V) - t\round{\inner{\phi, W} - \inner{\phi, V}}}{t\norm{2}{W-V}} &= 0\nonumber\\
            \implies \lim_{t\to 0}\frac{f(V) - f(W_t)}{t\norm{2}{W-V}} = \frac{\inner{\phi, V} - \inner{\phi, W} }{\norm{2}{V-W}} &\geq \eta_F\round{\squarebrack{V}} - \varepsilon > 0\nonumber\\
            \implies \lim_{t\to 0}\frac{f(V) - f(W_t)}{\norm{2}{W_t-V}} = \lim_{t\to 0}\frac{\round{f(V) - f(W_t)}^+}{t\norm{2}{W-V}} &\geq \eta_F\round{\squarebrack{V}} - \varepsilon\nonumber\\
            \implies \lim_{t\to 0}\frac{\round{F\round{\squarebrack{V}} - F\round{\squarebrack{W_t}}}^+}{\delta_2\round{\squarebrack{W_t},\squarebrack{V}}} &\geq \eta_F\round{\squarebrack{V}} - \varepsilon.\label{eq:lower_bound_upper_gradient}
        \end{align}
        Therefore, the curve $\round{\squarebrack{W_t}}_{t\in\interval{0,1}} \to \squarebrack{V}$ along which  equation~\eqref{eq:lower_bound_upper_gradient} holds for every $\varepsilon>0$. When $\eta_F\round{\squarebrack{V}} = 0$, equation~\eqref{eq:lower_bound_upper_gradient} trivially holds for $\varepsilon=0$.
    \end{enumerate}
    Combining the two parts, we find that $\abs{\pdiff F}\round{\squarebrack{V}} = \eta_F\round{\squarebrack{V}}$.

    For any $n\in\Natural$ and $\delta_n>0$, let $A_{\delta_n}\coloneqq \setinline{\abs{V}<\delta_n}\cup \setinline{V=1, \phi>0}\cup \setinline{V=-1, \phi<0}$. Note that for any $t_n>0$ and $\delta_n>0$, define $W_n\coloneqq V-t_n\phi \indicator{A_{\delta_n}}{}$ and 
    \begin{align}
    \label{eqn:aux}
        \frac{\round{\inner{\phi, V} - \inner{\phi, W_n} }^+ }{\norm{2}{V-W_n}} =\enorm{\phi\indicator{A_{\delta_n}}{}}.
    \end{align}
    Let $\round{\delta_n}_{n\in\Natural}$ be a sequence in $(0,1)$ such that $\lim_{n\rightarrow\infty}\delta_n= 1$. Since $\phi \in L^\infty\big(\interval{0,1}^2\big)$, for every $\delta_n>0$, there exists $t_n>0$ such that $W_n=V-t_n\phi \indicator{A_{\delta_n}}{}\in \Wcal$ for each $n\in \Natural$. It follows from equation~\eqref{eqn:aux} that 
    \begin{align}
    \label{eqn:lowerBound}
        \eta_F([V])\ge \limsup_{n\to\infty}\enorm{\phi\indicator{A_{\delta_n}}{}}=\enorm{\phi\indicator{G_V}{}},
    \end{align}
    where the last equality follows from the dominated convergence theorem and the fact that  $\indicator{A_{\delta_n}}{}\to \indicator{G_V}{}$ a.e. as $\delta_n\to 1$.
    
    For any $W\in \Wcal$, define $W_0=W$ on $G_V$ and $W_0=V$ otherwise. Note that 
    \begin{align}
        \inner{\phi, V-W} &=\int_{G_V}\phi(V-W)\diff \lambda_{\interval{0,1}^2}  +\int_{\interval{0,1}^2\setminus G_V} \phi (V-W)\diff \lambda_{\interval{0,1}^2} \nonumber\\
        &=\int_{G_V}\phi(V-W_0)\diff \lambda_{\interval{0,1}^2}  +\int_{\interval{0,1}^2\setminus G_V} \phi (V-W)\diff \lambda_{\interval{0,1}^2} \nonumber\\
        &\leq \int_{G_V}\phi(V-W_0)\diff \lambda_{\interval{0,1}^2}  = \int \phi (V-W_0)\diff \lambda_{\interval{0,1}^2}\nonumber\\
        &= \inner{\phi, V-W_0},\label{eq:phiVW_leq_phiVtilW}
    \end{align}
    where the inequality above follows from the fact that $\phi(V-W)\leq 0$ on $\interval{0,1}^2\setminus G_V$. Using that $\nrm{2}{V-W_0}\le \enorm{V-W}$, we obtain
    \begin{align*}
        \frac{\round{\inner{\phi, V} - \inner{\phi, W} }^+ }{\norm{2}{V-W}}\leq \frac{\round{\inner{\phi, V} - \inner{\phi, W_0} }^+ }{\enorm{V-W_0}}.
    \end{align*}
    It therefore follows that 
    \begin{align}
    \label{eqn:Eqn_eta}
    \eta_F\round{\squarebrack{V}} \coloneqq \sup_{W}\frac{\round{\inner{\phi, V} - \inner{\phi, W} }^+ }{\norm{2}{V-W}},
    \end{align}
    where the supremum is taken over $W\in\Wcal$ such that $W=V$ on $\interval{0,1}^2\setminus G_V$. For any such $W$, we obtain by the Cauchy--Schwarz inequality that $\inner{\phi, V-W}\leq \enorm{\phi\indicator{G_V}{}}\enorm{V-W}$. Therefore, it follows from  that 
    \begin{align}
        \label{eqn:upperBound}
        \eta_F\round{\squarebrack{V}} &\leq \enorm{\phi\indicator{G_V}{}}.
    \end{align}
   Combining equations~\eqref{eqn:lowerBound} and~\eqref{eqn:upperBound}, the conclusion follows.
\end{proof}

\begin{remark}
    We can define a similar expression for the set valued function $G$ when $\effdom(f)$ is a cubic domain. As an example, when $\effdom(f) = \setinline{W\in\Wcal \given a \leq W \leq b \text{ a.e.}}$ for some $-1\leq a \leq b\leq 1$ (see Section~\ref{sec:scalar_entropy} for a discussed example), we can define $G_V\subseteq [0,1]^{(2)}$ for any $V\in\effdom(f)$ as
    \begin{align}
        G_V = \Set*{a < V < b} \cup \Set*{V=b,\phi>0} \cup \Set*{V=a,\phi<0},
    \end{align}
    for $(V,\phi\coloneqq D_\Wcal f(V)) \in \squarebrack{V}\odot\squarebrack{\phi}$. Lemma~\ref{lem:upper_gradient_frechet} continues to hold when $V\in\effdom(f)\subset\Wcal$ whenever $\effdom(f)$ is a cubic domain. In this case, the set valued function $G$ is defined as described above and the proof of Lemma~\ref{lem:upper_gradient_frechet} can be modified accordingly.
\end{remark}

\begin{remark}\label{rem:LocalSlope_is_phi_and_hence_lsc}
Lemma~\ref{lem:upper_gradient_frechet} has an important consequence that will be used later. As the metric derivative of a gradient flow is given by its local slope at each point, Lemma~\ref{lem:upper_gradient_frechet} says that if $\omega$ is a gradient flow of $F$, then its local slope is given by the $L^2$-norm of its Fr\'echet-like derivative, i.e., $\abs{\pdiff F}(\omega_t)=\enorm{\phi(\omega_t)\indicator{G_{\omega_t}}{}}=\enorm{D_{\Graphons}F(\omega_t)\indicator{G_{\omega_t}}{}}$ for all $t>0$. Here for any $t>0$,
\[
    D_{\Graphons}F(\omega_t)\indicator{G_{\omega_t}}{} \coloneqq \Set*{\round{D_\Wcal f(U_t)\indicator{G_{U_t}}{}}^\varphi\in L^\infty\big(\interval{0,1}^{(2)}\big) \given \varphi\in\Tcal},
\]
for $U_t\in\omega_t$. Since the $L^2$-norm is invariant under measure preserving transformations~\cite[Lemma 5.5]{janson2010graphons}, the $L^2$-norms of graphons are well-defined. In fact, if one defines a kernel valued curve $(W_t)_{t\in [0, T]}$ by setting $W'_t=-D_{\Wcal}f(W_t)\indicator{G_{W_t}}{}$ pointwise, then the curve $t\mapsto \omega_t=[W_t]$ is a gradient flow (a.k.a. curve of maximal slope). This is shown in Lemma~\ref{lem:EVI_MM} which in turn shows the existence of a gradient flow under suitable assumption (See Theorem~\ref{thm:Existence_with_FrehetDerivative}). 
\end{remark}

\begin{lemma}\label{lem:ACtoGF}
    Let $F\colon\Graphons\to\rr\cup\Set*\infty$ be a function and $f\colon\Wcal\to\rr\cup\Set*\infty$ be its invariant extension. Let $F$ be Fr\'echet differentiable. Let us consider $\omega\in\mathrm{AC}(\Graphons,\delta_2)$, and let $\round{W_t}_{t\in\interval{0,1}}\in\mathrm{AC}\round{\Wcal,d_2}$ be its representative curve such that $W'_{t} = -\eta_{F}\round{\squarebrack{W_{t}}}N_t$ for a.e. $t\in\interval{0,1}$ for some $N_t\in L^\infty\big(\interval{0,1}^{(2)}\big)$ satisfying $\enorm{N_t}=1$ and $\inner{\phi_t, N_t}=\eta_F\round{\squarebrack{W_{t}}}$.
    Then, $\omega$ is a curve of maximal slope on $(\Graphons,\delta_2)$.
\end{lemma}
\begin{proof}
    Since $\round{W_{t}}_{t\in\interval{0,1}}\in \mathrm{AC}\round{\Wcal,d_2}$, the metric derivative of $\round{W_{t}}_{t\in\interval{0,1}}$ with respect to $d_2$ at any $t\in (0,1)$ is given by 
    \begin{align}
        \lim_{h\to 0}\frac{\enorm{W_{t+h} - W_{t}}}{\abs{h}} &= \enorm{W'_{t}}
        = \enorm{\eta_{F}\round{\squarebrack{W_{t}}}N_{t}}
        = \abs{\eta_{F}\round{\squarebrack{W_{t}}}}.
    \end{align}
    That is, the metric derivative of $\round{W_{t}}_{t\in\interval{0,1}}\in\mathrm{AC}\round{\Wcal,d_2}$ equals the upper gradient. Moreover, by the absolute continuity of the curve and from Definition~\ref{def:frechet_like_derivative},
    \begin{align}
        \deriv{}{t}f\round{W_{t}} &= \lim_{h\to 0}\frac{f\round{W_{t+h}} - f\round{W_{t}}}{h}\nonumber\\
        &= \lim_{h\to 0}\frac{\inner{\phi_{t},W_{t+h}} - \inner{\phi_{t},W_{t}} + o\round{\enorm{W_{t+h}-W_{t}}}}{h}\nonumber\\
        &= \inner{\phi_{t},W'_{t}} + 0
        = -\inner{\phi_{t},N_{t}}\eta_{F}\round{\squarebrack{W_{t}}}
        = -\eta_{F}^2\round{\squarebrack{W_{t}}},
    \end{align}
    where $\phi_{t} = D_{\Wcal}f\round{W_{t}}$. Thus, $\round{W_{t}}_{t\in\interval{0,1}}$ satisfies Definition~\ref{def:curves_of_maximal_slope} and is a curve of maximal slope on $\round{\Wcal,d_2}$, and $\omega$ is a curve of maximal slope on $(\Graphons,\delta_2)$.
\end{proof}

\subsubsection{Existence of gradient flow}\label{sec:Existence_with_FD}
We now prove the existence of a curve of maximal slope if $F$ satisfies reasonable assumptions. Moreover, as mentioned in the introduction, we show that the curve of maximal slope is the natural image of an absolutely continuous curve in $(\Wcal, d_2)$.
\begin{lemma}\label{lem:EVI_MM}
Let $f\colon\Wcal\to \rr\cup\Set*\infty$ be a $\lambda$-semiconvex invariant function for some $\lambda\in\rr$ such that the Fr\'echet-like derivative, $\phi(W) = D_\Wcal f(W)$, exists for all $W\in \effdom(f)$. Let $(W_t)_{t\in\interval{0,1}}\in \mathrm{AC}(\Wcal, d_2)$ be an absolutely continuous curve satisfying $W'_t=-\phi(W_t)\indicator{G_{W_t}}{} =- D_\Wcal f(W_t)\indicator{G_{W_t}}{}$ for a.e. $t\in\interval{0,1}$. Then, $(\squarebrack{W_t})_{t\in\interval{0,1}}$ is the unique \emph{minimizing movement curve} (MM) satisfying the following evolution variational inequality (EVI) 
\begin{align}
    \inv{2}\deriv{}{t}d_2^2\round{W_t,V} + \frac{\lambda}{2}\enorm{W_t-V}^2 + f(W_t) \leq f(V),\label{eq:EVI}
\end{align}
for every $V\in \effdom(f)$. 
\end{lemma}
\begin{proof}
The curve $(W_t)_{t\in[0,1]}$ is a curve of maximal slope follows from Lemma~\ref{lem:ACtoGF}. We now show that it satisfies the EVI. For $t\in\rr_+$, let $\phi_t \coloneqq D_\Wcal f(W_t)$. Fix $U\in \Wcal$ and define the function $g_U\colon\Wcal\to\rr\cup\Set*\infty$ by $g_U(V)\coloneqq f(V) - \lambda \enorm{U-V}^2/2$, for $V\in\effdom(f)$. We first observe that $D_\Wcal g_{W_t}(W_t)=\phi_t$. To see this, note that 
\begin{align}
    \lim_{s\to t} \frac{f(W_s) - f(W_t) - \inner{\phi_t,W_s-W_t}}{\enorm{W_s-W_t}} &= 0\nonumber\\
    \implies \lim_{s\to t} \frac{g_{W_t}(W_s) - g_{W_t}(W_t) - \inner{\phi_t,W_s-W_t}}{\enorm{W_s-W_t}} &= 0.\label{eq:phi_t_eq_psi_t}
\end{align}
The conclusion, that is $D_\Wcal g_{W_t}(W_t)=\phi_t$, now follows from the uniqueness of Fr\'echet-like derivatives (Lemma~\ref{lem:frechet_consistency}). Since $f$ is $\lambda$-semiconvex, $g_U$ is convex, i.e.,
\begin{align}
   g_{W_t}(V) &\geq g_{W_t}(W_t) + \inner{\phi_t,V-W_t}\ ,\quad V\in\effdom(f).\label{eq:g_convexity}
\end{align}
From equation~\eqref{eq:phiVW_leq_phiVtilW} and using the fact that $W'_t=\phi_t$, we obtain 
\begin{align}
    \inner{\phi_t,V-W_t} \leq \inner{\phi_t\indicator{G_{W_t}}{},V-W_t} &= \inner{-\deriv{}{t}W_t,V-W_t}\nonumber\\
    &=\inv{2}\deriv{}{t}\enorm{W_t-V}^2 = \inv{2}\deriv{}{t}d_2^2\round{W_t,V}, \label{eq:ddt_eq_innerprod}
\end{align}
where the second equality follows from the reflexivity of $L^2([0, 1]^2)$. Plugging equations~\eqref{eq:phi_t_eq_psi_t} and~\eqref{eq:ddt_eq_innerprod} in equation~\eqref{eq:g_convexity} and rearranging, we get
\begin{align}
    \inv{2}\deriv{}{t}d_2^2\round{W_t,V} + \frac{\lambda}{2}\enorm{W_t-V}^2 + f(W_t) &\leq f(V),\label{eq:EVI_lambda}
\end{align}
for all $V\in\effdom(f)$. Using equation~\eqref{eq:EVI_lambda} it follows from~\cite[Theorem 4.0.4]{ambrosio2005gradient} that the curve $\round{\squarebrack{W_t}}_{t\in\interval{0,1}}$ is the unique curve in $\mathrm{MM}_{\delta_2,\delta_\cut}\round{\Phi_F,\squarebrack{W_0}}$.
\end{proof}

\begin{theorem}[Existence of curve of maximal slope-II]\label{thm:Existence_with_FrehetDerivative}
    Let $F\colon\Graphons\to \rr\cup\Set*\infty$ be a real valued function such that the Fr\'echet-like derivative $D_{\Graphons}F([W])$ exists for all $[W]\in \effdom(F)$. Let $f\colon\Wcal\to \rr\cup\Set*\infty$ be its invariant extension. For $W_0\in [W_0]\in \effdom(F)$ and $t\geq 0$ define 
    \[
        W_t \coloneqq W_0-\int_0^{t}\phi(W_s)\indicator{G_{W_s}}{}\diff s, \qquad t\in\rr_+,
    \]
    where the above integral is pointwise. If $f$ is $\lambda$-semiconvex w.r.t. $d_2$, then the curve $t\mapsto \omega_t=[W_t]$ is a curve maximal slope for $F$ starting at $[W_0]\in \effdom(F)$. 
\end{theorem}
\begin{proof}
    Fix $[W_0]\in \effdom(F)$ and define 
    \[
        W_t \coloneqq W_0-\int_0^{t}\phi(W_s)\indicator{G_{W_s}}{}\diff s\, ,\quad t\in(0,1],
    \]
    where the above integral is a pointwise integral, i.e., for a.e. $(x,y)\in[0,1]^2$,
    \[
        W_t(x,y) \coloneqq W_0(x,y)-\int_0^{t}\phi(W_s)(x,y)\indicator{G_{W_s}}{(x,y)}\diff s\, ,\quad t\in(0,1].
    \]
    By construction, we have $\round{W_t}_{t\in\interval{0,1}}\in\mathrm{AC}(\Wcal, d_2)$ and $W_t'=-\phi(W_t)\indicator{G_{W_t}}{}$ for all $t\in\interval{0,1}$. It follows from Lemma~\ref{lem:EVI_MM} that $\round{W_t}_{t\in\interval{0,1}}$ is a minimizing movement. It follows from the definition of minimizing movements (see Section~\ref{sec:implicit_euler} and~\cite[Definition 2.0.6]{ambrosio2005gradient}) that there exists a sequence of discrete solutions in $\Graphons$ that converges to $\round{\squarebrack{W_t}}_{t\in\interval{0,1}}$ in $\delta_\cut$. Since $\Wcal$ is closed in $ \cutnorm{}$,  $W_t\in \Wcal$ for all $t\in\interval{0,1}$.
    
    Set $\omega_t=[W_t]$ for $t\in [0, 1]$. Then $\omega\in \mathrm{AC}(\Graphons, \delta_2)$. From Lemma~\ref{lem:upper_gradient_frechet}, we know that for any $t\in\interval{0,1}$, $\eta_F([W_t])=\norm{2}{\phi(W_t)\indicator{G_{W_t}}{}}$ and therefore we have $W_t'=-\eta_{f}(W_t)N_t$ where  $N_t\coloneqq \phi(W_t)\indicator{G_{W_t}}{}/\enorm{\phi(W_t)\indicator{G_{W_t}}{}}$. It follows from Lemma~\ref{lem:ACtoGF} that $\omega$ is a curve of maximal slope on $(\Graphons,\delta_2)$.
\end{proof}
\begin{remark}\label{rem:velocityofgradientflow}
    An important consequence of the above Theorem is that if $\omega$ is a gradient flow of $F$ then there exists an absolutely continuous curve $(W_t)_{t\in\interval{0,T}}\in \mathrm{AC}(\Wcal,d_2)$ such that $[W_t]=\omega_t$, $\abs{\omega'}(t)=\enorm{D_{\Wcal}f(W_t)\indicator{G_{W_t}}{}}$ and $W'_t=-D_{\Wcal}f(W_t)\indicator{G_{W_t}}{}$, for each $t\in (0,T]$.
\end{remark}

\begin{remark}\label{rem:exp_convergence}
    If $F$ is $\delta_\cut$-lower semicontinuous, $\lambda$-geodesically semiconvex for $\lambda\in\rr_+$, and bounded from below, then one can say more about the convergence rate of a gradient flow to a minimizer of $F$. When $\lambda>0$, let $\omega^*$ be the unique minimizer of $F$. Then following~\cite[Remark 4.0.5, part (d)]{ambrosio2005gradient}, a gradient flow $\omega$ of $F$ on $\Graphons$ starting at $\omega_0\in\Graphons$ satisfies
    \[
        \delta_2\round{\omega_t,\omega^*} \leq \eu^{-\lambda t}\delta_2\round{\omega_0,\omega^*}, \qquad t\in\rr_+.
    \]
    In the limiting case when $\lambda = 0$ the exponential decay does not occur, in general, but some weaker results on the asymptotic behavior of $\omega$ hold. Following~\cite[Corollary 4.0.6]{ambrosio2005gradient}, $\omega$ satisfies
    \[
        F(\omega_t) - F(\omega_\infty) \leq \frac{\delta_2^2(\omega_0,\omega_\infty)}{2t},\qquad t\in\rr_+,
    \]
    for some minimum point $\omega_\infty$ of $F$, such that the map $t\mapsto \delta_2(\omega_t,\omega_\infty)$ is  non-increasing. Moreover, $\lim_{t\to\infty}\delta_2(\omega_t,\omega_\infty) = 0$.
\end{remark}

\subsection{Finite dimensional Fr\'echet-like derivatives and upper gradients}\label{sec:finite_dim_frechet}
Recall the partition $Q_k\coloneqq\Set*{Q_{k,i}}_{i \in [k]}$ defined for any $k\in\Natural$ in Section~\ref{sec:topologies_on_graphons}. Given an invariant function $f\colon\Wcal\to\rr\cup\Set*\infty$, we can restrict its domain to kernels in $\Wcal_k$ and still consider the Fr\'echet-like derivative $D_{\Wcal_k}f\colon \Wcal_k\cap \effdom(f)\to L^\infty_k\big(\interval{0,1}^{(2)}\big)$.

There are two equivalent ways of doing this. First, suppose the Fr\'echet-like derivative of $f$ at $V$ is given by $D_\Wcal f(V) = \phi$. Then define $D_{\Wcal_k}f(V)=\phi_k$ by conditional expectations as $\phi_k \coloneqq \E{\phi \given \Fcal_k}$,
where $\Fcal_k \coloneqq \sigma\round{Q_k\times Q_k}$. 
The object $\phi_k$ is referred to as a `quotient' obtained by a `stepping' of $\phi$ in~\cite[Section 3.3]{borgs2008convergent} and~\cite[Section 9.2.1]{lovasz2012large} respectively. Since $\phi = D_\Wcal f(V)$, by the \textit{Tower Property} of conditional expectations we obtain that when $W,V\in\Wcal_k$,
\begin{align}
    &\inner{\phi_k, W} - \inner{\phi_k, V} = \inner{\phi, W} - \inner{\phi, V},\nonumber\\
    \implies &\lim_{\substack{W\in\Wcal_k,\\\enorm{W-V}\to 0}}\frac{f(W) - f(V) - \round{\inner{\phi_k, W} - \inner{\phi_k, V}}}{\norm{2}{W-V}} = 0.\label{eq:D_Wk_limit_verify}
\end{align}
The second method of defining $\phi_k$ is to relate it to the Euclidean gradient over $k\times k$ symmetric matrices. This is done in Lemma~\ref{lem:phi_equal_partial_derivatives} below.

In any case, we can define $\eta_{F,k}\colon\Graphons_k\cap \effdom(F)\to\rr_+$ for the function $F\colon\Graphons\to\rr\cup\Set*\infty$ as follows. If $V\in\squarebrack{V}\in\Graphons_k\cap \effdom(F)$, then
\begin{align}
    \eta_{F,k}\round{\squarebrack{V}} &\coloneqq \sup_{W\in\Wcal_k}\frac{\round{\inner{\phi_k, V} - \inner{\phi_k, W} }^+ }{\norm{2}{V-W}}.\label{eq:norm_frechet_def_k}
\end{align}
We can also define the local slope $\abs{\pdiff_k F}$ restricted to $\Graphons_k$ as
\begin{align}
    \abs{\pdiff_k F}\round{\squarebrack{V}} &\coloneqq \limsup_{\squarebrack{W}\in\Graphons_k,\,  \delta_{2}\round{\squarebrack{W},\squarebrack{V}} \to 0} \frac{\round{F\round{\squarebrack{V}} - F\round{\squarebrack{W}}}^+}{\delta_2\round{\squarebrack{W},\squarebrack{V}}},
\end{align}
for $\squarebrack{V}\in\Graphons_k\cap \effdom(F)$. Then, by a similar argument as shown in the proof of Lemma~\ref{lem:upper_gradient_frechet}, we have the following corollary.
\begin{corollary}\label{cor:upper_gradient_frechet}
    \sloppy Let $F\colon\Graphons\to\rr\cup\Set*\infty$ be a function and $f\colon\Wcal\to\rr\cup\Set*\infty$ be its invariant extension. Assume that for $\squarebrack{V}\in\Graphons_k\cap \effdom(F)$ the Fr\'echet-like derivative $D_\Wcal f(V)$ exists for all $V\in\squarebrack{V}$. Then the local slope (Definition~\ref{def:local_slope}) of $F$ at $\squarebrack{V}$ satisfies $\abs{\pdiff_k F}\round{\squarebrack{V}} = \eta_{F,k}\round{\squarebrack{V}}$.
\end{corollary}

\subsection{Convergence of finite dimensional gradient flows}\label{sec:convergence}
In Section~\ref{sec:intro} we discussed that implicit Euler iteration on $\Graphons_k$ for any $k\in\Natural$ can be viewed as the time scaling of the implicit Euler method on the Euclidean space of $k\times k$ symmetric matrices. The following lemma complements that discussion by saying that the gradient flow on $\Wcal_k$ can be obtained from the Euclidean gradient flow on the space of $k\times k$ symmetric matrices.

To set the stage, let $f\colon\Wcal\to\rr\cup\Set*\infty$ be an invariant function. Consider its restriction to $\Wcal_k$, viewed as the space of $k\times k$ symmetric matrices: $f_k \coloneqq f\circ K$. This is a function on a closed convex subset of an Euclidean space. Suppose the function $f_k$ is $C^1$ up to the boundary then we show that the Euclidean gradient and the Fr\'echet-like derivative are equal up to a scaling.

\begin{lemma}\label{lem:phi_equal_partial_derivatives}
    \sloppy Let $k\in\Natural$. Let $f\colon\Wcal\to \mathbb{R}\cup \{+\infty\}$ be an invariant function that is Fr\'echet differentiable according to Definition~\ref{def:frechet_like_derivative}, that is, $D_{\Wcal_k}f(V)$ exists for every $V\in\Wcal_k$. If $f_k\coloneqq f\circ K$ is differentiable up to the boundary of $\Mcal_k$, then
    \[
        k^2\round{\nabla f_k\circ M_k} (V) = \round{M_k\circ D_{\Wcal_k}f}\round{V},\qquad V\in \Wcal_k,
    \]
    where $\round{\nabla f_k}_{i,j} \coloneqq \pdiff_{i,j} f_k$ for all $(i,j)\in\squarebrack{k}^{(2)}$, and $M_k$ is as defined in Definition~\ref{def:K_Mk}.
\end{lemma}

\begin{proof} Since $f_k$ is assumed to be differentiable on a finite dimensional Euclidean space, $\nabla f_k$ is its Fr\'echet derivative as well. By composing with $M_k$, which scales distances by a factor, we get that $k^2\round{\nabla f_k\circ M_k}$ is a Fr\'echet-like derivative on $\Wcal_k$. We have already shown in equation~\eqref{eq:D_Wk_limit_verify} that $D_{\Wcal_k}f$ is also a Fr\'echet-like derivative on $\Wcal_k$. We are done by arguing that Fr\'echet-like derivatives are unique by following an argument very similar to that of Lemma~\ref{lem:frechet_consistency}. 
\end{proof}

\sloppy Let $k\in\Natural$, and $\round{W_{k,t} = W_k(t)\in\Wcal_k\cap \effdom(f)}_{t\in\rr_+}$ be the Euclidean coordinate gradient flow of $f$. This may be obtained by suitably scaling the solution of the differential equation
\begin{align}
    \deriv{}{t}M_k(W_{k}(t)) &= -\round{\nabla f_k\circ M_k}\round{W_{k}(t)},\label{eq:coordinate_gradient_flow_k}
\end{align}
with initial condition $W_k(0) = W_{k,0}\in \Wcal_k\cap \effdom(f)$, until the process hits the boundary when one or more entries is $\pm 1$. At the boundary, however, the gradient might push the process outside $\Mcal_k$ and it needs some care to have a proper definition. Instead, we consider the Euclidean gradient flow as the limit of implicit Euler iterations as the step size tends to zero. This definition is valid everywhere and is equivalent to the previous one on Euclidean spaces. 

As a consequence of Lemma~\ref{lem:phi_equal_partial_derivatives}, we obtain that the Euclidean coordinate-wise gradient flow on $\Wcal_k$ is the gradient flow on $\Wcal_k$. We are now ready to prove Theorem~\ref{thm:GF_convergence}. For completeness, we reproduce the theorem statement below. 

\begin{theorem}[Convergence of Gradient Flows]
    Suppose $F\colon\Graphons \to \rr\cup\Set*\infty$ satisfies the following conditions:
    \begin{enumerate}
        \item $F$ is continuous in $\delta_\cut$,
        \item $F$ is $\lambda$-semiconvex (Definition~\ref{def:lambda_cvx_along_curve}) along generalized geodesics on $(\Graphons,\delta_2)$ (Definition~\ref{def:gen_geodesics}), for some $\lambda\in\rr$.
    \end{enumerate}
    Consider the gradient flow $\omega^{(k)} = \big(\omega^{(k)}_t\big)_{t\in\rr_+}\subset \Graphons_k$ of $F$ on each $\Graphons_k$, starting at some $\omega^{(k)}_0 = \squarebrack{U_{k,0}}$ for $k\in\Natural$. Assume that the sequence $\round{\squarebrack{U_{k,0}}}_{k\in\Natural}\xrightarrow{\delta_\cut}\squarebrack{U_0}$, and $\abs{\pdiff F}\round{\squarebrack{U_{0}}}<\infty$ and $\limsup_{k\to\infty}\abs{\pdiff F}\round{\squarebrack{U_{k,0}}} \leq G < \infty$, for some $G\geq 0$. Then,
    \begin{align}
        \limsup_{k\to\infty}\sup_{t\in[0,T]}\cutmetric{\omega^{(k)}_t}{\omega_t} &= 0,
    \end{align}
    for any $T\in\rr_+$, where $\omega = \round{\omega_t}_{t\in\rr_+}$ is the unique minimizing movement curve~\cite[Definition 2.0.6, page 42]{ambrosio2005gradient} on $\Graphons$ for the function $F$ starting at $\omega_0 = \squarebrack{U_0}$~\cite[Theorem 4.0.4]{ambrosio2005gradient}.
    In addition, if the conditions for the existence of curves of maximal slope (Theorem~\ref{thm:existence} or  Theorem~\ref{thm:Existence_with_FrehetDerivative}) hold, then $\omega$ is also a curve of maximal slope.
\end{theorem}

\begin{proof} By increasing the constant $G$ suitably, we may assume that 
\begin{equation}\label{eq:increaseG}
\max\Set*{\sup_{k\ge 2}\abs{\pdiff F}\round{\squarebrack{U_{k,0}}}, \abs{\pdiff F}\round{\squarebrack{U_{0}}}}\le G < \infty. 
\end{equation}
Fix $T>0$ and let $\tauvec_m$ be a sequence of positive time steps such that $\abs{\tauvec_m}=T/m$. Since $F\colon\Graphons\to\rr\cup\Set*{\infty}$ is $\delta_\cut$-continuous and $\delta_2\round{\squarebrack{U},\slot{}}\colon\Graphons \to \rr\cup\Set*{\infty}$ is $\delta_\cut$-lower semicontinuous, the functional $\Phi_F\round{\tau,\squarebrack{U};\slot{}}$ is $\delta_\cut$-lower semicontinuous. From Lemma~\ref{lem:gen_geodesic_cvx_delta_2}, it follows that $\Phi_F$ satisfies~\cite[Assumption 4.0.1]{ambrosio2005gradient} and hence by~\cite[Proposition 4.0.4]{ambrosio2005gradient} we have that
\[
    \omega^{(k)}(t) = \toplim{\delta_\cut}_{m\to\infty}\round{J^{(k)}_{t/m}}^{m}\round{\squarebrack{U_{k,0}}}\, , \;\quad \omega(t) = \toplim{\delta_\cut}_{m\to\infty}\round{J_{t/m}}^{m}\round{\squarebrack{U_{0}}},
\]
exist and are unique for all $k\in\Natural$ and $t\in\interval{0,T}$. 

Let $\overline{\squarebrack{U_{k,\tauvec_m}}}\colon[0,T]\to\Graphons$ be the discrete solution (Definition~\ref{def:discrete_sol}) of the implicit Euler method with the sequence $\tauvec_m$ and initial point $\squarebrack{U_{k,0}}$, for each $k\in\Natural$. Inductively applying the Proposition~\ref{prop:gamma_convergence}, we obtain $\overline{\squarebrack{U_{\tauvec_m}}}\colon[0,T]\to\Graphons$ such that $\overline{\squarebrack{U_{\tauvec_m}}}$ is the discrete solution of the implicit Euler method with the sequence $\tauvec_m$ and initial point $\squarebrack{U_{0}}\in\Graphons$. Passing to a subsequence and relabeling, we may assume that $\round{\overline{[{U}_{k, \tauvec_m}]}}_{k\in\Natural}\xrightarrow{\delta_\cut} \overline{[{U}_{\tauvec_m}]}$ uniformly on $[0, T]$ as $k\to \infty$, that is, for any fixed sequence of step sizes $\tauvec_m$, we have $\delta_{\cut}\round{\overline{\squarebrack{U_{k,\tauvec_m}}}(t), \overline{\squarebrack{U_{\tauvec_m}}}(t)}\to 0$ uniformly over $t\in[0, T]$  as $k\to \infty$. For every $t\in\interval{0,T}$ we have
\begin{align}
    \delta_2\round{\overline{\squarebrack{U_{k,\tauvec_m}}}(t), \omega^{(k)}(t)} &< \gamma(\abs{\tauvec_m},\lambda,t;T,\abs{\pdiff F_k}\round{\squarebrack{U_{k,0}}}),\label{eq:4.0.6_k}\\
    \delta_2\round{\overline{\squarebrack{U_{\tauvec_m}}}(t), \omega(t)} &< \gamma(\abs{\tauvec_m},\lambda,t;T,\abs{\pdiff F}\round{\squarebrack{U_{0}}}),
    \label{eq:4.0.6}
\end{align}
for every $k\in\Natural$ where
\[
    \gamma\colon \Set*{(\tau,\lambda,t,T)\in \rr_{++}\times\rr\times\rr_{++}\times\rr_{++} \given \tau\lambda>-1, \tau\leq T,t\leq T} \times (0,\infty) \to \rr_+
\]
is defined as
\begin{align}
    \gamma(\tau,\lambda,t;T,G) &\coloneqq \begin{cases}
        \frac{\tau G}{\sqrt{2}}\, , & \text{if $\lambda=0$},\\
        \frac{1+2\abs{\lambda}T}{1+\lambda\tau}\cdot\frac{\tau G}{\sqrt{2}}\cdot \exp\round{-\ln\round{\frac{1+\lambda\tau}{\tau}}t}\, , & \text{if $\lambda<0$},\\
        \sqrt{1+2\lambda T} \cdot\frac{\tau G}{\sqrt{2}}\cdot \exp\round{-\ln\round{\frac{1+\lambda\tau}{\tau}}t}\, , & \text{if $\lambda>0$},
    \end{cases}
\end{align}
by~\cite[Equation 4.0.6, Theorem 4.0.9, Theorem 4.0.10]{ambrosio2005gradient}, and the uniform bound in equation~\eqref{eq:increaseG}. Note that $\gamma$ is independent of $k$. Using the triangle inequality, we get 
\begin{align}
    \delta_{\cut}\round{\omega^{(k)}_t, \omega_t} &\le \delta_{\cut}\round{\omega^{(k)}_t, \overline{\squarebrack{U_{k,\tauvec_m}}}(t)}+\delta_{\cut}\round{\overline{\squarebrack{U_{k,\tauvec_m}}}(t), \overline{\squarebrack{U_{\tauvec_m}}}(t)}\nonumber\\
    &\qquad\qquad\qquad+\delta_{\cut}\round{\overline{\squarebrack{U_{\tauvec_m}}}(t), \omega_t} \nonumber\\
    &\le\delta_2\round{\omega^{(k)}_t,\overline{\squarebrack{U_{k,\tauvec_m}}}(t)}+\delta_{\cut}\round{\overline{\squarebrack{U_{k,\tauvec_m}}}(t), \overline{\squarebrack{U_{\tauvec_m}}}(t)}\nonumber\\
    &\qquad\qquad\qquad+\delta_2\round{\overline{\squarebrack{U_{\tauvec_m}}}(t), \omega_t}\nonumber \\
    &\le 2\gamma(\abs{\tauvec_m}, \lambda,t; T, G)+ \delta_{\cut}\round{\overline{\squarebrack{U_{k,\tauvec_m}}}(t), \overline{\squarebrack{U_{\tauvec_m}}}(t)},
 \label{eqn:unifInequality}
\end{align}
for all $k\in\Natural$ and $t\in\interval{0,T}$ by equations~\eqref{eq:4.0.6_k} and~\eqref{eq:4.0.6}.

It is clear that $\gamma(\abs{\tauvec_m}, \lambda,t; T, G)\to 0$ uniformly on $\interval{0,T}$ as $\abs{\tauvec_m}\to 0$. Therefore, we conclude from equation~\eqref{eqn:unifInequality} that $\delta_{\cut}\round{\omega^{(k)}_t, \omega_t}\to 0$ uniformly on $t\in[0, T]$ as $k\to\infty$.
\end{proof}
\begin{remark}
    The proof of the Theorem~\ref{thm:GF_convergence} can be carried as long as we have the uniform estimates in equation~\eqref{eq:4.0.6_k}. In particular, if $\nabla f_k\circ M_k$ are uniformly Lipschitz, and there is a constant $m\in\rr_+$ such that
    \begin{align*}
        f_k(B) &\leq f_k(A) + \inner{\nabla f_k(A),B-A} + \frac{m}{2}\normF{B-A}^2,
    \end{align*}
    for all $A,B\in\Mcal_k$, where $f_k\coloneqq \round{f\circ K}\mid_{\Mcal_k}$, then~\cite[Theorem 212A]{butcher2016numerical} guarantees a uniform estimate in~\eqref{eq:4.0.6_k} and therefore the conclusion of the theorem remains valid.
\end{remark}

\subsection{Continuity Equations}\label{sec:continuity_eq}

It is well-known that any absolutely continuous curve in the Wasserstein space can be represented as the solution of a continuity equation~\cite[Section 5.3]{santambrogio2015optimal}. Something analogous is partially true for graphons as well. However, the presence of the boundary in $\Graphons$ makes the situation more delicate and we can only characterize AC curves via the continuity equation until it hits the boundary. 

Before we state the main result, we introduce some notations. Let $v\in L^1(\interval{0,1}^{(2)})$ and let $\squarebrack{W}\in\Graphons$. Let $W\in\squarebrack{W}$ be a representative of $[W]$. For any $k\in\Natural$,  we can define $X_k\colon\interval{0,1}^k\to\Mcal_k$ and $v_k\colon \Mcal_k \to \Rd{\squarebrack{k}^{(2)}}$ as
\begin{align}
    X_k\round{\round{u_\ell}_{\ell=1}^k} &\coloneqq \round{W(u_i,u_j)}_{(i,j)\in\squarebrack{k}^{(2)}},\nonumber\\
    v_k(z)(i,j) &= \E{v(U_i,U_j) \given X_k\round{\round{U_\ell}_{\ell=1}^k} = z},\label{eq:general_velocity}
\end{align}
where $\Set*{U_i}_{i\in\Natural}$ are i.i.d. as $\mathrm{Uni}\interval{0,1}$. Intuitively, formula~\eqref{eq:general_velocity} means that we average the edge weights from $v$ over all embedding of the vertex labeled weighted graph $z$ in the graphon $W$. 
Since $X_{k-1}$ is a leading principle submatrix of $X_k$, i.e., $X_{k-1} = \round{X_k(i,j)}_{(i,j)\in[k-1]^{(2)}}$, we have that $\sigma\round{X_{k-1}} \subseteq \sigma\round{X_k}$, which defines a filtration $\Fcal = \round{\Fcal_k \coloneqq \sigma(X_k)}_{k\in\Natural}$. It is clear that $\round{v_{k}}_{k\in\Natural}$ is a martingale with respect to the filtration $\Fcal$. We also note that that $v_k$ is a function of the graphon and not its kernel representative. We record both these observations as lemma below. 

\begin{lemma}\label{lem:v_martingale}
    For every $i, j\in\Natural$, the process $\round{v_{k}(X_k)(i, j)}_{k=\max\{i, j\}}^\infty$ is a martingale with respect to the filtration $\Fcal$.
\end{lemma}

\begin{lemma}\label{lem:invariance_of_v_kW_k}
    For any $\varphi\in \Tcal$, we have $v_k^\varphi\round{X^\varphi_k} = v_k(X_k)$,
    for all $k\in\Natural$, where $v^\varphi(x,y) \coloneqq v\round{\varphi(x),\varphi(y)}$ for all $(x,y)\in \interval{0,1}^{(2)}$.
\end{lemma}
\begin{proof}
For any $\varphi\in\Tcal$, given $\Set*{U_i}_{i\in\Natural}$, let us define $V_i \coloneqq \varphi(U_i)$ for all $i\in\Natural$. Since $\varphi_\sharp\lambda_{\interval{0,1}} = \lambda_{\interval{0,1}}$, $\Set*{V_i}_{i\in\Natural}$ is a set of i.i.d. uniform random variables in $\interval{0,1}$. Using this, observe that for any $(i,j)\in\squarebrack{k}^{(2)}$,
\begin{align}\label{eqn:v_k_definition}
    v^\varphi_k\round{X^\varphi_k}(i,j) &= \E{v^\varphi\round{U_i,U_j}\given X^\varphi_k\round{\round{U_\ell}_{\ell=1}^k}}\nonumber\\
    &= \E{v\round{\varphi(U_i),\varphi(U_j)}\given X_k\round{\varphi\round{U_\ell}_{\ell=1}^k}}\nonumber\\
    &= \E{v\round{V_i,V_j}\given X_k\round{\round{V_\ell}_{\ell=1}^k}}
    = v_k\round{X_k}(i,j),
\end{align}
holds, completing the proof.
\end{proof}

Suppose we are given some $\omega = \round{\omega_t}_{t\in[0,1]}\in\mathrm{AC}(\Graphons,\delta_2)$. From Lemma~\ref{lem:representation_graphons}, we obtain $\round{W_t}_{t\in\interval{0,1}}\in\mathrm{AC}\round{\Wcal,d_2}$ such that $\omega_t = \squarebrack{W_t}$ for all $t\in\interval{0,1}$. It follows from the Radon-Nikod\'ym property~\cite[page 30, Theorem 5]{HUFF19771} that there exists $v_t \coloneqq W'_t \in L^2(\interval{0,1}^{(2)})$, for a.e. $t\in\interval{0,1}$, such that $W_t-W_0=\int_0^t v_s \diff s$, where the integral is pointwise.

For any $t\in[0,1]$, let $v_{k, t}$ and $X_{k, t}$ be defined as in equation~\eqref{eq:general_velocity} with $W_t$ replacing the role of $W$ and $v_t$ replacing $v$. The following Proposition~\ref{prop:AC_to_CE} shows that the $\rho_{k, t}=\mathrm{Law}(X_{k, t}((U_i)_{i=1}^{k}))$ satisfies continuity equation with the velocity $v_{k,t}$ defined as
\[
    v_{k,t}(z) \coloneqq \E{\round{W'_{t}(U_i,U_j)}_{(i,j)\in\squarebrack{k}^{(2)}} \given \round{W_{t}(U_i,U_j)}_{(i,j)\in\squarebrack{k}^{(2)}} = z}, \qquad z\in  \Mcal_k.
\]

\begin{proposition}\label{prop:AC_to_CE}
    Let $k\in\Natural$, and $\rho_{k,t} = \mathrm{Law}\round{X_{k,t}(\round{U_i}_{i=1}^k)}$. Then, for a.e. $t\in\interval{0,1}$, the continuity equation $\pdiff_t\rho_{k,t} + \divergence\round{v_{k,t}\rho_{k,t}} = 0$ holds weakly with Dirichlet boundary conditions. 
    That is, for any continuously differentiable test function $f$ on $\Mcal_k$, vanishing at the boundary, 
    \[
        \pdiff_t \round{\int f(z) \diff\rho_{k,t}(z)} -\int \nabla f(z) v_{k,t}(z) \diff \rho_{k,t}(z)=0, \qquad\text{a.e.}\quad t\in[0,1].
    \]
\end{proposition}

\begin{proof}
\sloppy  By definition, $\rho_{k,t}$ is the law of the random symmetric matrix $X_{k,t}(\round{U_i}_{i=1}^k)$. Fix some $f\in C^1(\Mcal_k,\rr)$ vanishing at the boundary of $\Mcal_k$. By change of variables 

\begin{align}
    \int_{\Mcal_k} f(z) \rho_{k,t}(z) \diff z &=  \int_{\interval{0,1}^k} f\round{\round{W_t(u_i,u_j)}_{(i,j)\in\squarebrack{k}^{(2)}}} \diff \lambda_{\interval{0,1}^k}\round{u},\label{eq:test_integral}
\end{align}
where $\lambda_{\interval{0,1}^k}$ is the Lebesgue measure on $\interval{0,1}^k$, and $u = \round{u_i}_{i=1}^k$. Note that $v_{k,t}\in L^2\round{\rho_{k,t}}$. Taking time derivative on both sides of equation~\eqref{eq:test_integral} for $t$ in the set of full measure where $W_t'$ is defined,
\begin{align}
    &\;\pdiff_t \int_{\Mcal_k} f(z)  \rho_{k,t}(z) \diff z\nonumber\\
    &= \int_{\interval{0,1}^k} \pdiff_t f\round{\round{W_t(u_i,u_j)}_{(i,j)\in\squarebrack{k}^{(2)}}} \diff \lambda_{\interval{0,1}^k}\round{u}\nonumber\\
    &= \int_{\interval{0,1}^k} \inner{\round{\nabla f\circ X_{k,t}}\round{u},\round{W'_t(u_i,u_j)}_{(i,j)\in\squarebrack{k}^{(2)}}} \diff \lambda_{\interval{0,1}^k}\round{u}\nonumber\\
    &= \int_{\Mcal_k} \inner{\nabla f(z), \int_{\Set*{u \in \interval{0,1}^k \given X_{k,t} = z}} \round{W'_t\round{u_i,u_j}}_{(i,j)\in\squarebrack{k}^{(2)}} \diff \mu_z\round{u}}\diff z,
\end{align}
where $\Set*{\mu_z}_{z\in\Mcal_k}$ is the disintegration of $\lambda_{\interval{0,1}^k}$, with respect to the function $X_{k,t}$. By definition of $v_{t,k}$, the above expression is exactly equal to $\int_{\Mcal_k} \inner{\nabla f(z), v_{k,t}(z)\rho_{k,t}(z)}\diff z$. This completes the proof.
\end{proof}

Proposition~\ref{prop:AC_to_CE} shows that an absolutely continuous curve in $(\Graphons, \delta_2)$ determines a family of continuity equations. In~\cite{HOPST22}, the authors take this analogy further and show that in the presence of noise the limiting curve can be described by a certain McKean-Vlasov type equation. 

As an example, consider the continuity equation for the scalar entropy.
From equation~\eqref{eqn:phi_Ent} it follows that the velocity of gradient flow for scalar entropy function satisfies
\begin{equation}
   v([W])(x,y)= -\log\left( \frac{W(x,y)}{1- W(x,y)} \right), \qquad (x,y) \in [0,1]^{(2)}\;,
\end{equation}
if $0<W<1$ a.e. Let $\Set*{U_i}_{i=1}^{\infty}$ be i.i.d. $\mathrm{Uni}\interval{0,1}$, then using equation~\eqref{eq:general_velocity} we can compute the velocities $\round{v_k}_{k\in\Natural}$ appearing in the continuity equation as 
\begin{align*}
    v_k(z)(i,j) &= \E{v(U_i,U_j) \given X_k\round{\round{U_\ell}_{\ell=1}^k} = z}\;\\
    &=-\E{\log\left( \frac{W(U_i,U_j)}{1- W(U_i,U_j)} \right)\given X_k\round{\round{U_\ell}_{\ell=1}^k} = z}\\
    &=-\log\left( \frac{z(i,j)}{1- z(i,j)} \right), \qquad z\in\Mcal_k, \quad i,j\in[k].
\end{align*}
Unfortunately, for our other examples, the velocity fields do not have closed form expressions since they are averages over all possible embeddings of a labeled weighted graph in a graphon.
\section{Examples of Gradient Flows on Graphons}\label{sec:examples_GF}

\sloppy In this section we find some natural classes of examples of functions on graphons with Fr\'echet-like derivatives. For any graphon $\squarebrack{W}\in\Graphons$ and any $k\in\Natural$, sample $\Set*{Z_i}_{i\in\squarebrack{k}}$ i.i.d. from $\mathrm{Uni}[0,1]$. Let $G_k\squarebrack{W} = \round{W\round{Z_i,Z_j}}_{(i,j)\in\squarebrack{k}^{(2)}}$. 
Let $\rho_k\round{\squarebrack{W}}=\mathrm{Law}\round{G_k\squarebrack{W}}$ denote its law over $k\times k$ symmetric matrices. Consider the functions $F_k\colon\Graphons \to \rr\cup\Set*\infty$ defined through the function composition $F_k = H_k \circ \rho_k$, for different choices of $H_k\colon \Pcal(\Mcal_k) \to \rr\cup\Set*\infty$. Notice that the function $F_k$ is well-defined over $\Graphons$. 

\subsection{Linear functions}
Let $H_k\colon \Pcal(\Mcal_k) \to \rr\cup\Set*\infty$ be defined as a linear function such that $F_k$ takes the form
\begin{align}
    F_k\round{\squarebrack{W}} &= \inner{f_k,\rho_k\round{\squarebrack{W}}} \coloneqq \int_{\Mcal_k} f_k(z)\rho_k\round{\squarebrack{W}}(\diff z),\label{eq:def_Fk}
\end{align}
for all $\squarebrack{W}\in\Graphons$ such that $\supp\round{\rho_k\round{\squarebrack{W}}}\subseteq \effdom(f_k)$, where $f_k\colon \Mcal_k\to\rr\cup\Set*\infty$ satisfies $f_k\in L^1\round{\rho_k}$ and $\nabla f_k \in L^\infty\round{\rho_k}$. Let $f_k$ be continuously differentiable on an open set containing $\supp\round{\rho_k}$.
Suppose that $f_k$ satisfies either Assumption~\ref{asmp:taylor_1} or Assumption~\ref{asmp:taylor_2} stated below.
\begin{assumption}\label{asmp:taylor_1}
    For any $\varepsilon >0$ there is some $\delta_\varepsilon >0$ such that for all $X,X_0\in\Mcal_k$ satisfying $\norm{2}{X-X_0} \le \delta_\varepsilon$, we have
    \begin{align*}
        \left\lvert    f_k(X) - f_k(X_0) - \inner{\nabla f_k(X_0),X-X_0} \right\rvert\le \varepsilon \enorm{X-X_0},
    \end{align*}
    where $\inner{\slot{},\slot{}}\colon \Mcal_k(\rr)\times \Mcal_k(\rr) \to\rr$ and $\enorm{}$ are the standard Frobenius inner product and Frobenius norm over $k\times k$ matrices respectively.
\end{assumption}
\begin{assumption}\label{asmp:taylor_2}
    There is a constant $C_0$ such that 
    \[
        \sup_{X, X_0\in\Mcal_k} \left\lvert    f_k(X) - f_k(X_0) - \inner{\nabla f_k(X_0),X-X_0} \right\rvert\le C_0.
    \]
\end{assumption}
Then, $F_k$ admits a Fr\'echet-like derivative as shown below. By change of variables, note that
\begin{align}
    F_k(\squarebrack{W}) &= \Exp{X_k\sim\rho_k\round{\squarebrack{W}}}{f_k(X_k)}\nonumber\\
    &= \Exp{\Set*{Z_i\sim {\rm Uni}\interval{0,1}}_{i=1}^k}{f_k\round{\round{W\round{Z_i,Z_j}}_{(i,j)\in\squarebrack{k}^{(2)}}}}.\label{eq:change_of_var_Fk}
\end{align}
For two different graphons $[V], [W]\in\Graphons$, consider their representative kernels $V$ and $W$. Since kernels are identified a.e. on $[0,1]^{(2)}$, we may assume $W(x,x)=V(x,x)=0$ for a.e. $x\in[0,1]$. Now use the same sequence of $\mathrm{Uni}[0,1]$ random variables $\round{Z_i}_{i=1}^k$ to obtain a coupling of $\rho_k([V])$ and $\rho_k([W])$. This is used implicitly in the following derivation and, hence, we will skip referring to the random variables $\Set*{Z_i}_{i=1}^k$ from here on. Define $X_k \coloneqq \round{W\round{Z_i,Z_j}}_{(i,j)\in\squarebrack{k}^{(2)}}$ and $Y_k \coloneqq \round{V\round{Z_i,Z_j}}_{(i,j)\in\squarebrack{k}^{(2)}}$. As a consequence of this coupling,
\begin{align}\label{eq:normcomparison}
    &\E{\norm{2}{Y_k-X_k}^2} = \sum_{i=1,j\neq i}^k \E{(V(Z_i, Z_j) - W(Z_i, Z_j))^2}\nonumber\\
    &= k(k-1)\int_{\interval{0,1}^2} (V(x,y) - W(x,y))^2 \diff x \diff y\nonumber
    = k(k-1) \norm{2}{V-W}^2. 
\end{align}
The first equality is using the fact that the diagonal terms of $Y_k-X_k$ are all zeroes. Hence $\E{\norm{2}{Y_k-X_k}} \le k \norm{2}{V-W}$, by the Jensen's inequality.

If either Assumption~\ref{asmp:taylor_1} or Assumption~\ref{asmp:taylor_2} hold, then for any $\varepsilon>0$, 
\begin{align}
    &\abs{ f_k(Y_k) - f_k(X_k) - \inner{\nabla f_k(X_k),Y_k-X_k} }\nonumber\\
    &\le \varepsilon \enorm{Y_k - X_k} \indicator{}{ \enorm{Y_k - X_k } \le \delta_\varepsilon} + C_0 \indicator{}{ \enorm{Y_k - X_k} > \delta_\varepsilon}.
\end{align}
Taking expectations on both sides, 
\begin{align*}
    &\abs{ \E{f_k(Y_k)} - \E{f_k(X_k)} - \E{\inner{\nabla f_k(X_k),Y_k-X_k}}}\nonumber\\
    &\le \varepsilon \E{\enorm{Y_k - X_k}} + C_0 \E{\indicator{}{ \enorm{Y_k - X_k} > \delta_\varepsilon}}\nonumber\\
    &\le \varepsilon k \norm{2}{V-W} + C_0 \Prob{ \enorm{Y_k - X_k} > \delta_\varepsilon}\nonumber\\
    &\le  \varepsilon k \norm{2}{V-W} + \frac{C_0}{\delta_\varepsilon^2} k^2 \norm{2}{V-W}^2,
\end{align*}
by Markov's inequality. As $\norm{2}{V-W}$ approaches zero, it is now clear that, for any $\varepsilon'>0$, 
\begin{align}
    \limsup_{V\in\Wcal,\, \enorm{V-W}\to 0} \frac{\left\lvert \E{f_k(Y_k)} - \E{f_k(X_k)} - \E{\inner{\nabla f_k(X_k),Y_k-X_k}} \right\rvert}{\norm{2}{V-W}} \le \varepsilon'.
\end{align}
Since $\varepsilon'$ is arbitrary, the above $\limsup$ must be zero.

By the definition of Fr\'echet-like derivatives (Definition~\ref{def:frechet_like_derivative}), we want to obtain some $\phi\in L^\infty\big(\interval{0,1}^{(2)}\big)$ such that
\begin{align}
    \E{\inner{\nabla f_k(X_k),Y_k-X_k}} = \inner{\phi,V-W}. \label{eq:required_frechet_linear}
\end{align}
  \sloppy Let $U_k \coloneqq Y_k - X_k$ (also similarly measurable with respect to $\Set*{Z_i}_{i=1}^k$), and let us denote by $A(Z) \coloneqq \nabla f_k\round{X_k} = \nabla f_k\round{\round{W\round{Z_i,Z_j}}_{(i,j)\in\squarebrack{k}^{(2)}}}$. Observe that
\begin{align}
   \quad  &\E{\inner{\nabla f_k(X_k),Y_k - X_k}} = \sum_{i,j=1}^k \E{\round{A(Z)}_{i,j} \round{U_k(Z)}_{i,j}}\nonumber\\
    &= \sum_{i,j=1}^k \E{\E{\round{A(Z)}_{i,j} \round{U_k(Z)}_{i,j}\given Z_i,Z_j}}\nonumber\\
    &= \sum_{i=1,j\neq i}^k \E{\E{\round{A(Z)}_{i,j} U(Z_i,Z_j)\given Z_i,Z_j}}\nonumber\\
    &= \sum_{i=1,j\neq i}^k \E{\E{\round{A(Z)}_{i,j} \given Z_i,Z_j}U(Z_i,Z_j)}\nonumber\\
    &= \sum_{i=1,j\neq i}^k \int_{\interval{0,1}^2}\E{\round{A(Z)}_{i,j} \given (Z_i,Z_j)=(x,y)}U(x,y)\diff x\diff y\nonumber\\
    &= \int_{\interval{0,1}^2} \round{\sum_{i,j=1}^k \E{\round{A(Z)}_{i,j} \given (Z_i,Z_j)=(x,y)}}U(x,y)\diff x\diff y.\label{eq:conditioning_to_obtain_phi}
\end{align}
Notice that, including the term $i=j$ above makes no difference in the integral. Therefore, if we choose $\phi\in L^\infty\big(\interval{0,1}^{2}\big)$ to be defined as
\begin{align}
    \phi(x,y) &\coloneqq \sum_{i,j=1}^k\E{\round{\nabla f_k\round{\round{W\round{Z_i,Z_j}}_{(i,j)\in\squarebrack{k}^{(2)}}}}_{i,j}\given Z_i=x, Z_j=y},
\end{align}
for $(x,y)\in\interval{0,1}^{(2)}$, then the required equality in equation~\eqref{eq:required_frechet_linear} is satisfied. 
And the action of the Fr\'echet-like derivative $\phi$ on a kernel, say $U\in\Wcal$, is
\begin{align}
    \begin{split}
        &\E{\inner{\nabla f_k\circ W_k,G_k\squarebrack{U}}}\\
        &\coloneqq \E{\inner{\nabla f_k\round{\round{W\round{Z_i,Z_j}}_{(i,j)\in\squarebrack{k}^{(2)}}}, \round{U(Z_i,Z_j)}_{(i,j)\in[k]^{(2)}} } }.
    \end{split}
    \label{eq:frechetexample}
\end{align}

\begin{lemma}
\label{lem:convexity_fktoF_k}
    If $f_k$ is $\lambda$-semiconvex on the convex set $\Mcal_k$, then $F_k$ is generalized geodesically $k(k-1)\lambda$-semiconvex on $\Graphons$. In particular, it is also geodesically $k(k-1)\lambda$-semiconvex on $\Graphons$.
\end{lemma}
\begin{proof} Since every geodesic is also a generalized geodesic, it suffices to prove the result for a generalized geodesic. Since $f_k$ is $\lambda$-semiconvex, for some $\lambda \in \mathbb{R}$, it follows that for any $X_0, X_1 \in \Mcal_k$,
\begin{equation}\label{eq:semicxfin}
    f_k(X_t) \le (1-t) f_k(X_0) + t f_k(X_1) - \frac{\lambda}{2} t(1-t)\norm{2}{X_1-X_0}^2\, , \quad \forall\; t \in \interval{0,1},
\end{equation}
along the curve $(X_t\coloneqq(1-t)X_0 + t X_1)_{t\in[0,1]}$.

\sloppy Let $[W_0], [W_1]$ be two graphons and let $\omega = \round{[W_t]}_{t\in[0,1]}$ be a generalized geodesic between $\omega_0 = [W_0]$ and $\omega_1 = [W_1]$. It follows from Definition~\ref{def:gen_geodesics} that $\omega$ has a representation in the space of kernels given by the line segment $\round{W_t=(1-t)W_0 + t W_1}_{t\in[0,1]}$, where the kernels $W_0$ and $W_1$ are such that $\norm{2}{W_0-W_1}\geq \delta_2([W_0], [W_1])$. Now use the same set $\Set*{Z_i}_{i=1}^k$ of i.i.d. $\mathrm{Uni}\interval{0,1}$ random variables as above to get a process $\round{X_{t,k}}_{k\in\Natural}$ of random matrices with distributions $\round{\rho_k([W_t])}_{k\in\Natural}$ respectively, for each $t\in\interval{0,1}$. Note that $X_{t,k} = (1-t)X_{0,k} + t X_{1,k}$, $t \in \interval{0,1}$, $k\in\Natural$. Hence applying equation~\eqref{eq:semicxfin} to this line segment and then taking expectations with respect to the joint law of $(Z_i)_{i=1}^k$, we get 
\[
    F_k([W_t]) \le (1-t) F_k([W_0]) + t F_k([W_1]) - \frac{\lambda}{2} t(1-t) \E{\norm{2}{X_{1,k}-X_{0,k}}^2},\quad t\in[0,1].
\]
Now by equation~\eqref{eq:normcomparison},
\[
    \E{\norm{2}{X_{1,k}-X_{0,k}}^2}=k(k-1) \norm{2}{W_1-W_0}^2\geq k(k-1) \delta^2_2([W_1], [W_0]).
\]
Putting it back together we get that for $t\in[0,1]$,
\begin{equation}\label{eq:semiconvex2}
    F_k(\omega_t) \le (1-t) F_k(\omega_0) + t F_k(\omega_1) - \frac{k(k-1) \lambda}{2} t(1-t) \delta^2_2(\omega_1, \omega_0).
\end{equation}
Therefore $F_k$ is $k(k-1)\lambda$-semiconvex along the generalized geodesic $\omega$. 
\end{proof}
In the following subsections, we examine special cases of linear functions.

\subsubsection{Scalar Entropy function}\label{sec:scalar_entropy}
Recall the scalar entropy function $\Ent\colon\Graphons\to\rr\cup\Set*\infty$, defined in equation~\eqref{eq:scalarentropy} as
\[
    \Ecal([W]) = \int_{[0,1]^2} h(W(x, y))\diff x\diff y,\qquad [W]\in\Graphons,
\]
where $h\colon\rr \rightarrow \rr\cup\Set*\infty$ is the convex entropy function $h(p)\coloneqq p\log p + (1-p) \log (1-p)$, if $p\in(0,1)$, $h(0)=h(1)=0$, and $h(p)=\infty$, otherwise, as defined in Section~\ref{sec:intro_example}. Observe that it can also be thought of as a linear function with $k=2$. If
\begin{align}
    f_2\round{\round{x_{(i,j)}}_{(i,j)\in\squarebrack{2}^{(2)}}} &\coloneqq h(x_{(1,2)}),\qquad x\in\Mcal_2,
\end{align}
then from equations~\eqref{eq:def_Fk} and~\eqref{eq:change_of_var_Fk},
\begin{align}
    F_2\round{\squarebrack{W}} = \inner{f_2,\rho_2\round{\squarebrack{W}}} &= \Exp{\Set*{Z_i\sim\mathrm{Uni}\interval{0,1}}_{i=1}^2}{f_2\round{\round{W(Z_i,Z_j)}_{(i,j)\in[2]^{(2)}}}}\nonumber\\
    &= \Exp{\Set*{Z_i\sim\mathrm{Uni}\interval{0,1}}_{i=1}^2}{h(W(Z_1,Z_2))} = \Ent\round{\squarebrack{W}},
\end{align}
for $\squarebrack{W}\in\Graphons$.
Since, $h(p)/p$ is bounded when $\epsilon \leq p \leq 1-\epsilon$ for some $\epsilon\in(0,1/2)$, it follows that $\Ent$ restricted to $\Graphons_{\epsilon}\coloneqq \setinline{ [W]\in \Graphons\given \epsilon\leq W\leq  1-\epsilon \text{ a.e.} }$, is continuous with respect to the weak-$*$ topology on $L^2([0, 1]^2)$. Since this is a weaker topology than the one generated by $\delta_{\cut}$, by~\cite[Lemma 8.22]{lovasz2012large}, it follows that $\Ent$ restricted to $\Graphons_{\epsilon}$ is $\delta_{\cut}$-continuous.

Since $h''(p)=1/(p(1-p))\ge 4$ for $p\in\effdom(h)$, it follows that $h$ is $4$-semiconvex on $\mathbb{R}$. Let $E\colon\Wcal\to\rr\cup\Set*\infty$ be the invariant extension of $\Ent$ such that $E(W)\coloneqq \Ent\round{\squarebrack{W}}$ for all $W\in\Wcal$. Then for any $W_0,W_1\in\Wcal$, defining $W_t \coloneqq (1-t)W_0 + tW_1$ for $t\in\interval{0,1}$, we have
\begin{align}
    E\round{W_t} &= \int_0^1\int_0^1 h(W_t(x,y))\diff x\diff y\nonumber\\
    &\leq \int_0^1\int_0^1\squarebrack{ (1-t)h(W_0(x, y)) + th(W_1(x, y)) }\diff x\diff y\nonumber\\
    & \qquad - \int_0^1\int_0^1 \frac{4}{2}t(t-1)\round{W_0-W_1}^2(x,y) \diff x\diff y\nonumber\\
    &= (1-t)E(W_0) + tE(W_1) - \inv{2}4t(1-t)\enorm{W_0-W_1}^2,
\end{align}
which implies that $E$ is $4$-semiconvex on $\round{\Wcal,d_2}$ (see Definition~\ref{def:semiconvexity}). Following Remark~\ref{rem:gen_geo_cvx_from_W}, $\Ent$ is $4$-semiconvex along generalized geodesics on $(\Graphons,\delta_2)$.

Suppose $W$ is valued in $\round{0,1}$ a.e.. Then,
\begin{align}
    \E{\inner{\nabla f_2\circ W_2,G_2\squarebrack{U}}} &= \int_0^1\int_0^1 \log\round{\frac{W(z_1,z_2)}{1-W(z_1,z_2)}} U(z_1,z_2) \diff z_1\diff z_2 ,
\end{align}
for all $U\in\Wcal$. By the characterization of the Fr\'echet-like derivative in equation~\eqref{eq:frechetexample}, we get that $\phi_\Ent = D_\Wcal f_2\round{W}$ is given by
\begin{align}
    \phi_{\Ent}(x,y) &= \log\round{\frac{W(x,y)}{1-W(x,y)}},\qquad \text{a.e.}\quad (x,y)\in\interval{0,1}^{(2)}.\label{eqn:phi_Ent}
\end{align}
If $\squarebrack{W}\in\Graphons_\epsilon$ for some $\epsilon\in (0, 1/2)$, then by Lemma~\ref{lem:upper_gradient_frechet} the local slope of entropy $\abs{\pdiff \Ent}\round{\squarebrack{W}}$ is given by $\abs{\pdiff \Ent}\round{\squarebrack{W}}=\enorm{\phi_{\Ent}}$.

Note that $\Graphons_{\epsilon}$ is closed in $\Graphons$ and since $W\mapsto \enorm{W}$ is $\delta_{\cut}$-lower semicontinuous~\cite[Lemma 14.15]{lovasz2012large}, it follows that the local slope of entropy, $\abs{\pdiff \Ent}$, is also $\delta_{\cut}$-lower semicontinuous on $\Graphons_{\epsilon}$. Following Remark~\ref{rem:velocityofgradientflow}, we conclude that starting from $[W_0]\in\Graphons_\epsilon$ the gradient flow of $\Ent$ flows with the velocity $-D_{\Graphons}\Ent(\squarebrack{W_0})$ at $[W_0]$ if $\epsilon\leq W_0\leq 1-\epsilon$ a.e. Consider the flow $W_t(x, y)$ obtained by solving 
\begin{align}
    W_t'(x, y)=-\log\round{\frac{W_t(x,y)}{1-W_t(x,y)}},\qquad \text{a.e.}\quad (x,y)\in\interval{0,1}^{(2)},\label{eq:gradient_scalar_entropy}
\end{align}
with initial condition $[W_0] \in \Graphons_\epsilon$. Then it follows that $\round{\omega_t\coloneqq \squarebrack{W_t}}_{t\in\rr_+}$ is a gradient flow of $\Ent$ starting from $\omega_0 = [W_0]$. 

\begin{takeaway}
    It is worth emphasizing that, following equation~\eqref{eq:gradient_scalar_entropy}, for any $\epsilon\in(0,1/2)$, the stationary point for the gradient flow of $\Ent$ is the constant half graphon, which is also the minimizer of $\Ecal$ on $\Graphons$. We also observe that the curve $(\omega_t)_{t\in\rr_+}$ always stays inside $\Graphons_\epsilon$ if $\omega_0\in\Graphons_\epsilon$. Since $\Ecal$ is $4$-semiconvex, it follows from Remark~\ref{rem:exp_convergence} that we are guaranteed an exponential rate of convergence of the gradient flow to the minimizer.
\end{takeaway}

\subsubsection{Homomorphism functions}\label{sec:homomorphism_example}
For $k\in\Natural\setminus\Set*{1}$ we can consider interactions such as the homomorphism functions that are continuous in the cut-metric:
\begin{align}
    T_H(\squarebrack{W}) &\coloneqq \E{\prod_{\Set*{i,j}\in E(H)}W(Z_i,Z_j)} = \int_{\Mcal_k} f_k(z)\rho_k(\diff z),
\end{align}
where $H$ is a simple graph with $\abs{V(H)} = k$, $E(H) = \Set*{e_i}_{i=1}^m$, $\Set*{Z_i}_{i\in\squarebrack{k}}$ are i.i.d. uniformly in $\interval{0,1}$, and $f_k\big(\round{x_{(i,j)}}_{(i,j)\in\squarebrack{k}^{(2)}}\big) = \prod_{\Set*{i,j}\in E(H)} x_{(i,j)}$. In particular, the function $f_k$ is a monomial for every simple graph $H$. Let $t_H\colon\Wcal\to\rr$ be the invariant extension of $T_H$ such that $t_H \coloneqq T_H\circ\squarebrack{\slot{}}$. Since
\begin{align}
    \round{\nabla f_k(X)}_{(p,q)} &= \indicator{E(H)}{p,q}\cdot\prod_{\Set*{i,j}\in E(H)\setminus\Set*{\Set*{p,q}}} X_{(i,j)},
\end{align}
for $p,q\in\squarebrack{k}$, the action of the Fr\'echet-like derivative $\phi_{T_H} = D_\Wcal t_H\round{W}$ on $U\in\Wcal$ according to equation~\eqref{eq:frechetexample} is given by
\begin{align}
    \sum_{\ell=1}^m\E{ \prod_{r=1}^{\ell-1} W\round{Z_{e_r}} \cdot U\round{Z_{e_\ell}} \cdot \prod_{r=\ell+1}^{m} W\round{Z_{e_r}}},
\end{align}
where $Z_e \coloneqq \round{Z_{e(1)},Z_{e(2)}}$ for all $e\in E(F)$. Following a similar approach to equation~\eqref{eq:conditioning_to_obtain_phi}, we obtain that $\phi_{T_H}$ satisfies
\begin{align}
\label{eqn:phiH_F}
    \begin{split}
        \phi_{T_H}(x,y) &= \sum_{\ell=1}^m \E{\prod_{r=1, r\neq \ell}^m W\round{Z_{e_r}} \given Z_{e_\ell} = (x,y) }\\
        &=\sum_{\ell=1}^m {\bf t}_{x, y}(H_{e_{\ell}}, W),\quad x,y\in\interval{0,1},
    \end{split}
\end{align}
where $H_{e}$ is the graph obtained by removing edge $e$ from $H$ and ${\bf t}_{x, y}(H_{e}, W)$ is the \emph{homomorphism density of a partially labeled graph}~\cite[Section 7.4]{lovasz2012large} defined
\[
    {\bf t}_{x, y}(H_{e}, W)\coloneqq \E{\prod_{f\in E(H_{e})} W\round{Z_{f}} \given Z_{e} = (x,y) }.
\]
Note that for fixed $W$ and $H_e$, we can think of $(x, y)\mapsto {\bf t}_{x, y}(H_e, W)$ as a bounded measurable function on $[0, 1]^2$. We now show that the kernel valued map $W\mapsto {\bf t}_{(\cdot, \cdot)}(H_{e}, W)$ is continuous with respect to $\norm{\cut}{}$. To this end, let $W, W'\in \Wcal$ and consider any product function $(x,y)\mapsto f(x)g(y)$ where $\norm{\infty}{f}, \norm{\infty}{g}\leq 1$. Note that
\begin{align*}
    {\bf t}_{x, y}(H_{e}, W)=\int_{[0, 1]^{k-2}} \prod_{\Set*{i,j}\in E(H_{e})} W(x_i, x_j) \prod_{\ell\in V(H_{e})\setminus e} \diff x_\ell.
\end{align*}
For each edge $\Set*{m,n}\in E(H_e)$, consider the integral 
\begin{align*}
   I_{m,n} &\coloneqq \int_{[0, 1]^{k}} (W(x_m, x_n)-W'(x_m, x_n))\prod_{\Set*{i,j}\in E(H_{e})\setminus\{m,n\}} W(x_i, x_j)\cdot\\
   &\qquad\qquad\qquad \prod_{\ell\in V(H_{e})\setminus e} \diff x_\ell \; f(x)g(y) \diff x\diff y.
\end{align*}
It follows from~\cite[Lemma 8.10]{lovasz2012large} (or see the second last display in the proof of~\cite[Lemma 10.24]{lovasz2012large}) that $\abs{I_{m,n}}\le \norm{\cut}{W-W'}$. From the same references above, it follows that
\begin{align*}
  &\abs{ \int_{[0, 1]^2} ({\bf t}_{x, y}(H_{e}, W)-{\bf t}_{x, y}(H_{e}, W'))f(x)g(y) \diff x\diff y}\\
  &\leq \sum_{\Set*{m,n}\in E(H_{e})} \abs{I_{m,n}} \leq \abs{E(H_e)}\norm{\cut}{W-W'}.
\end{align*}

Taking supremum over Borel measurable functions $f, g$ such that $\norm{\infty}{f}, \norm{\infty}{g}\leq 1$, we get that $W\mapsto {\bf t}_{(\cdot, \cdot)}(H_{e}, W)$ is Lipschitz continuous with respect to $\norm{\cut}{}$. Since $\abs{{\bf t}_{x, y}(H_{e}, W)}\leq 1$ for $W\in \Wcal$, it follows that $\norm{\infty}{\phi_{T_H}(W)}\le \abs{E(H)}$, that is, $\phi_{T_H}$ is uniformly bounded on $\Wcal$. 

For example, when $H$ is a triangle, the velocity $\phi_{T_H}$ follows from equation~\eqref{eqn:phiH_F} as $\phi_{T_H}(x, y)=3\int_0^1 W(x, z)W(z, y)\diff z$, for a.e. $x,y\in\interval{0,1}$, i.e., thrice the `operator product' of $W$ with itself~\cite[Section 7.4]{lovasz2012large}. If $H$ is a path on $3$ vertices and $2$ edges, then $\phi_{T_H}(x, y)=\int_0^1 W(x, z)\diff z + \int_0^1 W(y, z)\diff z = \deg(x)+\deg(y)$, where $\deg(x)\coloneqq\int_0^1 W(x, z)\diff z$ for a.e. $x,y\in\interval{0,1}$.

Obtaining the expression for the Fr\'echet-like derivative in equation~\eqref{eqn:phiH_F}, given a gradient flow $\omega$ of $T_H$, we can compute the velocity of the gradient flow as $D_{\Graphons}T_H(\omega_t)\indicator{G_{\omega_t}}{}$ for $t>0$.

The Hessian of the function $f_k$ can be easily computed as
\[
    \partial_{x_{(i,j)}x_{(p,q)}} f_k=\prod_{\Set*{a,b}\in E(H) \setminus\Set*{\Set*{i,j}, \Set*{p,q}}} x_{(a,b)},
\]
if $\{i,j\} \neq \{p,q\}$ are both edges in $E(H)$, and zero otherwise.

Since every $x_{(i,j)}\in\interval{-1,1}$ for all $(i,j)\in\squarebrack{k}^{(2)}$, every entry of the Hessian is uniformly bounded in $\interval{-1,1}$ and hence $\opnorm{\textrm{Hess}(f_k)}\le k(k-1)/2$. Therefore, $f_k$ is $(-k(k-1)/2)$-semiconvex w.r.t. $d_2$. It follows from Lemma~\ref{lem:convexity_fktoF_k} that homomorphism function $T_H$ is $(-k^2(k-1)^2/2)$-semiconvex w.r.t. $\delta_2$. In fact, since $\textrm{Hess}(f_k)(\{i, j\}, \{p, q\})$ is non-zero only when $\{i, j\}, \{p, q\}\in E(H)$, it follows that $\opnorm{\textrm{Hess}(f_k)}\le \sqrt{m(m-1)}\le m$. This would yield that $T_H$ is $(-mk(k-1))$-semiconvex w.r.t. $\delta_2$. This is useful when $H$ is sparse.

\begin{takeaway}
    Note that in this example, it is not clear if the minimizer is a constant graphon or not. If $H$ has odd number of edges however constant graphon $W\equiv -1$ is trivially a minimizer. In the case of graphs $H$ with even number of edges, explicitly determining the minimizer is trickier. As we discussed above, it is also not clear if the homomorphism density function is convex, therefore we cannot guarantee an exponential rate of convergence of the gradient flow to a minimizer following Remark~\ref{rem:exp_convergence}. To alleviate this, one can regularize the objective function by adding the scalar entropy function. We discuss this in the next example in Section~\ref{sec:linear_comb_entropy_hom}. However, in the simulation discussed in Section~\ref{sec:mantel}, it does appear that the gradient flow converges to the minimizer of the function of interest at an exponential rate.
\end{takeaway}

\subsubsection{Linear combination of Scalar Entropy and Homomorphism function}\label{sec:linear_comb_entropy_hom}
Let $\beta\in \rr$ and let $H$ be a finite simple graph with $k\in\Natural$ vertices and $m\in\Natural$ edges. Define the function $G=G_{\beta, H}\coloneqq \Ent+\beta T_H$ on the set $\Graphons_{\epsilon}$ for $\epsilon\in(0,1/2)$. The function $G$ is of particular interest in the theory of exponential graph models (see Section~\ref{sec:intro}). The function $G$ is $\delta_{\cut}$-continuous on $\Graphons_{\epsilon}$ and since $\Ent$ is $4$-semiconvex and $T_H$ is $\lambda$-semiconvex w.r.t. $\delta_2$, it follows that $G$ is also $(4+\beta\lambda)$-semiconvex w.r.t. $\delta_2$ for some $\lambda\in \rr$ that we estimate in Section~\ref{sec:homomorphism_example}. The gradient flow of $G$, therefore, exists and the Fr\'echet-like derivative $\phi_{G}=D_\Wcal G\round{W}=\phi_{\Ent}+\beta\phi_{T_H}$. Since $\Ent$ is $4$-semiconvex and $T_H$ is $(-k^2(k-1)^2/2)$-semiconvex w.r.t. $\delta_2$, it follows that $G_{\beta, H}$ is $(4-\beta k^2(k-1)^2/2)$-semiconvex w.r.t. $\delta_2$. Therefore, $G_{\beta, H}$ is at least $0$-semiconvex (i.e., convex) w.r.t. $\delta_2$ when $\beta\le 8/k^2(k-1)^2$. 

\begin{takeaway}
    Note that $\beta< 8/k^2(k-1)^2$ guarantees exponential rates of convergence of the gradient flow curve. See Remark~\ref{rem:exp_convergence}. 
\end{takeaway}

\subsection{Interaction energy}

\sloppy In the optimal transport literature, linear functionals of probability measures are often called \textit{potential energy}~\cite[page 249]{santambrogio2015optimal}. Inspired by similar definitions, one can define \textit{interaction energy}. Let $f_k\colon \Mcal_k \times \Mcal_k \rightarrow \mathbb{R}$ be a function defined on pairs of $k\times k$ symmetric matrices with entries in $[-1,1]$. Given a graphon $[W] \in \Graphons$, as before let $\rho_k=\mathrm{Law}(G_k[W])$. This defines a function $\mathbb{F}_k\colon \Graphons \rightarrow \mathbb{R}\cup\Set*\infty$ given by 
\[
    \mathbb{F}_k([W])\coloneqq \int_{\Mcal_k} \int_{\Mcal_k} f_k(z,z')\rho_k(\diff z) \rho_k(\diff z').
\]

Although it looks different than before, this is also a particular case of equation~\eqref{eq:frechetexample} as shown below. Define two independent sequences of i.i.d. $\mathrm{Uni}[0,1]$ random variables: $(Z_1, \ldots, Z_k)$ and $(Z'_1, \ldots, Z'_k)$, and consider the corresponding matrices 
\[
    X_k\coloneqq (W(Z_i, Z_j))_{(i,j)\in [k]^{(2)}}, \quad \text{and}\quad X'_k\coloneqq(W(Z'_i, Z'_j))_{(i,j)\in [k]^{(2)}}.
\]
Then $X_k$ and $X_k'$ are i.i.d. samples from $\rho_k$ and $\mathbb{F}_k([W])= \E{f_k(X_k, X'_k)}$. On the other hand, one can concatenate $\round{Z_i}_{i=1}^k$ and $\round{Z'_i}_{i=1}^k$ to construct a single vector $(Z_1, \ldots, Z_k, Z_1', \ldots, Z_k')$ of dimension $2k$ and consider the corresponding $2k\times 2k$ symmetric matrix $X_{2k}$ whose block diagonal components are $X_k$ and $X_k'$. By defining $\bar{f}_k(X_{2k})\coloneqq f_k(X_k, X_k')$, we represent $\mathbb{F}_k([W])= \E{\bar{f}_{2k}(X_{2k})}= \int \bar{f}_{2k}(w) \rho_{2k}(\diff w)$ and equation~\eqref{eq:frechetexample} continues to hold.   

An example of interaction energy is given by ``variance of homomorphism functions''. As before, let $H$ be a simple graph and $W_k, W'_k$ be i.i.d. sampled $(k\times k)$ symmetric matrices from a graphon $[W]$. Consider the function 
\begin{equation}\label{eq:interacexample}
    \mathbb{F}_k([W])\coloneqq \inv{2}\E{ \round{ \prod_{e \in E(H)} W(Z_e) - \prod_{e \in E(H)} W(Z'_e) }^2}=\Var{ \prod_{e \in E(H)} W(Z_e)},
\end{equation}
by symmetry, where $Z_e \coloneqq \round{Z_{e(1)},Z_{e(2)}}$ and $Z'_e \coloneqq (Z'_{e(1)},Z'_{e(2)})$ for all $e\in E(H)$. In fact, the above identity holds for whenever we take $f_k(z,z')= \round{g_k(z) - g_k(z')}^2$, for any function $g_k$ on $\Mcal_k$ that is square-integrable w.r.t. $\rho_k$. Unfortunately this particular function $\mathbb{F}_k$ in~\eqref{eq:interacexample} is continuous in $\delta_2$ but not in $\delta_\cut$. See~\cite[Theorem 10.15]{janson2010graphons}. Hence, although the curve of maximal slope exists due to the existence of Fr\'echet-like derivatives, this particularly natural example does not satisfy the assumptions of our convergence theorem.

A similar but slightly different example of interaction which does satisfy our conditions can be constructed as follows. For a simple graph $H$ with $k$ vertices, consider its simple subgraphs $H_1$ and $H_2$ with $k_1$ and $k_2$ vertices respectively, such that every vertex and edge in $H$ is contained in at least one of the subgraphs. Pool all the uniform random variables $\Set*{Z_i}_{i=1}^{k_1}\cup \Set*{Z'_i}_{i=1}^{k_2}$ to get a single set of $k_1+k_2\geq k$ i.i.d. $\mathrm{Uni}\interval{0,1}$ random variables $\Set*{U_i}_{i=1}^{k_1+k_2}$ such that $\Set*{U_i}_{i=1}^{k_1}=\Set*{Z_i}_{i=1}^{k_1}$ and $\Set*{U_i}_{i=k_1+1}^{k} = \Set*{Z_j'}_{j\in V(H)\setminus V(H_1)}$. 
We can then define $I_{H_1,H_2}(\slot{};H)\colon\Graphons_\epsilon \to \rr\cup\Set*\infty$ as
\begin{align}
    \begin{split}
    I_{H_1,H_2}([W];H) &\coloneqq \log\round{\E{\prod_{\{i,j\} \in E(H_1)} W(Z_i, Z_j)}}\\
    & \qquad\qquad + \log\round{\E{\prod_{\{i,j\} \in E(H_2)} W(Z'_i, Z'_j)}} \\
    & \qquad \qquad \qquad \quad - 2\log\round{\E{ \prod_{\{i,j\} \in E(H)} W(U_i, U_j)}},
    \end{split}
    \label{eq:I_F1F2}
\end{align}
for some $\epsilon\in (0,1)$. Each of the terms in the expression in equation~\eqref{eq:I_F1F2} is the logarithm of the homomorphism densities of a simple graph. Logarithms of homomorphisms are well studied in graph theory and, in particular, related to the max-cut problem (see~\cite[Remark 5.4, Example 5.18]{lovasz2012large}).

We can construct a graph $H_1H_2$ by forming the disjoint union of $H_1$ and $H_2$, identifying the vertices of the same label, adding all the edges between vertices with the same label according to~\cite[Section 4.2]{lovasz2012large}. $H_1H_2$ can have multiple edges. Then, using~\cite[Proposition 7.1]{lovasz2012large} for the homomorphism density as a simple graph parameter, we get that the determinant of the connection matrix of the homomorphism density 
\begin{equation}\label{eq:determinant}
\begin{split}
    T_{H_1H_1}(\squarebrack{W})T_{H_2H_2}(\squarebrack{W}) - T_{H_1H_2}^2(\squarebrack{W}) &\ge 0 ,\\
    \text{i.e.,}\quad \frac{T_{H_1H_1}(\squarebrack{W})T_{H_2H_2}(\squarebrack{W})}{T_{H_1H_2}^2(\squarebrack{W})} &\ge 1. 
\end{split}
\end{equation} 
By the assumption that each vertex and edge of $H$ is contained in at least one of the subgraphs, if we identify the multiple edges between the same pair of vertices in $H_1H_2$ we get back $H$.
Since $T$ is a simple graph parameter, we have $T_{H_1H_2} = T_H$, $T_{H_1H_1} = T_{H_1}$ and $T_{H_2H_2} = T_{H_2}$, by definition.
Thus, taking logarithms in the final expression of~\eqref{eq:determinant}, we get $I_{H_1, H_2}(\slot{};H)\ge 0$. It is exactly zero if $H_1$ and $H_2$ are vertex disjoint. Thus, one may think that $I_{H_1, H_2}(\slot{};H)$ measures the dependence of the two subgraphs on the homomorphism density. 

We can similarly construct higher order interactions by considering multiple subgraphs instead of just two. In that case, one may consider the logarithm of the determinant of the connection matrix of the homomorphism density and define $I$ suitably.  

The argument in Section~\ref{sec:homomorphism_example} shows that this function satisfy all our conditions. The computation for its Fr\'echet-like derivative also follows from Section~\ref{sec:homomorphism_example} followed by the application of the chain rule for derivatives. Logarithms of determinants of matrices play an important role in displacement convexity of Wasserstein optimal transport (see~\cite[proof of Theorem 2.2]{M97}). It is an interesting coincidence that they also appear in this context.

\subsection{Internal energy}\label{sec:internal_energy}
Similar to potential and interaction energies, one can define non-linear functions of $\rho_k$ corresponding to what are called `internal energies' in the optimal transport literature. Let $u$ be a real-valued function on $\rr_+$ such that $u(0)=0$. For a probability measure $\rho$ on an Euclidean space, define the function
\[
U(\rho) \coloneqq \begin{cases}
    \int_{\Mcal_k} u(\rho(z))\diff z, & \text{if $\rho$ is absolutely continuous}, \\
    \infty, & \text{otherwise}.
\end{cases}
\]
Here we have used the standard abuse of notation in optimal transport literature of denoting an absolutely continuous measure and its density by the same notation. This defines a nonlinear function on graphons in the following manner. For $1\le l \le k$, consider a function $G_{k,l}\colon \Mcal_k\rightarrow \interval{-1,1}^l$ that selects a particular subset of length $l$ from the upper-diagonal elements of a $k\times k$ matrix. Consider the pushforward $\rho_{k,l}([W])\coloneqq \left(G_{k,l}\right)_{\sharp}\round{\rho_k([W])}$. Since $\rho_k$ is generated by $k$ i.i.d. $\mathrm{Uni}[0,1]$ random variables, and $l \le k$, it is easy to come up with examples where $\rho_{k,l}$ has a density. For example, take $W(x,y)=\sin(x+y)$, $k=3, l=2$, and $G_{k,l}(A)=(A_{1,2}, A_{1,3})$. Thus $\round{G_{k,l} \circ G_k}([W])=\left( \sin(Z_1+Z_2), \sin(Z_1+Z_3) \right)$. This random vector has a density.  
Thus, the function $U\left( \rho_{k,l}([W])\right)$ for $\squarebrack{W}\in\Graphons$ has a non-empty domain. A prominent example of such functions is the differential entropy for which $u\colon x\mapsto x\log x$ where one computes the differential entropy of $\rho_{k,l}([W])$. 

Such functions cannot have Fr\'echet-like derivatives since a necessary condition for its existence is that the function must be continuous in $L^2$. On the other hand, one cannot expect a discrete to continuum convergence as considered here. Even in the case of Wasserstein gradient flows, gradient flow of entropy is obtained by adding Brownian noise to particles and not by taking limits of Euclidean gradient flows~\cite{JKO98}.



\bibliographystyle{alpha} 
\bibliography{references}

\newcommand{\etalchar}[1]{$^{#1}$}
\begin{thebibliography}{BEFLY21}

\bibitem[AdHR21]{AHR21}
Siva Athreya, Frank den Hollander, and Adrian R\"{o}llin.
\newblock {Graphon-valued stochastic processes from population genetics}.
\newblock {\em The Annals of Applied Probability}, 31(4):1724 -- 1745, 2021.

\bibitem[AGS08]{ambrosio2005gradient}
Luigi Ambrosio, Nicola Gigli, and Giuseppe Savar\'e.
\newblock {\em Gradient Flows: In Metric Spaces and in the Space of Probability
  Measures. Second Edition.}
\newblock Lectures in Mathematics. ETH Z{\"u}rich. Birkh{\"a}user Verlag AG,
  Basel, 2008.

\bibitem[Ald81]{aldous1981representations}
David~J. Aldous.
\newblock Representations for partially exchangeable arrays of random
  variables.
\newblock {\em Journal of Multivariate Analysis}, 11(4):581--598, 1981.

\bibitem[Ald82]{aldous1982exchangeability}
David~J Aldous.
\newblock On exchangeability and conditional independence.
\newblock {\em Exchangeability in probability and statistics (Rome, 1981)},
  pages 165--170, 1982.

\bibitem[AOY19]{araujo2019mean}
Dyego Ara{\'u}jo, Roberto~I Oliveira, and Daniel Yukimura.
\newblock A mean-field limit for certain deep neural networks.
\newblock arXiv preprint arXiv:1906.00193, 2019.

\bibitem[Aus08]{austin2008exchangeable}
Tim Austin.
\newblock On exchangeable random variables and the statistics of large graphs
  and hypergraphs.
\newblock {\em Probability Surveys}, 5:80--145, 2008.

\bibitem[Aus12]{austin2012exchangeable}
Tim Austin.
\newblock Exchangeable random arrays.
\newblock In {\em Notes for IAS workshop}, 2012.

\bibitem[Aus15]{AustinExchMeasure}
Tim Austin.
\newblock {Exchangeable random measures}.
\newblock {\em Annales de l'Institut Henri Poincaré, Probabilités et
  Statistiques}, 51(3):842 -- 861, 2015.

\bibitem[BC21]{bach2021gradient}
Francis {Bach} and Lena{\"\i}c {Chizat}.
\newblock Gradient descent on infinitely wide neural networks: Global
  convergence and generalization.
\newblock arXiv preprint arXiv:2110.08084, 2021.

\bibitem[BCCH18]{borgs2016sparse}
Christian Borgs, Jennifer~T. Chayes, Henry Cohn, and Nina Holden.
\newblock Sparse exchangeable graphs and their limits via graphon processes.
\newblock {\em Journal of Machine Learning Research}, 18(210):1--71, 2018.

\bibitem[BCL{\etalchar{+}}08]{borgs2008convergent}
Christian Borgs, Jennifer~T Chayes, L{\'a}szl{\'o} Lov{\'a}sz, Vera~T S{\'o}s,
  and Katalin Vesztergombi.
\newblock Convergent sequences of dense graphs {I}: Subgraph frequencies,
  metric properties and testing.
\newblock {\em Advances in Mathematics}, 219(6):1801--1851, 2008.

\bibitem[BCL{\etalchar{+}}12]{borgs2012convergent}
Christian Borgs, Jennifer~T Chayes, L{\'a}szl{\'o} Lov{\'a}sz, Vera~T S{\'o}s,
  and Katalin Vesztergombi.
\newblock Convergent sequences of dense graphs {II}. multiway cuts and
  statistical physics.
\newblock {\em Annals of Mathematics}, pages 151--219, 2012.

\bibitem[BEFLY21]{ben2021ordered}
Omri Ben-Eliezer, Eldar Fischer, Amit Levi, and Yuichi Yoshida.
\newblock {Ordered Graph Limits and Their Applications}.
\newblock In James~R. Lee, editor, {\em 12th Innovations in Theoretical
  Computer Science Conference (ITCS 2021)}, volume 185 of {\em Leibniz
  International Proceedings in Informatics (LIPIcs)}, pages 42:1--42:20,
  Dagstuhl, Germany, 2021. Schloss Dagstuhl--Leibniz-Zentrum f{\"u}r
  Informatik.

\bibitem[BG20]{bhattacharya2020upper}
Bhaswar~B. Bhattacharya and Shirshendu Ganguly.
\newblock Upper tails for edge eigenvalues of random graphs.
\newblock {\em SIAM Journal on Discrete Mathematics}, 34(2):1069--1083, 2020.

\bibitem[Bon71]{bondy1971pancyclic}
J~Adrian Bondy.
\newblock Pancyclic graphs i.
\newblock {\em Journal of Combinatorial Theory, Series B}, 11(1):80--84, 1971.

\bibitem[But16]{butcher2016numerical}
John~Charles Butcher.
\newblock {\em Numerical methods for ordinary differential equations}.
\newblock John Wiley \& Sons, Hoboken, NJ, 2016.

\bibitem[CB18]{chizat2018global}
L\'{e}na\"{\i}c Chizat and Francis Bach.
\newblock On the global convergence of gradient descent for over-parameterized
  models using optimal transport.
\newblock In {\em Proceedings of the 32nd International Conference on Neural
  Information Processing Systems}, page 3040–3050, Red Hook, NY, USA, 2018.
  Curran Associates Inc.

\bibitem[CCP19]{carrillo2019blob}
Jos{\'e}~Antonio Carrillo, Katy Craig, and Francesco~S Patacchini.
\newblock A blob method for diffusion.
\newblock {\em Calculus of Variations and Partial Differential Equations},
  58(2):1--53, 2019.

\bibitem[CD13]{chatterjee2013estimating}
Sourav Chatterjee and Persi Diaconis.
\newblock {Estimating and understanding exponential random graph models}.
\newblock {\em The Annals of Statistics}, 41(5):2428 -- 2461, 2013.

\bibitem[CD20]{cook2020large}
Nicholas Cook and Amir Dembo.
\newblock Large deviations of subgraph counts for sparse
  {E}rd{\H{o}}s--{R}{\'e}nyi graphs.
\newblock {\em Advances in Mathematics}, 373:107289, 2020.

\bibitem[Cha17]{chatterjee2017large}
Sourav Chatterjee.
\newblock {\em Large deviations for random graphs: {\'E}cole d'{\'E}t{\'e} de
  {P}robabilit{\'e}s de {S}aint-{F}lour {XLV}-2015}, volume 2197.
\newblock Springer, New York, 2017.

\bibitem[Che16]{ChernThesis}
Bobbie~G. Chern.
\newblock {\em Large deviations approximation to normalizing constants in
  exponential models}.
\newblock PhD thesis, Stanford University, 2016.

\bibitem[Cra16]{CraneAAP}
Harry Crane.
\newblock {Dynamic random networks and their graph limits}.
\newblock {\em The Annals of Applied Probability}, 26(2):691 -- 721, 2016.

\bibitem[CV11]{chatterjee2011large}
Sourav Chatterjee and S.~R.~S. Varadhan.
\newblock The large deviation principle for the {E}rd{\H{o}}s-{R}{\'e}nyi
  random graph.
\newblock {\em European Journal of Combinatorics}, 32(7):1000--1017, 2011.
\newblock Homomorphisms and Limits.

\bibitem[DGKR15]{DIAO2015183}
Peter Diao, Dominique Guillot, Apoorva Khare, and Bala Rajaratnam.
\newblock Differential calculus on graphon space.
\newblock {\em Journal of Combinatorial Theory, Series A}, 133:183--227, 2015.

\bibitem[DJ08]{diaconis2007graph}
Persi Diaconis and Svante Janson.
\newblock Graph limits and exchangeable random graphs.
\newblock {\em Rendiconti di Matematica e delle sue Applicazioni},
  28(1):33--61, 2008.

\bibitem[DSS{\etalchar{+}}20]{GWsample}
Pinar Demetci, Rebecca Santorella, Bj{\"o}rn Sandstede, William~Stafford Noble,
  and Ritambhara Singh.
\newblock {Gromov-Wasserstein optimal transport to align single-cell
  multi-omics data}.
\newblock {\em bioRxiv}, 2020.

\bibitem[EG18]{eldan2018exponential}
Ronen Eldan and Renan Gross.
\newblock {Exponential random graphs behave like mixtures of stochastic block
  models}.
\newblock {\em The Annals of Applied Probability}, 28(6):3698 -- 3735, 2018.

\bibitem[FK99]{frieze1999quick}
Alan Frieze and Ravi Kannan.
\newblock Quick approximation to matrices and applications.
\newblock {\em Combinatorica}, 19(2):175--220, 1999.

\bibitem[GK20]{ghafouri2020survey}
Saeid Ghafouri and Seyed~Hossein Khasteh.
\newblock A survey on exponential random graph models: an application
  perspective.
\newblock {\em PeerJ Computer Science}, 6:e269, 2020.

\bibitem[GT19]{Gangbo19}
Wilfrid Gangbo and Adrian Tudorascu.
\newblock On differentiability in the {W}asserstein space and well-posedness
  for {H}amilton-{J}acobi equations.
\newblock {\em Journal de Math\'{e}matiques Pures et Appliqu\'{e}s},
  125:119--174, 2019.

\bibitem[Hoo82]{hoover1982row}
David~N Hoover.
\newblock Row-column exchangeability and a generalized model for probability.
\newblock {\em Exchangeability in probability and statistics (Rome, 1981)},
  pages 281--291, 1982.

\bibitem[HOP{\etalchar{+}}22]{HOPST22}
Z.~Harchaoui, S.~Oh, S.~Pal, R.~Somani, and R.~Tripathi.
\newblock {Stochastic optimization on matrices and a graphon McKean-Vlasov
  limit}.
\newblock arXiv preprint arXiv:2210.00422, 2022.

\bibitem[Huf77]{HUFF19771}
R.E. Huff.
\newblock The {R}adon-{N}ikod{\`y}m property for {B}anach-spaces — a survey
  of geometric aspects.
\newblock In Klaus-Dieter Bierstedt and Benno Fuchssteiner, editors, {\em
  Functional Analysis: Surveys and Recent Results}, volume~27 of {\em
  North-Holland Mathematics Studies}, pages 1--13. North-Holland, Germany,
  1977.

\bibitem[Hun14]{hunter2014notes}
John~K Hunter.
\newblock Notes on partial differential equations.
\newblock {\em Lecture Notes,
  \url{https://www.math.ucdavis.edu/~hunter/pdes/pde_notes.pdf}, Department of
  Mathematics, University of California}, 2014.

\bibitem[Jan13]{janson2010graphons}
Svante Janson.
\newblock Graphons, cut norm and distance, couplings and rearrangements.
\newblock {\em NYJM Monographs}, 4, 2013.

\bibitem[Jan16]{janson2016graphons}
Svante Janson.
\newblock Graphons and cut metric on sigma-finite measure spaces.
\newblock arXiv preprint arXiv:1608.01833, 2016.

\bibitem[JKO98]{JKO98}
Richard Jordan, David Kinderlehrer, and Felix Otto.
\newblock {The Variational Formulation of the Fokker--Planck Equation}.
\newblock {\em SIAM Journal on Mathematical Analysis}, 29(1):1--17, 1998.

\bibitem[Kal89]{kallenberg1989representation}
Olav Kallenberg.
\newblock On the representation theorem for exchangeable arrays.
\newblock {\em Journal of Multivariate Analysis}, 30(1):137--154, 1989.

\bibitem[KY17]{kenyon_yin_2017}
Richard Kenyon and Mei Yin.
\newblock On the asymptotics of constrained exponential random graphs.
\newblock {\em Journal of Applied Probability}, 54(1):165–180, 2017.

\bibitem[Lin94]{lindelof1894application}
Ernest Lindel{\"o}f.
\newblock Sur l'application de la m{\'e}thode des approximations successives
  aux {\'e}quations diff{\'e}rentielles ordinaires du premier ordre.
\newblock {\em Comptes rendus hebdomadaires des s{\'e}ances de l'Acad{\'e}mie
  des sciences}, 116(3):454--457, 1894.

\bibitem[Lov12]{lovasz2012large}
L{\'a}szl{\'o} Lov{\'a}sz.
\newblock {\em Large Networks and Graph Limits}, volume~60 of {\em Colloquium
  publications}.
\newblock American Mathematical Society, Providence, RI, 2012.

\bibitem[LS06]{lovasz2006limits}
L{\'a}szl{\'o} Lov{\'a}sz and Bal{\'a}zs Szegedy.
\newblock Limits of dense graph sequences.
\newblock {\em Journal of Combinatorial Theory, Series B}, 96(6):933--957,
  2006.

\bibitem[LS07]{Lovsz2007SzemerdisLF}
L{\'a}szl{\'o} Lov{\'a}sz and Bal{\'a}zs Szegedy.
\newblock Szemer{\'e}di's lemma for the analyst.
\newblock {\em Geometric And Functional Analysis}, 17:252--270, 2007.

\bibitem[LZ15]{lubetzky2015replica}
Eyal Lubetzky and Yufei Zhao.
\newblock On replica symmetry of large deviations in random graphs.
\newblock {\em Random Structures \& Algorithms}, 47(1):109–--146, August
  2015.

\bibitem[LZ17]{DiffCal}
L\'{a}szl\'{o}~Mikl\'{o}s Lov\'{a}sz and Yufei Zhao.
\newblock On derivatives of graphon parameters.
\newblock {\em Journal of Combinatorial Theory Series A}, 145(C):364–368,
  January 2017.

\bibitem[Man07]{mantel1907problem}
Willem Mantel.
\newblock Problem 28.
\newblock {\em Wiskundige Opgaven}, 10(2):60--61, 1907.

\bibitem[McC97]{M97}
Robert~J. McCann.
\newblock A convexity principle for interacting gases.
\newblock {\em Advances in Mathematics}, 128(1):153--179, 1997.

\bibitem[M{\'{e}}m11]{GWMemoli}
Facundo M{\'{e}}moli.
\newblock {G}romov–{W}asserstein distances and the metric approach to object
  matching.
\newblock {\em Found. Comput. Math}, 11:417--487, 2011.

\bibitem[MMM19]{mei2019mean}
Song Mei, Theodor Misiakiewicz, and Andrea Montanari.
\newblock Mean-field theory of two-layers neural networks: dimension-free
  bounds and kernel limit.
\newblock In Alina Beygelzimer and Daniel Hsu, editors, {\em Proceedings of the
  Thirty-Second Conference on Learning Theory}, volume~99 of {\em Proceedings
  of Machine Learning Research}, pages 2388--2464, 2019.

\bibitem[Mun00]{munkres2000topology}
James~R Munkres.
\newblock {\em Topology}.
\newblock Prentice Hall Upper Saddle River, NJ, 2000.

\bibitem[NP20]{nguyen2020rigorous}
Phan-Minh Nguyen and Huy~Tuan Pham.
\newblock A rigorous framework for the mean field limit of multilayer neural
  networks.
\newblock arXiv preprint arXiv:2001.11443, 2020.

\bibitem[RVE18]{rotskoff2018parameters}
Grant~M Rotskoff and Eric Vanden-Eijnden.
\newblock Parameters as interacting particles: long time convergence and
  asymptotic error scaling of neural networks.
\newblock In {\em Proceedings of the 32nd International Conference on Neural
  Information Processing Systems}, pages 7146--7155, 2018.

\bibitem[San15]{santambrogio2015optimal}
Filippo Santambrogio.
\newblock {\em Optimal Transport for Applied Mathematicians: Calculus of
  Variations, PDEs, and Modeling}, volume~87 of {\em Progress in Nonlinear
  Differential Equations and Their Applications}.
\newblock Springer International Publishing, Cham, 2015.

\bibitem[San17]{santambrogio2017euclidean}
Filippo Santambrogio.
\newblock $\{$Euclidean, metric, and Wasserstein$\}$ gradient flows: an
  overview.
\newblock {\em Bulletin of Mathematical Sciences}, 7(1):87--154, 2017.

\bibitem[SMN18]{song2018mean}
Mei Song, Andrea Montanari, and P~Nguyen.
\newblock A mean field view of the landscape of two-layers neural networks.
\newblock {\em Proceedings of the National Academy of Sciences},
  115:E7665--E7671, 2018.

\bibitem[SS20a]{sirignano2020clt}
Justin Sirignano and Konstantinos Spiliopoulos.
\newblock Mean field analysis of neural networks: A central limit theorem.
\newblock {\em Stochastic Processes and their Applications}, 130(3):1820--1852,
  2020.

\bibitem[SS20b]{sirignano2020lln}
Justin Sirignano and Konstantinos Spiliopoulos.
\newblock Mean field analysis of neural networks: A law of large numbers.
\newblock {\em SIAM Journal on Applied Mathematics}, 80(2):725--752, 2020.

\bibitem[Stu12]{Sturm12}
Karl-Theodor Sturm.
\newblock The space of spaces: curvature bounds and gradient flows on the space
  of metric measure spaces.
\newblock Available at ar{X}iv:1208.0434v1, 2012.

\bibitem[TR20]{tzen2020mean}
Belinda Tzen and Maxim Raginsky.
\newblock A mean-field theory of lazy training in two-layer neural nets:
  entropic regularization and controlled {M}c{K}ean-{V}lasov dynamics.
\newblock arXiv preprint arXiv:2002.01987, 2020.

\bibitem[Vil03]{V03}
C{\'e}dric Villani.
\newblock {\em Topics in Optimal Transportation}, volume~58 of {\em Graduate
  Studies in Mathematics}.
\newblock American Mathematical Society, Providence, RI, 2003.

\end{thebibliography}


\end{document}